\documentclass[12pt, leqno, british]{amsart}
\usepackage[
	style=alphabetic,
	citestyle=alphabetic,
	backend=biber,
	maxalphanames=4,
]{biblatex}
\usepackage{harmony}
\usepackage{a4, amsmath}
\usepackage{mathtools}
\usepackage{amssymb}
\usepackage{epsfig}
\usepackage{graphics}
\usepackage{amsthm, amscd, mathdots}
\swapnumbers
\usepackage{enumerate}
\usepackage{url}
\usepackage{cancel}
\usepackage[british]{babel}
\usepackage{csquotes}
\usepackage{color}
\usepackage{datetime}
\usepackage{todonotes}
\usepackage{stmaryrd}
\usepackage{ifthen}

\usepackage[pass]{geometry}

\usepackage[all]{xy}

\usepackage[hidelinks]{hyperref}
\usepackage{cleveref}
\usepackage{thmtools} 
\usepackage{xurl}

\newtheorem{defi}{Definition}[section]

\theoremstyle{plain}
\newtheorem{prop}[defi]{Proposition}
\newtheorem{lem}[defi]{Lemma}
\newtheorem{thm}[defi]{Theorem}
\newtheorem{cor}[defi]{Corollary}

\newtheorem{ques}[defi]{Question}
\newtheorem*{stel*}{Theorem}
\newtheorem*{stel1}{Theorem I}
\newtheorem*{stel2}{Theorem II}
\newtheorem*{stel3}{Theorem III}

\newtheorem*{ques*}{Question}
\theoremstyle{definition}
\newtheorem{rem}[defi]{Remark}
\newtheorem{ex}[defi]{Example}

\newcommand{\La}{\mathcal{L}}
\newcommand{\nat}{\mathbb{N}}
\newcommand{\zz}{\mathbb{Z}}
\newcommand{\rr}{\mathbb{R}}
\newcommand{\ff}{\mathbb{F}}
\newcommand{\cc}{\mathbb{C}}
\newcommand{\qq}{\mathbb{Q}}

\newcommand{\SO}{\mathsf{SO}}
\newcommand{\Or}{\mathsf{O}}

\newcommand{\car}{\mathsf{char}}
\newcommand{\mbb}{\mathbb}

\newcommand{\Val}{\mathcal{V}}

\newcommand{\Lar}{{\mathcal{L}_{\rm ring}}}

\newtheorem*{prop*}{Proposition}

\newcommand{\mg}[1]{{#1}^{\times}}
\newcommand{\la}{\langle}

\newcommand{\lla}{\la\!\la}

\newcommand{\llangle}{\langle\!\langle}

\newcommand{\mc}{\mathcal}
\newcommand{\mf}{\mathfrak}
\newcommand{\mfm}{\mf{m}}
\newcommand{\mfb}{\mf{b}}

\newcommand{\ovl}{\overline}

\newcommand{\Split}{\mbb{S}}
\newcommand{\RF}[2]{\mathrm{r}_{#1}(#2)}
\DeclareMathOperator{\Quad}{\mathsf{Quad}}
\newcommand{\hh}{\mathbb{H}}

\newcommand{\sep}[3][]{
\ifx &#1&
{{#2}_{(#3)}}
\else
{{#2}_{(#3, #1)}}
\fi
}

\DeclareMathOperator{\Spec}{Spec}

\DeclareMathOperator{\Gal}{Gal}

\DeclareMathOperator{\gen}{\mathfrak{g}}

\newcommand{\qse}[1]{{#1}^{\#}}
\newcommand{\s}{\sigma}
\let\dim\relax
\DeclareMathOperator{\dim}{\mathsf{dim}}
\renewcommand{\max}{\mathsf{max}}
\renewcommand{\min}{\mathsf{min}}
\renewcommand{\bmod}{\,\,\mathsf{mod}\,\,}
\renewcommand{\setminus}{\smallsetminus}
\renewcommand{\leq}{\leqslant}
\renewcommand{\geq}{\geqslant}
\newcommand{\pfi}[2]{\lla #1\hspace{0.4mm}]\hspace{-0.25mm}]_{#2}}
\newcommand{\qss}[2]{\mathsf{QS}_{#2}(#1)}
\DeclareMathOperator{\N}{\mathsf{N}}

\newcommand{\pow}[1]{^{(#1)}}

\title[Uniform existential definitions of valuations in function fields]{Uniform existential definitions of valuations in function fields in one variable}
\author{Karim Johannes Becher}
\author{Nicolas Daans}
\author{Philip Dittmann}
\date{\today}
\address{University of Antwerp, Department of Mathematics, Middelheimlaan 1, 2020 Antwerp, Belgium.}
\email{karimjohannes.becher@uantwerpen.be}

\address{Charles University, Faculty of Mathematics and Physics, Department of Algebra, Sokolov\-sk\' a 83, 18600 Praha~8, Czech Republic}
\curraddr{
KU Leuven, Department of Mathematics, Celestijnenlaan 200B, 3001 Heverlee, Belgium
}
\email{nicolas.daans@kuleuven.be}

\address{Institut für Algebra, Technische Universität Dresden, 01062 Dresden, Germany}
\curraddr{Department of Mathematics, University of Manchester, Manchester M13 9PL, United Kingdom}
\email{philip.dittmann@manchester.ac.uk}

\thanks{This article contains material from the second author's thesis \autocite{DaansThesis}, supervised by the first and the third author.
  The thesis work was supported by the FWO PhD Fellowship fundamental research grants 51581 and 83494. \\
  The second author further gratefully acknowledges support by {Czech Science Foundation} (GA\v CR) grant 21-00420M, and {Charles University Research Centre program} UNCE/SCI/022.
  }

\begin{document}
\begin{abstract}
We study function fields of curves over a base field $K$ which is either a global field or a large field having a separable field extension of degree divisible by $4$.
We show that, for any such function field,
 Hilbert's 10th Problem has a negative answer,
the valuation rings containing $K$ are uniformly existentially definable, and 
 the finitary holomorphy rings are definable by a universal-existential formula.
 In order to obtain these results, we develop further the usage of local-global principles for quadratic forms in function fields to definability of certain subrings.
 We include a first systematic presentation of this general method, without restriction on the characteristic. 


\medskip\noindent
{\sc Classification (MSC 2020):} 12L05, 12L99, 11E81, 12F20. 
\end{abstract}
\maketitle

\section{Introduction}

Let $F$ be a field. A subset $A\subseteq F$ is called \emph{existentially definable} (or \emph{diophantine}) if there exist $k,m\in\nat$ and polynomials $f_1, \ldots, f_k \in F[X, Y_1, \ldots, Y_m]$ such that
\begin{displaymath}
A = \lbrace x \in F \mid \exists y \in F^m : f_1(x, y) = \ldots = f_k(x, y) = 0 \rbrace.
\end{displaymath}
Unless the arithmetic of $F$ is very well-understood (e.g. when $F$ is algebraically closed, real closed or $p$-adically closed), it is hard to decide whether a given subset of $F$ is existentially definable.

Historically, the study of existentially definable sets in fields is largely inspired by Hilbert's 10th Problem for rings and fields. For an integral domain $F$ and a finitely generated subring $F_0$, \emph{Hilbert's 10th Problem for $F$ with coefficients in $F_0$} asks whether there is an algorithm which, on being presented with $r \in \nat$ and a polynomial $f \in F_0[X_1, \dotsc,X_r]$, decides whether the polynomial equation $f(X_1, \dotsc,X_r)=0$ has a solution in $F^r$.
Hilbert originally posed this question for $F = F_0 = \zz$, and in this form it was solved negatively by the work of Davis, Putnam, Robinson and Matiyasevich; see~\autocite[Section 3.3]{Koe13} for a survey on these and other related results.
Positive answers were given for $F=\rr$ and for $F = \cc$, by the work of Tarski \autocite{Tar51}, as well as for $F=\qq_p$ by the work of Ax and Kochen \autocite{Ax-Kochen-DiophantineLocalII, Ax-Kochen-DiophantineLocalIII}, each time for $F_0 = \zz$.
Infamously, Hilbert's 10th Problem for the field of rational numbers $\qq$ is still wide open.
If we knew that $\zz$ were existentially definable in $\qq$, the negative answer for $\qq$ would follow immediately from that for $\zz$, but this question is equally open.

Based on the link to existential definability, a negative answer to Hilbert's 10th Problem for quite a few classes of fields has meanwhile been established.
In the research in this direction, the importance of building a library of existentially definable sets was recognised early on, whence the interest in proving that arithmetically significant sets are frequently existentially definable.
Of particular concern here are valuation rings, whose importance for undecidability questions was already implicitly recognised and used in \autocite[Section 3]{Rob59} for number fields and in \autocite{DenefDiophantine} for function fields.
From this arose an interest in studying existentially definable sets for their own sake; see for instance \autocite{Koe16,Morrison_cyclic,Dit17} for various results on existentially definable sets in global fields.

In this article, we study the situation where $F$ is an extension of transcendence degree $1$ of some base field $K$ such that the field extension $F/K$ is finitely generated. In this case the extension $F/K$ is called a \emph{function field in one variable}.
To make statements about Hilbert's 10th Problem precise,
we will for such a function field consider a finitely generated subfield $F_0 \subseteq F$, and make statements in this setting about Hilbert's 10th Problem for $F$ with coefficients in $F_0$.

To this date, there is not a single instance of this setting where a positive answer for Hilbert's 10th Problem for $F$ with coefficients in a field $F_0$ transcendental over $K$ is known, while there are quite a few cases where the negative answer has been established.
This includes the following cases:
\begin{itemize}
\item $F$ is real, i.e.~$-1$ is not a sum of squares in $F$ \cite{Mor05};
\item $K$ is a subfield of a $p$-adic local field (finite extension of $\qq_p)$ for some prime number $p$  \autocite{Mor05, EisHil10padic, DD12}; for $p=2$ this is covered explicitly in the literature only in the case where $F/K$ is a rational function field;
\item $K$ has positive characteristic and contains no algebraically closed subfield \autocite[Theorem 1.1]{EisShlap17};
\item $K$ is a finitely generated transcendental extension of an algebraically closed field of characteristic different from $2$ \autocite{EisFunctionFieldsOverC, EisFunctionFieldsPositiveCharacteristic}.
\end{itemize}
In this article we  study the cases where the base field $K$ of the extension $F/K$ is a global field or a large field.
By a \emph{global field} we mean a number field (i.e.~a finite field extension of $\qq$) or a function field in one variable over a finite field.
Following \cite{Pop}, a field $K$ is called \emph{large} if every smooth curve over $K$ which has a $K$-rational point already has infinitely many $K$-rational points. 
Typical examples are given by algebraically closed, pseudo-algebraically closed and real closed fields, and by fields that carry a nontrivial henselian valuation.
Finite fields, global fields and arbitrary finitely generated non-algebraic extensions of some other field are not large.
We refer to \Cref{sec:large} below for more information on large fields.

Our main result can be stated as follows.

\begin{stel1}[\Cref{T:DefiningValuationsFunctionLargeUniform}, \Cref{T:DefiningValuationsFunctionGlobal}]
  Let $K$ be either a global field or a large field for which $K(\sqrt{-1})$ has a separable quadratic extension.
  Let $F/K$ be a function field in one variable.
Then the valuation rings of $F$ containing $K$ are uniformly existentially definable.
\end{stel1}

The \emph{uniformity} in the statement above expresses that we can define all nontrivial valuation rings of $F$ containing $K$ by a single type of existential formula, with only a different choice of parameters for each valuation ring.
In fact, even the dependence of the formula on the specific field $F$ is only very mild.
See \Cref{T:DefiningValuationsFunctionLargeUniform}, \Cref{T:DefiningValuationsFunctionGlobal} and \Cref{rem:E-global-uniform} for the precise formulation.

Theorem~I covers several cases where the existential definability of all valuation rings of $F$ containing $K$ is already known, albeit without the uniformity statement.
Recently  this was shown in \autocite[Theorem 6.1]{MillerShlapentokh_v2} for the case where $K$ is a number field or, more generally, an algebraic extension of $\qq$ contained in $\qq_p$ for some odd prime $p$.

One motivation to look for uniform existential (or, more generally, first-order) definitions of valuation rings is their crucial role in certain general definability problems, such as the definitions of Gödel functions in \autocite{Rumely} and the resolution of the Elementary Equivalence vs.\ Isomorphism Problem for finitely generated fields in \autocite{Pop_ElemEquivVsIsomII,DittmannPop}.

For a function field in one variable $F/K$, it is well-known that the presence of a discrete valuation ring of $F$ containing $K$ which is existentially definable in $F$ implies a negative answer to Hilbert's 10th Problem for $F$; see \Cref{P:Hilbert10CriterionPrecise}.
As such, from Theorem~I we obtain the following:
\begin{stel2}[\Cref{C:Hilbert10-large}, \Cref{C:Hilbert10-global}]
  Let $K$ be either a global field or a large field for which $K(\sqrt{-1})$ has a separable extension of even degree.
  Then, for any function field in one variable $F/K$, there exists a finitely generated subfield $F_0$ of $F$ such that Hilbert's 10th Problem for $F$ with coefficients in $F_0$ has a negative answer.
\end{stel2}
One can be more precise about the choice of the subfield $F_0$; see  \Cref{C:Hilbert10-large} for details.

As mentioned above, this recovers several cases where this undecidability statement was already known.
It further covers now the case where $K=k(\!(t)\!)$, the Laurent series field in one variable, for an arbitrary field $k$.
Since $k$ here may be of any characteristic and contain an algebraically closed subfield, this case is novel, and in fact it lies outside of the scope of the previously existing techniques.
(We refer to the beginning of \Cref{sec:large} for further discussion.)
In the case $K = \cc(\!(t)\!)$, the undecidability statement was also shown in the unpublished manuscript \cite{Demeyer-H10-ffvalued} by different methods, in particular without passing through the existential definability of valuation rings.

Furthermore, the uniformity of our existential definition in Theorem~I opens the door to proving (existential) definability of other arithmetically significant subsets of the function fields which we consider.
This is further explored in follow-up work by the second and third author \cite{DD}; see \Cref{rem:E-global-uniform}.
For the sake of illustration we include a first consequence in this direction in this article.
When $F/K$ is a function field in one variable, the integral closures of finitely generated $K$-subalgebras of $F$ not algebraic over $K$ are called \emph{finitary holomorphy rings of $F/K$}; see \Cref{P:fingenhol} for alternative characterisations.

\begin{stel3}[{\Cref{T:AE-large}, \Cref{T:AE-global}}]
Let $K$ be either a global field or a large field for which $K(\sqrt{-1})$ has a separable extension of even degree.
Then, for any function field in one variable $F/K$, the finitary holomorphy rings of $F/K$ are uniformly universally-existentially definable in $F$.
\end{stel3}

 Our approach to proving the aforementioned theorems is largely uniform across all fields which are covered, also regardless of their characteristic.
Of course, the ingredients which we are relying on may be proven by very different methods, as is for example the case for  \Cref{P:Hilbert10CriterionPrecise}, whose proofs in characteristic zero and in positive characteristic are very different.

Like in many preceding works on this topic, we make extended use of quadratic form theory to obtain our existential definitions of valuation rings in function fields.
In particular, the usage of classical local-global principles for quadratic forms, such as the Hasse-Minkowski Theorem for number fields, has a long history in this context.
Typically, such a local-global principle (or \emph{Hasse principle}) determines whether a quadratic form has a nontrivial zero in terms of local conditions.
For example, this was used already in \autocite{Rob49} to show (essentially) that valuation rings of $\qq$ are existentially definable.
A novelty in this article is that we are making use of local-global principles beyond those from classical number theory.
When $K$ is a global field, we use a local-global principle for so-called $3$-fold Pfister forms over $F$, which is derived from a local-global principle in Galois cohomology from \autocite{Kato}.
To obtain results in the situation where $K$ is large, we use a very recent local-global principle for quadratic forms 
from \cite{CTPS12, Mehmeti_PatchingBerkQuad}.
The applicability of this approach relies on a transfer of the problem where $K$ is replaced by  $K(\!(t)\!)$.

In order to include fields of characteristic $2$ in our study, we present the required quadratic form theory in a characteristic-free way, 
and some of our preparatory results in \Cref{sect:qfs} and \Cref{S:QFVF} about quadratic forms over valued fields and over function fields of characteristic $2$ are of independent interest.

Besides the ingredients already mentioned, our techniques most directly owe a debt on the one hand to the existential definitions of semilocal subrings of number fields introduced by Poonen in \autocite{Poo09}, having led in the sequel to the results of \autocite{Koe16,Par13,Eis18,DaansGlobal}, and on the other hand to the usage of higher-rank Pfister forms and cohomological results premiered in \autocite{Pop_ElemEquivVsIsomI}, which were previously used in \autocite{Pop_ElemEquivVsIsomII,DittmannPop}.

\smallskip
We now give an overview of the structure of the article.
In the short Section 2, we define the terminology concerning existential definability which we use.

Section 3 introduces basic material on quadratic forms.
Much of this is standard; our specific presentation is chosen to minimise case distinctions between characteristic $2$ and other characteristics.
We introduce a set $\Split(q)$ associated to a quadratic form $q$, encoding information about the separable quadratic extensions of the ground field where $q$ becomes isotropic.
This set is existentially definable, and it plays an important role throughout.

Section 4 is concerned with quadratic forms over valued fields, in particular henselian ones.
Once again some of this material is standard, but we take particular care to include the case of residue characteristic $2$ throughout.

Section 5 is dedicated to the properties of the quadratic splitting set $\qss q c$ associated to a quadratic form $q$ and some parameter $c$.
This existentially definable set is the core tool for the existential definitions which follow.
While the precise definition of $\qss q c$ is new, it has a number of predecessors in the literature, which we discuss in detail.

Section 6 discusses function fields in one variable.
A core aspect here is the interaction among the infinite family of natural valuations on a function field.
The later part of this section is quite technical, but only needed for some of our uniformity statements, not for the applications to Hilbert's 10th Problem.

In Section 7 we state a number of local-global principles for quadratic forms for function fields in one variable over certain base fields, with a view towards making the local theory of Section 4 applicable to function fields.
Much of this section involves reformulating results from the literature in a way suited to our goals.

The first half of Section 8 notes some consequences of the local-global principles for definability purposes.
The second half discusses briefly the well-established relation between existential definability of valuation rings and Hilbert's 10th Problem which we depend on.

Finally, Section 9 contains our main results in the case of a large base field, and Section 10 in the case of a global base field.
In both cases, the key additional ingredient compared to preceding work lies in the determination of a sufficiently large existentially definable subset of the base field. To find such a set, we rework classical arguments using elliptic and hyperelliptic curves.

\subsection*{Acknowledgments}
The authors wish to express their gratitude to Stevan Gajo\-vic, Yong Hu, David Leep, Vler\"e Mehmeti, H\'ector Past\'en, Jean-Pierre Tignol and Christian Wuthrich for inspiring discussions and helpful comments, and to the anonymous referee for their careful reading of the manuscript and several pertinent comments and suggestions.

\subsection*{Tool disclosure}
AI (Claude Opus 5.5 and GPT-6 Pro) was used in the final revision stage to scrutinize the manuscript. It helped to identify several minor gaps and mistakes, which were repaired.

\section{Existentially definable sets over fields} 

We shall always denote by $\nat$ the set of natural numbers including $0$, and by $\nat^+$ the subset of nonzero natural numbers.

In the focus of this article are certain first-order definability results for fields.
We refer to \autocite[Sections II and III]{Ebb94} for an introduction to the syntax and semantics of first-order languages.

Given a first-order language $\La$, by an \emph{existential $\La$-formula} we mean an $\La$-for\-mula of the form $\exists x_1 \ldots \exists x_n \psi$ for certain variables $x_1, \ldots, x_n$ and a quantifier-free $\La$-formula $\psi$.
By a \emph{universal-existential $\La$-formula} we mean an $\La$-formula of the form $\forall x_1 \ldots \forall x_n \psi$ for certain variables $x_1, \ldots, x_n$ and an existential $\La$-formula $\psi$.

We work in the first-order language of rings $\Lar$ consisting of three binary function symbols $+$, $\cdot$ and $-$ and two constant symbols $0$ and $1$. A ring naturally carries an $\Lar$-structure by interpreting $+$ as addition, $\cdot$ as multiplication, $-$ as subtraction, and $0$ and $1$ as the neutral elements for addition and multiplication respectively.

Consider a ring $R$ and a  subset $S\subseteq R$. 
We denote by $\Lar(S)$ the first-order language obtained by extending $\Lar$ by a list of constant symbols $(c_s)_{s\in S}$.
By the \emph{natural interpretation of $R$ as an $\Lar(S)$-structure} we mean the interpretation given by interpreting,
 for any $s \in S$, the symbol $c_s$ as the element~$s$.

Consider $n\in\nat$ and a set $A \subseteq R^n$. We say that $A$ is \emph{existentially $\Lar(S)$-definable in $R^n$} if there is some existential $\Lar(S)$-formula $\varphi(x_1, \ldots, x_n)$ such that $$ A = \lbrace (a_1, \ldots, a_n) \in R^n \mid R \models \varphi(a_1, \ldots, a_n) \rbrace.$$
We will simply call $A$ \emph{existentially definable in $R^n$} if it is existentially $\Lar(R)$-definable in $R^n$.
If $R$ is a field, then this definition is equivalent to the one given at the beginning of the introduction; see e.g.~\cite[Remark 3.4 and paragraph after Example 4.8]{DDF}.
To an \emph{existential $\Lar$-formula} (resp.~to an existential $\Lar(S)$-formula), we will shortly refer as an
\emph{$\exists$-$\Lar$-formula} (resp.~ an \emph{$\exists$-$\Lar(S)$-formula}).

Similarly, we call $A$ \emph{universally-existentially $\Lar(S)$-definable in $R^n$} if there is a universal-existential $\Lar(S)$-formula defining $A$ in $R^n$.
We will shortly refer to a universal-existential $\Lar(S)$-formula as an $\forall\exists$-$\Lar(S)$-formula.

Conversely, given an $\Lar(S)$-formula $\varphi(x_1, \ldots, x_n)$, we write $\varphi(R^n)$ for the subset of $R^n$ defined by $\varphi$.
We might also, for example, given $a_2, \ldots, a_{n} \in R$, write $\varphi(R, a_2, \ldots, a_{n})$ for the set
$ \lbrace a \in R \mid R \models \varphi(a, a_2, \ldots, a_{n}) \rbrace$.

\section{Quadratic forms and separable quadratic extensions}\label{sect:qfs} 

In this section we introduce some basic quadratic form theory over fields.
Although most notions and definitions are very standard, a particular flavour in our presentation is that we want it to cover them uniformly over base fields of arbitrary characteristic as far as possible.
We refer to \autocite[Chapter II]{ElmanKarpenkoMerkurjev} for a  systematic treatment of quadratic form theory over fields of arbitrary characteristic.
We opt for an ad hoc approach to quadratic Pfister forms in arbitrary characteristic, circumventing the use of the tensor product of quadratic and symmetric bilinear forms in their construction.
Some of the statements are coined especially for the study of the interaction of quadratic forms with valuations, which will be then the topic of the next section. 
\medskip

Let $K$ be a commutative ring. (Soon we will assume $K$ to be a field.)
In the first place, a \emph{quadratic form over $K$} is a homogeneous polynomial $q$ of degree $2$ with coefficients in $K$ in a given number of variables, which is denoted by $\dim(q)$ and called the \emph{dimension of $q$}.
For $n\in\nat$, we denote by $\Quad_n(K)$ the $K$-submodule of $K[X_1,\dots,X_n]$ consisting of the $n$-dimensional quadratic forms.

Let $n\in\nat$  and $q\in\Quad_n(K)$.
Evaluation of $q$ yields a map $K^n\to K$, which we also denote by $q$. 
The features of $q$ which are interesting to us are expressed in terms of this map. By stating that $q$ is a quadratic form \emph{over $K$}, we are therefore expressing that this evaluation takes place over $K$.
If $K$ is a domain, then the $n$-dimensional quadratic form $q$ is characterised by the induced map $q:K^n\to K$.
Given an extension of commutative rings $L/K$, we write $q_L$ for $q$ if we formulate properties of $q$ related to the evaluation over $L$ obtained from $q$ as a polynomial.

A quadratic form $q\in\Quad_n(K)$ induces a map $$K^n \times K^n \to K : (v, w) \mapsto q(v+w) - q(v) - q(w)\,,$$ which is a symmetric $K$-bilinear form on $K^n$ and which we denote by $\mathfrak{b}_q$.

Assume from now on that $K$ is a field. 
Although we primarily work with quadratic forms that are concretely given as polynomials, we will make use of the coordinate-free view on quadratic forms.
Let $n\in\nat$ and let $V$ be an $n$-dimensional $K$-vector space.
A \emph{quadratic map on $V$} is a map $q:V\to K$ for which the map
$$\mathfrak{b}_q:V \times V \to K : (v, w) \mapsto q(v+w) - q(v) - q(w)$$
is $K$-bilinear and such that $q(\lambda v)=\lambda^2q(v)$ for all $\lambda\in K$ and all $v\in V$.
In this case, $\mathfrak{b}_q$ is a symmetric $K$-bilinear form on $V$.
We denote by $\Quad(V)$ the $K$-vector space of quadratic maps on $V$.
For $q\in\Quad(V)$, we also refer to the pair $(V,q)$ or simply to $q$ as a \emph{quadratic form over $K$}.
Note that, with this setup, we have $\Quad_n(K)=\Quad(K^n)$.

Hence, for $q\in\Quad_n(K)$, we have that $(K^n,q)$ is an $n$-dimensional quadratic space over $K$, and conversely, given an $n$-dimensional quadratic space $(V,q)$ over $K$, taking any $K$-basis $\mc{B}=(v_1,\dots,v_n)$, there exists a  unique quadratic form $q_{\mc{B}}\in\Quad_n(K)$ satisfying
$q(\sum_{i=1}^nx_iv_i)=q_{\mc{B}}(x_1,\dots,x_n)$ for all $(x_1,\dots,x_n)\in K^n$.
A change of the basis $\mc{B}$ changes the coefficients of the corresponding form $q_{\mc{B}}$, but this is given by a 
linear transformation on the variables of $q_{\mc{B}}$, which is given by the matrix for the given change of basis.
We are interested in properties of quadratic forms which are invariant under linear variable transformations.

We call the quadratic form $(V,q)$ over $K$ a \emph{subform} of another quadratic form $(V',q')$ over $K$ if
there exists an injective $K$-linear map $f:V\to V'$ such that $q=q'\circ f$.
By an \emph{isometry} between two quadratic forms $(V,q)$ and $(V',q')$ over $K$ we mean a bijective $K$-linear map $f:V\to V'$ with $q=q'\circ f$, and if such $f$ exists, we call $(V,q)$ and $(V',q')$ \emph{isometric} and we write
$(V,q)\simeq (V',q')$, or simply $q\simeq q'$.
Given a field extension $L/K$, we say that $q$ and $q'$ are \emph{isometric over $L$} if $q_L \simeq q_L'$ as quadratic forms over $L$.

The quadratic form $q$ over $K$ is called \emph{isotropic} if $q(x)=0$ holds for some $x\in K^n\setminus\{0\}$, and it is called \emph{anisotropic} otherwise. 
Note that, for a field extension $L/K$, isotropy (or anisotropy) of $q_L$ is referring to the map $q_L:L^n\to L$, and hence it is a feature not only depending on $q$ but also on $L$.

We call a quadratic form $q \in \Quad_n(K)$ \emph{regular} if for all $x \in K^n \setminus \lbrace 0 \rbrace$ with $q(x) = 0$ there exists $y \in K^n$ with $\mf{b}_q(x, y) \neq 0$.
In particular, every anisotropic quadratic form is regular.
Regularity is a feature of a quadratic form which a priori depends on the base field.
We call the quadratic form $q$ \emph{non-degenerate} if $q_L$ is regular for every field extension $L/K$.
Clearly, non-degenerate quadratic forms are regular.
Both properties are in fact equivalent over fields of characteristic different from $2$.

\begin{prop}\label{P:non-degenerate-characterisations}
Assume that $n \geq 2$. Let $q\in\Quad_n(K)$ and 
consider the $K$-subspace $U= \lbrace x \in K^n \mid \forall y \in K^n : \mf{b}_q(x, y) = 0 \rbrace$ of $K^n$. The following are equivalent:
\begin{enumerate}[$(1)$]
\item\label{it:ndg1} $q$ is non-degenerate.
\item\label{it:ndg2} the projective quadric in $\mbb{P}^{n-1}_K$ given by $q$ is smooth.
\item\label{it:ndg3} Either $U=\{0\}$, or we have $\car(K)=2$, $n$ is odd, $\dim_K U=1$  and $q|_U\neq 0$.
\end{enumerate}
\end{prop}
\begin{proof}
The equivalence between \eqref{it:ndg1} and $\eqref{it:ndg2}$ is \autocite[Proposition 22.1]{ElmanKarpenkoMerkurjev}.
The implication from \eqref{it:ndg3} to \eqref{it:ndg1} is \autocite[Lemma 7.16]{ElmanKarpenkoMerkurjev}.
The converse implication is \autocite[Remark 7.21]{ElmanKarpenkoMerkurjev} if $\car(K) = 2$, and if $\car(K) \neq 2$, then it follows because $q(x) = \frac{1}{2}\mf{b}_q(x, x)$ for all $x \in K^n$.
\end{proof}

We refer to \autocite[Section 7.A]{ElmanKarpenkoMerkurjev} for a discussion of non-degenerate quadratic forms.
Examples of non-degenerate quadratic forms are given by Pfister forms.
We shall use the term \emph{Pfister form} for what is called a quadratic Pfister form in \autocite[Section 9]{ElmanKarpenkoMerkurjev}, where they are defined in a conceptual way and discussed more systematically.
Here, we use an ad hoc definition, which works independently of the characteristic of $K$. 

We need the definition of the orthogonal sum of two quadratic forms.
For $n,n'\in\nat$, $q\in\Quad_n(K)$ and $q'\in\Quad_{n'}(K)$,
the $(n+n')$-dimensional quadratic form $q(X_1,\dots,X_n)+q'(X_{n+1},\dots,X_{n+n'})$ over $K$
is denoted by
$q\perp q'$, and it is called the \emph{orthogonal sum of $q$ and $q'$}.

For $c\in K$, we denote the $2$-dimensional quadratic form $$X_1^2 - X_1X_2 - cX_2^2$$ over $K$  by $\pfi{c}K$.
Recursively, we now set 
$$\pfi{a_1, \ldots, a_n}K=\pfi{a_2, \ldots, a_n}K \perp -a_1\pfi{a_2, \ldots, a_n}K$$
for any $n\in\nat\setminus\{0,1\}$ and $a_1,\dots,a_n\in K$.
 
For a domain $R$, we introduce the notation
$$\qse{R}=\{x\in R\mid 1+4x\in\mg{R}\}\,.$$
For $n\in\nat^+$, a quadratic form over $K$ is called an \emph{$n$-fold Pfister form} if it is isometric over $K$ to $\pfi{a_1, \ldots, a_n}K$ for certain $a_1,\dots,a_{n-1}\in \mg{K}$ and $a_n\in \qse{K}$.

In the sequel we will make crucial use of the flexibility in presenting a given Pfister form up to isometry.
Note that, given an $n$-fold Pfister form $q$ over $K$, $q\perp aq$ is an $(n+1)$-fold Pfister form over $K$ for any $a\in\mg{K}$.

\begin{rem}
Our use of the notation $\pfi{\dots }{}$ disagrees with the convention taken in \autocite{ElmanKarpenkoMerkurjev}. There, for $c\in\mg{K}$, if $\car K\neq 2$, the notation
 $\pfi{c}K$ refers to the quadratic form $X_1^2 - cX_2^2$, see \autocite[Example 9.4~(1)]{ElmanKarpenkoMerkurjev}.
Nevertheless, this does not affect the resulting notion of quadratic Pfister form, which is up to isometry.
\end{rem}

For $c\in K$, the $2$-dimensional quadratic form $\pfi{c}K$ is related to the quadratic $K$-algebra $K[T]/(T^2-T-c)$.
For a finite-dimensional $K$-algebra $L$, we denote by $\N_{L/K}:L\to K$ the norm map, and we observe that, expressed in a $K$-basis of $L$, it is given by a homogeneous polynomial of degree equal to $[L:K]$.

\begin{prop}\label{P:1Pfi-quadalg}
Let $c\in K$ and $L=K[T]/(T^2-T-c)$, viewed as a $K$-algebra.
\begin{enumerate}[$(a)$]
\item The norm form of $L/K$ is given by $\pfi{c}K$.
\item $L$ is a field if and only if $\pfi{c}K$ is anisotropic.
\item $L/K$ is separable if and only if $\pfi{c}K$ is non-degenerate, if and only if $c\in \qse{K}$.
\end{enumerate}
\end{prop}

\begin{proof}
Let $p=T^2-T-c$ and $\vartheta=T+(p)\in L$. Then $(1,\vartheta)$ is a $K$-basis of $L$ and 
$\N_{L/K}(x-y\vartheta)=x^2-xy-y^2c$ for $x,y\in K$. This shows $(a)$.

We have $\mg{L}=\{x\in L\mid \N_{L/K}(x)\neq 0\}$. This yields $(b)$.

The discriminant of $p$ is $1+4c$. Hence $1+4c\neq 0$ if and only if $p$ is separable, which is if and only if $L/K$ is separable. This shows $(c)$.
\end{proof}

We will now study separable quadratic field extensions of $K$ over which a given Pfister form is isotropic.
Let $c \in K$.
We denote by $\sep{K}{c}$ the splitting field of the polynomial $T^2 - T - c$ over $K$.
In other terms, 
\begin{equation*}
\sep{K}{c} \,\, = \,\,  \left\{\begin{array}{cl} K & \mbox{ if $T^2 - T - c$ has a root in $K$}, \\ K[T]/(T^2 - T - c) & \mbox{ otherwise.} \end{array}\right.
\end{equation*}
Since $T^2 - T - c$ is a separable polynomial as long as $c \in \qse{K}$,
and otherwise it has the double root $\frac{1}{2}$, which is in $K$, 
we see that in all cases $\sep{K}{c}/K$ is a separable field extension of degree at most $2$.

\begin{prop}\label{L:sepQuadPfister}
$K$ has a separable quadratic field extension if and only if there exists  $c \in \qse{K}$ such that $\pfi{c}K$ is anisotropic.
\end{prop}
\begin{proof}
Any separable quadratic $K$-algebra is isomorphic to $K[T]/(T^2-T-c)$ for some $c\in \qse{K}$, and, by \Cref{P:1Pfi-quadalg}, it is a field if and only if $\pfi{c}K$ is anisotropic.
\end{proof}

Let $(V,q)$ be an $n$-dimensional quadratic space over $K$.

\begin{lem}\label{L:f(T)isotropy}
If there exist $x, y \in V$ such that $q(x) + T\mf{b}_q(x, y) + T^2q(y)$
has a simple root in $K$, then $q$ is isotropic.
\end{lem}
\begin{proof}
Consider $x,y\in K^n$ and set $f = q(x) + T\mf{b}_q(x, y) + T^2q(y)$.
If $y = 0$, then $f = q(x)\in K$, so $f$ has no simple root.
If $y\neq 0$ but $q(y)=0$, then $q$ is isotropic.
Assume now that $q(y)\neq 0$ and that $f$ has a simple root $t_1$ in $K$. 
Then it follows that $f$ has precisely one root $t_2\in K$ with $t_2\neq t_1$.
Then $q(x + ty) = f(t) = 0$ for $t\in\{t_1,t_2\}$.
Since $t_1\neq t_2$ and $y\neq 0$, we can choose $t\in\{t_1,t_2\}$ in such way that $x+ty\neq 0$.
Hence $q$ is isotropic.
\end{proof}

We now introduce a set that will serve to parametrise the separable quadratic extensions of $K$ where the quadratic form $q$ becomes isotropic.
We set 
$$\Split(q) = \qse{K}\cap\left\lbrace \mbox{$-\frac{q(x)q(y)}{\mf{b}_q(x, y)^2}$} \enspace\middle|\enspace x, y \in V: \mf{b}_q(x, y) \neq 0 \right\rbrace.$$
For any $x,y\in V$ with $\mf{b}_q(x,y)\neq 0$,  letting $\lambda=\mf{b}_q(x,y)^{-1}$, we have $\mf{b}_q(x,\lambda y)=1$ and $\frac{q(x)q(y)}{\mf{b}_q(x,y)^2}=q(x)q(\lambda y)$.
Therefore we have that
$$\Split(q) = \qse{K}\cap\lbrace -q(x)q(y)\mid x,y\in V: \mf{b}_q(x,y)=1\}\,.$$
Note further that $\Split(q) = \Split(aq)$ for any $a \in K^\times$. 
Hence when studying $\Split(q)$ we can rescale $q$, to assume for example that $q$ represents $1$.

\begin{rem}\label{R:Ulinearlydependent}
For any $x \in V$ and $\lambda \in K$ with $\mf{b}_q(x, \lambda x) \neq 0$, we
have that $2\lambda q(x)=\mf{b}_q(x, \lambda x)\neq 0$, whereby $\car(K)\neq 2$ and 
$-\frac{q(x)q(\lambda x)}{\mf{b}_q(x,\lambda x)^2} = -\frac{1}4\notin \qse{K}$.
Therefore, for $x,y\in V$ with $\mf{b}_q(x, y) \neq 0$ and $-\frac{q(x)q(y)}{\mf{b}_q(x,y)^2}\in \qse{K}$, we have
that $x$ and $y$ are $K$-linearly independent.
\end{rem}

\begin{ex}\label{EX:S(H)}
We denote by $\hh_K$ the quadratic form $X_1X_2\in \Quad_2(K)$.
Up to isometry, $\hh_K$ is the unique isotropic regular $2$-dimensional form over $K$.
One finds that $$\Split(\hh_K) = \qse{K} \cap \lbrace t^2 - t \mid t \in K \rbrace\,.$$
\end{ex}

We will use the set $\Split(q)$ to various related purposes, in particular in the context of valuations and to study questions on existential definability of certain sets. We will give the link at the end of this section.

First, we relate properties of $q$ to the set $\Split(q)$.

\begin{prop}\label{P:0-in-Split(q)}
Assume that $q$ is regular. Then the following hold:
\begin{enumerate}[$(a)$]
\item $\Split(q)=\{c\in \qse{K}\mid q \mbox{ has a subform } d\pfi{c}K\mbox{ with }d\in \mg{K}\}$.
\item $0\in\Split(q)$ if and only if $q$ is isotropic.
\item $\Split(q)=\emptyset$ if and only if $\mf{b}_q$ is the zero map or $\dim q\leq 1$.
\item If $q$ is isotropic with $\dim q\geq 3$, then $\Split(q)=\qse{K}$.
\item If $q\not\simeq \hh_K$, then $\Split(q)=\{c\in\qse{K}\mid q_{\sep{K}c}\mbox{ isotropic}\}$.
\end{enumerate}
\end{prop}
\begin{proof}
$(a)$ Clearly, if $d\in\mg{K}$ is such that $d\pfi{c}K$ is a subform of $q$, then we have
$c\in\Split(\pfi{c}K)=\Split(d\pfi{c}K)\subseteq \Split(q)$.
Assume conversely that $c\in\Split(q)$.
Hence there exist $x,y\in V$ such that $\mf{b}_q(x,y)=1$ and $c=-q(x)q(y)$.
Since $c\in\qse{K}$ we have by \Cref{R:Ulinearlydependent} that $x$ and $y$ are $K$-linearly independent.
We set $W=Kx\oplus Ky$ and consider the form $\psi=(W,q|_W)$.
Note that $\psi$ is a $2$-dimensional quadratic form over $K$.
If $c=0$, then $\psi\simeq \hh_K\simeq\pfi{c}K$.
Assume now that $c\neq 0$.
Set $d=q(x)$. Then $d\in\mg{K}$ and $\psi\simeq dX_1^2+X_1X_2-cd^{-1}X_2^2\simeq d\pfi{c}K$.

$(b)$\, Clearly, if $0\in\Split(q)$, then $q$ is isotropic.
Assume now that $q$ is isotropic. Fix $x\in V\setminus\{0\}$ with $q(x)=0$.
Since $q$ is regular, there exists $y\in V$ with $\mf{b}_q(x,y)=1$.
Then  $0=-q(x)q(y)\in\Split(q)$.

$(c)$\, Clearly, if $\dim(q)\leq 1$ or $\mf{b}_q$ is the zero map, then $\Split(q)=\emptyset$.
By $(b)$, if $q$ is isotropic, then $0\in\Split(q)$.
Assume now that $q$ is anisotropic, $\dim(q)\geq 2$ and $\mf{b}_q$ is not the zero map.
Then $q$ has a regular $2$-dimensional subform $q'$ such that $\mf{b}_{q'}$ is not the zero map.
It follows that $q'\simeq d\pfi{c}K$ for some $d\in\mg{K}$ and $c\in\qse{K}$.
Then $c\in\Split(q')\subseteq \Split(q)$, whereby $\Split(q)\neq \emptyset$.

$(d)$\,
The hypothesis implies that there exist  $x,y,z\in V$ with $\mf{b}_q(x,y)=1$, $q(x)=q(y)=\mf{b}_q(x,z)=\mf{b}_q(y,z)=0$ and $q(z)\neq 0$.
Since we may scale first $q$ and then $y$, we can assume that $q(z)=1$.
For any $\lambda\in K$ we obtain that $q(\lambda x + y+z)q(x-z)/\mf{b}_q(\lambda x+y+z,x-z)^2 = \lambda+1$.
This implies that $\Split(q)=\qse{K}$.

$(e)$\, Consider $c\in \qse{K}$.
If $c\in\Split(q)$, then there exists $d\in\mg{K}$ such that $d\pfi{c}K$ is a subform of $q$, and since $\pfi{c}{\sep{K}{c}}$ is isotropic, it follows that $q_{\sep{K}c}$ is isotropic. 

If $q$ is anisotropic, but $q_{\sep{K}c}$ is isotropic, then
it follows by \autocite[Prop.~22.11]{ElmanKarpenkoMerkurjev} that $q$ has a subform $d\pfi{c}K$ for some $d\in\mg{K}$, whereby 
$c\in\Split(q)$.

If $q$ is isotropic, then so is $q_{\sep{K}c}$, and as $q\not\simeq\hh_K$, we have $\dim q\geq 3$ and obtain by $(d)$ that 
$\Split(q)=\qse{K}$, whereby $c\in \Split(q)$.
\end{proof}

\begin{prop}\label{P:Split(q)=rightslot}
Let $n \in \nat$. Assume that $q$ is an $(n+1)$-fold Pfister form. Let $c\in\qse{K}$.
Then $c \in \Split(q)$ if and only if $q \simeq \pfi{a_1, \ldots, a_n, c}K$ for certain $a_1, \ldots, a_{n} \in K^\times$.
\end{prop}
\begin{proof}
Clearly we have $c\in\Split(\pfi{c}K)$. 
Hence, if $q \simeq \pfi{a_1, \ldots, a_n, c}K$ for some $a_1,\dots,a_n\in\mg{K}$, then $c\in\Split(q)$.
Assume now that $c\in\Split(q)$.  Hence, for some $d\in\mg{K}$, $q$ has $d\pfi{c}K$ as a subform by \Cref{P:0-in-Split(q)}.
Since $q$ is a Pfister form, we have that $dq\simeq q$. Hence $\pfi{c}K$ is a subform of $q$, and
we conclude by \autocite[Proposition 24.1~(1)]{ElmanKarpenkoMerkurjev} that $q\simeq\pfi{a_1, \dotsc, a_{n}, c}K$
for certain $a_1,\dots,a_{n}\in\mg{K}$.
\end{proof}

For our purposes in the context of definability, we need to study the following problem.
Consider a subring $R$ of $K$ with fraction field $K$ and a quadratic form $q$ over $K$, and define
$$\Split_R(q)=\Split(q)\cap \mg{R}\cap\qse{R}\,.$$

\begin{ques}\label{Q:slot-ring}
Is $\Split_R(q)\neq \emptyset$?
\end{ques}

The first situation to look at is that of a valuation ring itself, which we will handle in the next section.
In \Cref{S:ffiov}, we will consider this question for certain intersections of valuation rings in a function field in one variable.

We end this section by the basic observation that, for $q\in\Quad_n(K)$, $\Split(q)$ is existentially definable in terms of the coefficients of $q$.

\begin{prop}\label{P:S-existential}
Let $n\in\nat$.
There exists an $\exists$-$\Lar$-formula $\varphi$ in $\binom{n+1}{2}+1$ free variables such that,
for any field $K$ and any $(a_{ij})_{1\leq i\leq j\leq n}$, we have 
$$\varphi(K,a_{11},a_{12},a_{22},a_{13},\dots, a_{nn})=\Split(q)$$
for the quadratic form $q=\sum_{1\leq i\leq j\leq n} a_{ij}X_iX_j$ over $K$.
\end{prop}
\begin{proof}
The  conditions involved in the definition of $\Split(q)$ above can be expressed as $\exists$-$\Lar$-formulas in terms of  the coefficients of $q$, that is, by a formula $\varphi$ with $\binom{n+1}{2}+1$ free variables as claimed. 
\end{proof}

For later use, we reformulate this for the case of Pfister forms.

\begin{cor}\label{EX:Pfi-Split}
Let $n\in\nat^+$. There exists an $\exists$-$\Lar$-formula $\phi$ in $n+1$ free variables such that, for every field $K$ and every $a_1,\dots,a_{n}\in {K}$, we have
$$\phi(K,a_1,\dots,a_n)=\Split(\pfi{a_1,\dots,a_n}K)\,.$$
\end{cor}
\begin{proof}
This follows by combining \Cref{P:S-existential} with the construction of the form $\pfi{a_1,\dots,a_n}K$.
\end{proof}

\section{Quadratic forms over valued fields} 
\label{S:QFVF}

In this section, we study quadratic forms on fields in the context of a given valuation on the field.
For a valuation $v$ on $K$, we write $K_v$, $\widehat{K_v}$, $\mathcal{O}_v$, $\mathfrak{m}_v$, $Kv$ and $vK$ for the henselisation, completion, valuation ring, valuation ideal, residue field, and  value group of $v$, respectively.
We call a valuation $v$ \emph{dyadic} if $v(2)>0$ and \emph{non-dyadic} otherwise.
We call a valuation a \emph{$\zz$-valuation} if its value group is $\zz$, and we call a valuation \emph{discrete} if its value group is isomorphic to $\zz$.
We will only consider completions of valuations of rank $1$, that is, for which the value group can be seen as a nontrivial ordered subgroup of $\rr$.
For such a valuation $v$, the completion $\hat{K}_v$ is henselian, and thus in particular it contains the henselisation $K_v$ as a subfield.

Let $(K, v)$ be a valued field. For $a \in \mc{O}_v$, we denote by $\ovl{a}$ (or $\ovl{a}^v$) the residue $a+\mfm_v\in Kv$.
For $n \in \nat$ and a vector $x = (x_1, \ldots, x_n) \in K^n$, we set 
$v(x)=\min\,\{v(x_1),\dots,v(x_n)\}$ and, if 
$x\in \mc{O}_v^n$, then we denote by $\ovl{x}^v$ the vector $(\ovl{x_1}^v, \ldots, \ovl{x_n}^v) \in (Kv)^n$.
To a $K$-subspace $W \subseteq K^n$, we associate the $Kv$-subspace
$ \ovl{W}^v = \lbrace \ovl{x}^v \mid x \in \mc{O}_v^n \cap W \rbrace \subseteq Kv^n$.

\begin{lem}
\label{L:LinearIndependence}
Let $n\in\nat$ and let $W$ be a $K$-subspace of $K^n$.
Then $\dim_{Kv}(\ovl{W}^v) = \dim_K(W)$.
In particular, for $k = \dim_K(W)$ and $x_1, \ldots, x_k \in \mc{O}_v^n\cap W$,
the family $(\ovl{x_1}^v, \ldots, \ovl{x_k}^v)$ is a $Kv$-basis of $\ovl{W}^v$ provided that it is $Kv$-linearly independent.
\end{lem}
\begin{proof}
This follows by combining the different parts of \autocite[Proposition 5]{MMW91}.
\end{proof}

For $n\in\nat$ and $f \in \mc{O}_v[X_1, \ldots, X_n]$, we denote by $\ovl{f}$ (or $\ovl{f}^v$) the residue polynomial of $f$ in $Kv[X_1, \ldots, X_n]$.

\begin{prop}\label{P:binaryFormValuations}
Let $(K, v)$ be a valued field, $d,n \in \nat$ and $f \in \mathcal{O}_v[X_1, \ldots, X_n]$ homogeneous of degree $d$.
Assume that $\ovl{f} \in Kv[X_1, \ldots, X_n]$ is anisotropic.
Then, for any $a_1, \ldots, a_n \in K$, we have 
\begin{displaymath}
v(f(a_1, \ldots, a_n)) = d\,\min \lbrace v(a_1),\dots,v(a_n) \rbrace \,.
\end{displaymath}
\end{prop}
\begin{proof}
Let $\gamma=\min\{v(a_1),\dots,v(a_n)\}$.
If $\gamma=\infty$, then $a_1=\ldots=a_n=0$ and $f(a_1,\dots,a_n)=0$, so both sides of the claimed equality are $\infty$.
Suppose now that $\gamma\neq \infty$. Set $\iota=\min\{i\in\{1,\dots,n\}\mid v(a_i)=\gamma\}$.
Then $v(a_\iota)=\gamma\neq\infty$, so $a_\iota\in\mg{K}$. 
Set $\lambda=a_\iota^{-1}$.
We have $\lambda a_1,\dots,\lambda a_n\in\mc{O}_v$ and $\lambda a_\iota=1$, so in particular $\ovl{\lambda a_\iota}\neq 0$ in $Kv$.
Since $\ovl{f}$ is anisotropic, it follows that $ \ovl{f}(\ovl{\lambda a_1},\dots,\ovl{\lambda a_n})\neq 0$.
Since $f(\lambda a_1,\dots,\lambda a_n)=\lambda^d f(a_1,\dots,a_n)$, we conclude that
$$0=v(f(\lambda a_1,\dots,\lambda a_n))=dv(\lambda) + v(f(a_1,\dots,a_n))\,.$$
As $-v(\lambda)=v(a_\iota)=\gamma$, we obtain the claimed equality.
\end{proof}

\begin{cor}\label{uniformapprox}
Let $w$ be a valuation on $K$.
Let $c\in\mg{\mc{O}}_w\cap\qse{\mc{O}}_w$ be such that $\pfi{\ovl c}{Kw}$ is anisotropic.
Let $f \in K$ and  $g \in \mg{K}$.
Then
$w(g)\leq \max\{w(f),0\}$ if and only if $\frac{f^2}{(1 - f - cf^2)g^2} \in \mathcal{O}_w$.
\end{cor}
\begin{proof}
By \Cref{P:binaryFormValuations}, the hypothesis implies that, for all $x, y \in K$, we have $w(x^2 - xy - cy^2) = 2\min \lbrace w(x), w(y) \rbrace$.
Hence
\begin{align*}
w\left(\mbox{$\frac{f^2}{(1- f -cf^2)g^2}$}\right) &= 2w(f) - 2w(g) - w(1 - f - cf^2) \\ &= 2w(f) - 2w(g) - \min \lbrace 0, 2w(f) \rbrace
= \max \lbrace 0, 2w(f) \rbrace - 2w(g)\,.
\end{align*}
This yields the claimed equivalence. 
\end{proof}

\begin{cor}\label{L:existsPfister}
Let $(K, v)$ be a $\zz$-valued field and $n \in \nat^+$.
Let $q \in\Quad_n(\mc{O}_v)$ be such that $\ovl{q} \in \Quad_n(Kv)$ is anisotropic over $Kv$.
Let $\pi\in K$ be such that $v(\pi)$ is odd.
Then the $2n$-dimensional quadratic form $q \perp \pi q$ over $K$ is anisotropic.
\end{cor}
\begin{proof}
This is immediate from \Cref{P:binaryFormValuations}.
\end{proof}

Next to the particularities of fields of characteristic $2$, we now also take the difference between real and nonreal fields into the picture.

The field $K$ is called \emph{nonreal} when $-1$ is a sum of squares in $K$ and \emph{real} otherwise. (Some authors use the term \emph{formally real} for the latter.)
By the Artin-Schreier Theorem (see e.g.~\cite[Corollary 6.1.6]{PfisterBook}), $K$ is real if and only if it admits a field ordering, i.e.~a total order relation $\leq$ such that, for any $a,b,c\in K$ with $a\leq b$, we have 
$a+c\leq b+c$ and, if $0\leq c$, then $ac\leq bc$.

A quadratic form over $K$ is called \emph{totally indefinite} if, for every field ordering $\leq$ of $K$, it represents positive as well as negative elements with respect to $\leq$.
In particular, if $K$ is nonreal, then every quadratic form over $K$ is totally indefinite.
Given any $a\in \qse{K}$, the $1$-fold Pfister form $\pfi{a}{K}$ represents $1$ and $-(1+4a)$, and hence, if $1+4a$ is a sum of squares in $K$, then $\pfi{a}{K}$ is totally indefinite.

For a field $K$, we will denote by $K\pow{2}$ the set of squares in $K$.

\begin{prop}\label{P:aniso-2-fold-via-residue}
Let $v$ be a $\zz$-valuation on $K$ such that $Kv$ is nonreal and has a quadratic field extension.
Then there exist $a\in \mg{K}$ and $b\in \qse{K}$ such that $\pfi{a,b}K$ is anisotropic and totally indefinite.
\end{prop}
\begin{proof}
Suppose first that $Kv$ has a separable quadratic field extension.
Since $Kv$ is nonreal, we find $b\in\mg{\mc{O}}_v$ such that the polynomial $T^2-T-\ovl{b}$ over $Kv$ is irreducible and $1+4b$ is a sum of squares in $K$.
We then take $a\in K$ with $v(a)=1$, to obtain by \Cref{L:existsPfister} that $\pfi{a,b}{K}$ is anisotropic. 
Furthermore, $\pfi{b}{K}$ is totally indefinite, and hence so is $\pfi{a,b}K$.

Suppose now that $Kv$ has an inseparable quadratic field extension. In particular there exists $a \in \mc{O}_v$ such that $\ovl{a}\in Kv\setminus Kv\pow{2}$. 
Fix $e\in K$ with $v(e)=-1$. We set $b=e$ if $K$ is nonreal and $b=e(e+\frac{1}4)$ otherwise.
Then $b\in \qse{K}$, $v(b)$ is an odd negative integer and $1+4b$ is a sum of squares in $K$.
Then every element of $K$ represented by $\pfi{b}K$ belongs to $(K\pow{2}\cup -bK\pow{2})(1+\mfm_v)$.
Thus $a$ is not represented by $\pfi{b}K$. Hence $\pfi{a,b}K$ is anisotropic and totally indefinite.
\end{proof}

We now collect some facts about quadratic forms over henselian valued fields.

Let $n\in\nat$. Since $\Quad_n(K)$ is a finite-dimensional $K$-vector space, we can endow it with the vector space topology induced by the $v$-adic topology on $K$.
If $(K,v)$ is henselian, then the following proposition shows that isotropy of regular quadratic forms is stable under small perturbances of the coefficients with respect to this topology.

\begin{prop}\label{P:IsotropyQuadraticFormOpen}
Let $(K, v)$ be a henselian valued field.
Let $q\in\Quad_n(K)$ be regular and isotropic.
There exists a neighbourhood of $q$ in $\Quad_n(K)$ in which all forms are isotropic.
\end{prop}
\begin{proof}
We fix $x \in K^n \setminus \lbrace 0 \rbrace$ with $q(x) = 0$.
Since $q$ is regular, there exists $y \in K^n$ with $\mf{b}_q(x, y) =1$.
We may rescale $x$ and $y$ to have additionally that $q(y) \in \mc{O}_v$.
Since addition and multiplication in $K$ are continuous, there exists a neighbourhood $U\subseteq \Quad_n(K)$ of $q$
such that $q'(x)\in\mfm_v$, $q'(y) \in \mc{O}_v$, and $\mfb_{q'}(x,y)\in 1+\mfm_v$ for all $q'\in U$.
We claim that any quadratic form in $U$ is isotropic.

To show this, consider $q'\in U$ and let $f(T) = T^2q'(y) + \mf{b}_{q'}(x, y)T + q'(x)$.
Since $\ovl{f}(T)\in Kv[T]$ has $0$ as a simple root and $(K,v)$ is henselian, 
$f(T)$ has a simple root in $K$. We conclude by \Cref{L:f(T)isotropy} that $q'$ is isotropic over $K$.
\end{proof}

\begin{prop}\label{P:anisotropicResidue}
Let $(K, v)$ be a henselian valued field, $n\in\nat$, and consider ${q \in \Quad_n(\mc{O}_v)}$ such that $\ovl{q}\in \Quad_n(Kv)$ is regular.
Then $q$ is anisotropic if and only if $\ovl{q}$ is anisotropic.
\end{prop}
\begin{proof}
If $\ovl{q}$ is anisotropic, then it follows by \Cref{P:binaryFormValuations} that $q$ is anisotropic.
(This implication does not require that $\ovl q$ is regular or that $v$ is henselian.)

For the converse implication, assume now that $\ovl{q}$ is regular and isotropic over $Kv$. 
Then there exist $x, y \in \mc{O}_v^n$ such that $\ovl{x} \neq 0$, $\ovl{q}(\ovl{x}) = 0$ and $\mf{b}_{\ovl{q}}(\ovl{x}, \ovl{y})\neq 0$.
Consider the quadratic polynomial
$$f(T) = q(x) + T\mf{b}_q(x, y) + T^2q(y) \in \mc{O}_v[T]\,.$$
The element $0$ in $Kv$ is a simple root of $\ovl f\in Kv[T]$.
Since $v$ is henselian, we conclude that $f$ has a simple root in $K$.
Hence $q$ is isotropic, by \Cref{L:f(T)isotropy}.
\end{proof}

Let $q$ be a quadratic form over $K$.
We return to the set $\Split(q)$ introduced in the previous section. 
Before looking (in the last part of this section) more closely at its interaction with valuations, we already mention here a crucial relation in the case of a henselian valuation, known as the \emph{Schwarz Inequality}.
We include a short argument involving the set $\Split(q)$; see~\autocite[Lemma 9]{SpringerTameQuadratic} for an alternative presentation.

\begin{prop}[Schwarz Inequality]\label{P:SI}
Let $(K,v)$ be a henselian valued field and $(V,q)$ an anisotropic quadratic form over $K$.
Then  $$\Split(q) \cap \mfm_v = \emptyset\,.$$
In other terms, for any $x,y\in V$ we have $$2v( \mf{b}_q(x, y))\geq v(q(x)) + v(q(y))\,.$$
\end{prop}
\begin{proof} Consider any $c \in \mfm_v$. 
As $(K,v)$ is henselian, it follows that $T^2 - T - c$ has a root in $K$, whereby $\pfi{c}K$ is isotropic, and as $q$ is anisotropic, it follows by \Cref{P:0-in-Split(q)} that $c\notin\Split(q)$.
This shows that $\Split(q) \cap \mfm_v = \emptyset$. From this together with the definition of $\Split(q)$, the second claim follows immediately.
\end{proof}

A crucial device in the study of quadratic forms in the presence of valuations is the concept of residue forms.
In residue characteristic different from $2$, this is quite simple, as quadratic forms can be diagonalised and the henselian condition can be applied just by lifting squares.
For our purposes, however, it is essential to cover the situation of characteristic $2$ as well.
Quite some work on residue forms in residue characteristic $2$ has been done, partially with restriction to the case of discrete (henselian) valuations.
In the sequel of this section, we focus on discrete valuations while avoiding any condition on the characteristics.
We follow the exposition in \autocite{MMW91}.

Let $(K, v)$ be a henselian $\zz$-valued field.
Let $n \in \nat^+$ and consider an anisotropic quadratic form $q \in \Quad_n(K)$.  
For $i \in \zz$, it follows by \Cref{P:SI} that 
$$ V_i = \lbrace x \in K^n \mid v(q(x)) \geq i \rbrace\,$$
is an $\mc{O}_v$-module, and the quotient module $V_{i}/V_{i+1}$ becomes naturally a $Kv$-vector space.
Fixing $\pi\in K$ with $v(\pi)=1$, we further obtain that 
\begin{align*}
&\RF{1}{q} : V_0/V_1 \to Kv : x \mapsto \ovl{q(x)} \\
&\RF{\pi}{q} : V_1/V_2 \to Kv : x \mapsto \ovl{\pi^{-1}q(x)}
\end{align*}
are anisotropic quadratic forms over $Kv$. They are called \emph{residue forms of $q$}.
Notice that, while one can vary the construction to remove the condition on the value group being $\zz$ (see e.g.~\autocite{SpringerTameQuadratic}), the construction of these residue forms is only applicable to an anisotropic quadratic form and for a henselian valuation. 

We will mainly need the following fact about these residue forms.
\begin{thm}[Mammone, Moresi, Wadsworth]\label{MMW}
Let $(K, v)$ be a henselian $\zz$-valued field, $\pi\in K$ with $v(\pi)=1$ and $q$ a non-degenerate, anisotropic quadratic form over $K$.
Then
$$ \dim(q) = \dim(\RF{1}{q}) + \dim(\RF{\pi}{q}).$$
\end{thm}
\begin{proof}
This follows from \autocite[Theorem 1]{MMW91}, using characterisation \eqref{it:ndg3} of being non-degenerate given in \Cref{P:non-degenerate-characterisations}.
\end{proof}

Using the construction of residue forms, we now obtain a sufficient condition for anisotropy of a quadratic form extended from a henselian discretely valued base field. It appears to be new in the case where the residue field has characteristic~$2$.

\begin{prop}\label{P:QF-over-unramified}
Let $(L,w)$ be a $\zz$-valued field. Let $K$ be a subfield of $L$ and $\pi\in K$ such that $w(\pi)=1$. 
Set $v=w|_K$.
Let $n \in \nat^+$ and let $q \in \Quad_n(K)$ be non-degenerate and such that $q_{K_v}$ is anisotropic. Assume that $\RF{1}{q}$ and $\RF{\pi}{q}$ are anisotropic over $Lw$.
Then $q_{L}$ is anisotropic.
\end{prop}

\begin{proof}
We may assume without loss of generality that $(L, w)$ and $(K, v)$ are henselian.
Let $V=K^n$.
For $i\in\zz$, $V_i$ is an $\mc O_v$-submodule of $V$.
Let $m = \dim(\RF{1}{q})$ and fix $x_1, \ldots, x_m \in V_0$ such that $(\ovl{x}_1, \ldots, \ovl{x}_m)$ is a $Kv$-basis of $V_0/V_1$.
Since $q$ is non-degenerate, it follows by \Cref{MMW} that $\dim(\RF{\pi}{q}) = n-m$, so we can further find $x_{m+1}, \ldots, x_n \in V_1$ such that $(\ovl{x}_{m+1}, \ldots, \ovl{x}_n)$ is a $Kv$-basis of $V_1/V_2$.
Set  $W=\bigoplus_{i=1}^mLx_i$ and $W'=\bigoplus_{i=m+1}^n Lx_i$.
Since $\RF{1}{q}_{Lw}$ and $\RF{\pi}{q}_{Lw}$ are anisotropic, it follows by \Cref{P:binaryFormValuations} that,  for any $l_1, \ldots, l_n \in L$, we have
$w(q(\sum_{i=1}^m l_i x_i)) = 2\min\{w(l_1),\dots,w(l_m)\}$ and $w(q(\sum_{i=m+1}^n l_i x_i)) = 1+2\min\{w(l_{m+1}),\dots,w(l_n)\}$.
Therefore $W\cap W'=0$, whereby $L^n=W\oplus W'$.

Now consider $x \in L^n \setminus \lbrace 0 \rbrace$ arbitrary. We write $x = y + z$ with $y \in W$ and $z \in W'$.
Let $l_1, \ldots, l_n \in L$ be such that $y = \sum_{i=1}^m l_i x_i$ and $z = \sum_{i=m+1}^n l_i x_i$.
Since $x\neq 0$, we obtain that $w(q(y))\neq w(q(z))$.
For $i\in\{1,\dots,m\}$ and $j\in\{m+1,\dots,n\}$, we have $v(\mf{b}_q(x_i,x_j)) >0$, by \Cref{P:SI}, and hence $v(\mf{b}_q(x_i,x_j)) \geq 1$.
Therefore
\begin{align*}
2w(\mf{b}_q(y, z)) &= 2w\left( \sum_{i=1}^m \sum_{j=m+1}^n l_il_j \mf{b}_q(x_i, x_j)\right) \\
&\geq  2\min\{w(l_1),\dots,w(l_m)\}+ 2\min\{w(l_{m+1}),\dots,w(l_n)\}+ 1\\
& =  w(q(y))+w(q(z)).
\end{align*}
Since $q(x) = q(y) + q(z) + \mf{b}_q(y, z)$ and $w(q(y)) \neq w(q(z))$, we conclude that $$w(q(x)) = \min \lbrace w(q(y)), w(q(z)) \rbrace \neq \infty\,.$$
This shows that $q_L$ is anisotropic.
\end{proof}

\begin{cor}\label{C:QF-over-unramified}
Let $L/K$ be a field extension and $w$ a $\zz$-valuation on $L$ such that $wK=\zz$ and $(K,w|_K)$ is henselian. 
If every anisotropic quadratic form over $Kw$ remains anisotropic over $Lw$, then every non-degenerate anisotropic quadratic form over $K$ remains anisotropic over $L$.
\end{cor}
\begin{proof}
Set $v=w|_K$ and fix $\pi\in K$ with $v(\pi)=1$.
Let $q$ be a non-degenerate anisotropic quadratic form over $K$.
Then $\RF{1}{q}$ and $\RF{\pi}{q}$ are anisotropic quadratic forms over $Kv$.
Now the statement follows by  \Cref{P:QF-over-unramified}.
\end{proof}

At the end of the previous section we defined $\Split_R(q)=\Split(q)\cap\mg{R}\cap \qse{R}$ for any subring $R\subseteq K$ with fraction field $K$ and a quadratic form $q$ over $K$.
Under certain conditions on $q$,  we can now give a positive answer to \Cref{Q:slot-ring} for the case where $R$ is a discrete valuation ring.
This is crucially based on \Cref{MMW}.
For a valuation $v$ on $K$ and a quadratic form $q$ over $K$, we abbreviate $$\Split_v(q)=\Split_{\mc{O}_v}(q)=\Split(q)\cap \mg{\mc{O}}_v\cap\qse{\mc{O}}_v\,.$$

\begin{prop}\label{L:splitoverinertial}
Let $(K,v)$ be a $\zz$-valued field and $q$ a non-degenerate quadratic form over $K$ with $\dim q\geq 3$.
If $\car(Kv)=2$ and $q_{K_v}$ is anisotropic, then assume further that $\dim q>2[Kv : (Kv)\pow{2}]$.
Then  $\Split_v(q)\neq \emptyset$.
\end{prop}

\begin{proof}
Let $n=\dim(q)$.
As $K$ is dense in $K_v$, and thus also $K^n$ is dense in $K_v^n$, we may replace $K$ by $K_v$ and assume without loss of generality that $(K, v)$ is henselian. 
Note that $\mg{\mc{O}}_v\cap \qse{\mc{O}}_v$ contains $1$ or $-1$, so it is non-empty.
If $q$ is isotropic, then by \Cref{P:0-in-Split(q)} we have $\Split(q) = \qse{K}\supseteq\qse{\mc{O}}_v$.
Hence, in order to show $\Split_v(q)\neq\emptyset$,
we may now assume that $q$ is anisotropic.

Fix $\pi\in K$ with $v(\pi)=1$.
We have $\dim q = \dim \RF{1}{q} + \dim \RF{\pi}{q}$, by \Cref{MMW}.
Since scaling $q$ does not affect the validity of the claim, we may replace $q$ by $\pi q$ if necessary to assume without loss of generality that $\dim \RF{1}{q}\geq \dim \RF{\pi}q$.
Then $\dim \RF{1}q\geq 2$ and, if $\car(Kv) = 2$, then the extra hypothesis yields that $\dim \RF{1}{q} > [Kv : (Kv)\pow{2}]$.
Since $\RF{1}{q}$ is anisotropic, we conclude in either case that $\mf{b}_{\RF{1}{q}}$ is not the zero map; see \autocite[Section 10]{ElmanKarpenkoMerkurjev}. 
By \Cref{P:0-in-Split(q)}, we obtain that $\Split(\RF{1}q)\neq \emptyset$. 
This readily implies that $\Split_v(q) \neq \emptyset$.
\end{proof}

The extra hypothesis in \Cref{L:splitoverinertial} in residue characteristic $2$ is necessary.

\begin{ex}\label{ex:splitoverinertialchartwo}
Let $(K, v)$ be a $\zz$-valued field with $\car(Kv) = 2$ and such that $Kv$ is not perfect, i.e.~$[Kv : (Kv)\pow{2}] > 1$.
Take $a \in \mc{O}_v$ with $\ovl{a} \in Kv \setminus (Kv)\pow{2}$, let $b \in K$ with $v(b) = -1$, and set $q = \pfi{a, b}K$.
We claim that $\Split_v(q)=\emptyset$.
To see this, consider $c\in\mg{\mc{O}}_v\cap \qse{\mc{O}}_v$.
Then $v$ extends to a $\zz$-valuation $v'$ on $\sep{K}{c}$. Note that $\sep{K}cv'/Kv$ is separable.
It follows that $\ovl{a}\in \sep{K}cv'\setminus(\sep{K}cv')\pow{2}$ and that $q_{\sep{K}c}$ is anisotropic.
By \Cref{P:0-in-Split(q)}, this yields that $c\notin\Split(q)$.
\end{ex}

We now give some first observations addressing \Cref{Q:slot-ring}, first for the valuation ring $\mc{O}_v$ in the current setting where $v$ is a $\zz$-valuation, and then for a finite intersection of such valuation rings.

\begin{cor}\label{C:splitoverinertial}
Let $q$ be a non-degenerate quadratic form over a $\zz$-valued field $(K, v)$ with $\dim(q) \geq 3$.
If $ \Split(q) \cap \mc{O}_v\neq \emptyset$, then  $\Split_v(q)\neq\emptyset$.
\end{cor}
\begin{proof}
If $v$ is non-dyadic or $q$ is isotropic over $K_v$, then it follows by \Cref{L:splitoverinertial} that $ \Split_v(q)\neq \emptyset$.
If $v$ is dyadic and $q_{K_v}$ is anisotropic, then using \Cref{P:SI} we obtain that $\Split(q)\cap \mc{O}_v=\Split(q)\cap\mg{\mc{O}}_v\cap\qse{\mc{O}}_v=\Split_v(q)$.
\end{proof}

\begin{cor}\label{C:split-over-multi-inertial}
Let $q$ be a non-degenerate quadratic form over a field $K$ with $\dim(q) \geq 3$.
Let $m \in \nat$ and $v_1, \ldots, v_m$ be $\zz$-valuations on $K$.
If $ \Split(q) \cap \mc{O}_{v_i}\neq \emptyset$ for $1\leq i\leq m$, then  $\Split_R(q)\neq\emptyset$ for $R = \bigcap_{i=1}^m \mc{O}_{v_i}$.
\end{cor}
\begin{proof}
We may assume $v_1,\dots,v_m$ to be distinct. 
Let $d = \dim(q)$. By fixing a basis of the $K$-vector space on which $q$ is given, we may assume that $q \in \Quad_d(K)$.
By \Cref{C:splitoverinertial}, for $1\leq i\leq m$, we have $\Split_{v_i}(q)\neq \emptyset$, so there exist $x_i, y_i \in K^d$ with $\mf{b}_q(x_i,y_i)=1$ and $-q(x_i)q(y_i) \in \mg{\mc{O}}_{v_i}\cap\qse{\mc{O}}_{v_i}$.
By Weak Approximation, we may find $x, y \in K^d$ such that, for $1\leq i\leq m$, we have $\mf{b}_q(x,y)\in 1+\mfm_{v_i}$ and 
$q(x)q(y)\in q(x_i)q(y_i)+\mfm_{v_i}$.
Then $-q(x)q(y)/\mf{b}_q(x, y)^2 \in \bigcap_{i=1}^m \Split_{v_i}(q)= \Split_R(q)$.
\end{proof}

In \Cref{Pfisterstronggev} below we will prove a variation of the above result for function fields in one variable, which applies also to certain infinite intersections of valuation rings containing the constant field.

We spell out the following consequence for presentations of quadratic Pfister forms. A refinement of this is obtained in \cite[Lemma 2.5]{DD}.
\begin{cor}\label{C:presentation-Pfister-semiglobal}
Let $K$ be a field and $S$ be a finite set of $\zz$-valuations on $K$.
Let $n \geq 2$ be such that for all dyadic valuations $v \in S$ we have $2^{n-1} > [Kv : (Kv)\pow{2}]$.
Then every $n$-fold Pfister form over $K$ is isometric to $\pfi{a_1, \dotsc, a_{n-1}, c}K$ for some $a_1, \dotsc, a_{n-1} \in K^\times$
and $c \in K$ with $v(c) = v(1 + 4c) = 0$ for all $v \in S$.
\end{cor}
\begin{proof}
This follows directly by combining \Cref{P:Split(q)=rightslot}, \Cref{L:splitoverinertial}, and \Cref{C:split-over-multi-inertial}.
\end{proof}

\section{Quadratic splitting sets} 

Let $K$ be a field.
For a $K$-variety $\mathbb V$ and $d\in\nat$, one can consider the set of monic degree $d$ polynomials $f\in K[T]$ such that $\mathbb{V}(K[T]/(f))\neq \emptyset$. 
One can build an existential $\Lar(K)$-formula $\phi$ in $d$ free variables such that, for any field extension $K'/K$ and any $c_1,\dots,c_d\in K'$, one has 
\begin{center}
$\mathbb{V}(K'[T]/(T^d+\sum_{i=1}^d c_iT^{d-i}))\neq \emptyset$ if and only if $K'\models \phi(c_1,\dots,c_d)$.
\end{center}
In that case, one obtains from $\phi$ for any $(c_2,\dots,c_d)\in K^{d-1}$ an existentially definable subset 
$\phi(K,c_2,\dots,c_d)$ of $K$.
This process has been successfully applied for some special types of varieties to obtain existential definitions of particular subrings of $K$.
Of course, to obtain a proper subset of $K$ in this way, $\mathbb{V}$ should not have $K$-points.
Experience suggests to take for $\mathbb{V}$ a type of variety with particularly nice properties, for example geometrically rational varieties.
The easiest type is that of smooth projective quadrics in $\mathbb{P}^{n-1}_K$, given by a non-degenerate $n$-dimensional quadratic form $q$ over $K$, for some integer $n\geq 2$. For $n=2$, the quadric is not geometrically irreducible, which might explain some anomalies in this case.
Already the choice $(d,n)=(2,3)$ has been applied very successfully in the literature.

We will now take a rather ad hoc approach to the case of certain projective quadrics and where $d=2$, as well as to the question of having points over certain quadratic extensions of $K$.
The preceding outline in geometric terms only serves for the intuition and motivation of certain choices.
In the sequel, our exposition will be in the terms of quadratic form theory.

Hence, we will study certain subsets of fields determined by quadratic equations. 
These subsets are described by the relation between a particular quadratic form $q$ defined over the base field $K$ and the quadratic field extensions $L/K$ for which $q_L$ is isotropic.
The set of parameters defining $q$ and $L/K$ will later be allowed to vary, and for certain fields $K$ and certain subrings $R$ of $K$, one can obtain in this way a first-order definition of $R$ inside $K$.
This approach turns out to be particularly successful in situations where the quadratic forms to which it is applied satisfy a local-global principle over all quadratic extensions of $K$ and where $R$ is given by an intersection of valuation rings.
The history of this method in the context of first-order definability in arithmetic goes back at least to J.~Robinson's first-order definition of $\zz$ in $\qq$ from \autocite{Rob49}.
Before looking at situations where our quadratic form related sets become particularly useful, we study them in a general setup.
\medskip

Let $K$ be a field.
Consider a quadratic form $q$ over $K$ and an element $c\in K$.
Based on the set $\Split(q)$ studied in the preceding two sections, we define
\begin{displaymath}
\qss{q}c = \lbrace e \in K \mid \exists\, d\in\Split(q): (1-e)^2d^2=cd\mbox{ or } 4(c-d)=1=e \rbrace\,.
\end{displaymath}
By its definition, the set $\qss{q}c$ is existentially definable in $K$.
We refine this observation for the case where $q$ is a Pfister form.
\begin{prop}\label{P:Pfister-qss-uniformdef}
Let $n\in\nat^+$. 
There exists an existential $\Lar$-formula $\varphi$ in $n+2$ free variables such that, for every field $K$ and 
any $a_1,\dots,a_{n}, c \in K$, we have 
$$\varphi(K,a_1,\dots,a_{n},c)=\qss{\pfi{a_1,\dots,a_{n}}{K}}c\,.$$
\end{prop}
\begin{proof}
This is obvious from \Cref{EX:Pfi-Split} and the definition of the sets $\qss{q}c$.
\end{proof}

Our attention to the set $\qss{q}c$ is motivated by the following characterisation.

\begin{prop}\label{P:QS-characterisation}
Let $c,e\in K$ with $2c\neq 0$ or $e\neq 1$. 
Let $q$ be a regular quadratic form over $K$.
Set $\Theta=T^2-(1-e)T-c\in K[T]$ and let $L$ be the splitting field of $\Theta$ over $K$.
Then $e\in\qss{q}c$ if and only if $q_L$ is isotropic.
\end{prop}
\begin{proof}
By \Cref{P:0-in-Split(q)}, $q$ is isotropic if and only if $0\in\Split(q)$.
Hence, if $q$ is isotropic, then so is $q_L$, and we have $\qss{q}c = K$, whereby $e\in\qss{q}c$.
Assume now that $q$ is anisotropic. Then $q\not\simeq\hh_K$ and $0 \not\in \Split(q)$.

If $e\neq 1$, then set $d^\ast=(1-e)^{-2}c$. If $e=1$, then $2c\neq 0$, and we set $d^\ast=c-\frac{1}4$.
In either case, we obtain that $L\simeq_K\sep{K}{d^\ast}$.
Furthermore, if $d^\ast\notin\qse{K}$, then $L=K$, whereby $q_L$ is anisotropic.
Using this, we conclude by \Cref{P:0-in-Split(q)} that $q_L$ is isotropic if and only if  $d^\ast\in\Split(q)$.

Hence, if $q_L$ is isotropic, then it follows that $e\in \qss{q}c$.
Assume conversely that $e\in\qss{q}c$. 
So there exists $d\in\Split(q)$ such that $(1-e)^2d^2=cd$ or $4(c-d)=1=e$. 
As $0\notin\Split(q)$, we conclude that $d=d^\ast$, whereby $d^\ast\in\Split(q)$ and $q_L$ is isotropic.
\end{proof}

In view of the characterisation in \Cref{P:QS-characterisation} we refer to the sets $\qss{q}c$, where $q$ is a quadratic form over $K$ and $c\in\mg{K}$, as \emph{quadratic splitting sets}, motivating also the choice of the notation.

\begin{rem}
In \autocite{DittmannDefiningSubrings}, for a quadratic form $q$ over $K$ and $c\in K$, the notation $S_c(q)$ was taken for the subset of $K$ equal to $\qss{q}{-c}$ in the current notation. 
Many variations of the sets $\qss{q}{c}$ were used before to obtain definability results.
Because of their nice interaction with local-global principles for quadratic forms (see \Cref{localglobalaxiom-c} below), they have proven especially useful in defining subsets of a field which are characterised by local conditions.

A version of the set $\qss{q}{c}$ appears in the proof of \autocite[Theorem 3.1]{IntegralityAtPrime}.
It comes further to the fore in \autocite{Poo09} in the context of defining valuation rings and rings of $S$-integers in global fields.
The quadratic forms $q$ considered there are the norm forms of quaternion algebras, that is, $2$-fold Pfister forms.
The special quadratic splitting sets from \autocite{Poo09} are used in Koenigsmann's universal definition of $\zz$ in $\qq$ \autocite{Koe16} and in the subsequent generalisations of this result to other global fields \autocite{Par13, Eis18, DaansGlobal}.
In \autocite[Chapters 2 and 3]{Dittmann_Thesis}, \autocite{Dit17} and \autocite{DittmannDefiningSubrings},
the splitting sets (denoted $S_{c}(q)$ there) are considered for a general Pfister form $q$, along with variants where $q$ is replaced by a central simple algebra or a symbol in $K_3(K)/n$.

In all those applications, the presence of a local-global principle is crucial.
The usefulness of local--global principles in proving definability statements on its own was discovered long before, see e.g.~\autocite{Rumely, Pop_ElemEquivVsIsomII}.
\end{rem}

We collect some basic properties of quadratic splitting sets.

\begin{prop}\label{Sbasic-c}
Let $q$ be a quadratic form over $K$ and $c \in {K}$.
\begin{enumerate}[$(a)$]
\item If $q$ is regular and isotropic, then $\qss{q}c = K$.
\item $\qss{q}c  \subseteq \qss{q_L}c $ for any field extension $L/K$.
\item If $q$ is anisotropic over $K$ and $e \in \qss{q}c $, then either $T^2 - (1-e)T - c$ is an irreducible, separable polynomial over $K$, or $1-e=2c=0$.
\item If $q$ is regular, then $q_{\sep{K}{c}}$ is isotropic if and only if $c\in\Split(q)$ or $q$ is isotropic,  if and only if $0 \in \qss{q}c$.
\end{enumerate}
\end{prop}
\begin{proof}
These statements follow from the definition of $\qss{q}c$ together with \Cref{P:0-in-Split(q)} 
 and \Cref{P:QS-characterisation}.
\end{proof}

One method involved in the study of quadratic splitting sets is to pass to certain field extensions of $K$ where quadratic forms are easier to handle. A suitable setting for this will be that of a local-global principle for isotropy of quadratic forms.
We indicate how such a local-global principle can help to reduce the study of quadratic splitting sets to the corresponding local cases.

\begin{prop}\label{localglobalaxiom-c}
Let $q$ be a regular quadratic form over $K$ and $c \in K$. 
Let $\mathbb{E}$ be a set of field extensions of $K$ such that, for every separable field extension $L/K$ with $[L:K]\leq 2$ and such that $q_L$ is anisotropic, there exists a field $E\in\mathbb{E}$ such that $L\subseteq E$ and $q_{E}$ is anisotropic. Then
$\qss{q}{c} = K \cap \bigcap_{E \in \mathbb{E}} \qss{q_E}c$.
\end{prop}
\begin{proof}
The left-to-right inclusion holds by \Cref{Sbasic-c} $(b)$. 
For the right-to-left inclusion, consider an element $e\in K\setminus \qss{q}c$.
We need to show that there exists a field $E\in\mathbb{E}$ such that $e\notin\qss{q_E}{c}$.
As $\qss{q}c\neq K$, we have that $q$ is anisotropic.
First assume that $e=1$ and $2c=0$.
Then we pick any $E\in\mathbb{E}$ such that $q_E$ is anisotropic (by applying the hypothesis on $\mathbb{E}$ for $L=K$), so that $0\notin\Split(q_E)$.
If $c \neq 0$ (so necessarily $\car(K) = 2$), then we immediately obtain $e=1\notin\qss{q_E}c$ from the definition of $\qss{q_E}c$.
If $c=0$, then $e \in K \setminus \qss{q}c$ forces $\Split(q) = \emptyset$,
which by the criterion of \Cref{P:0-in-Split(q)} $(c)$ forces $\Split(q_E) = \emptyset$ and hence $e=1\not\in \qss{q_E}c$, as desired.

We may now assume that $e\neq 1$ or $2c\neq 0$.
Let $L$ denote the splitting field of the polynomial $T^2 - (1-e)T - c$ over $K$.
It follows by \Cref{P:QS-characterisation} that $q_L$ is anisotropic.
Since $[L:K]\leq 2$, we obtain from the hypothesis that there exists some field $E\in\mathbb{E}$ 
with $L\subseteq E$ and such that $q_E$ is anisotropic.
Then $T^2 - (1-e)T - c$ splits over $E$, and we conclude by \Cref{Sbasic-c}~$(c)$ that $e\notin\qss{q_E}{c}$.
\end{proof}

The local-global principles for quadratic forms which we will use are formulated in terms of completions or henselisations of a given field with respect to certain valuations.
We therefore need to relate valuations to quadratic splitting sets. 
We first make an observation on quadratic polynomials over henselian valued fields.

\begin{lem}\label{L:henselianKrasner}
Let $(K, v)$ be a henselian valued field. Let $a, b \in K$. Consider the polynomial $\Psi_{a, b} = T^2 - aT + b\in K[T]$.
\begin{enumerate}[$(a)$]
\item If $a, b \in K$ are such that $v(b) > 2v(a)$, then $\Psi_{a, b}$ splits over $K$.
\item If $a, b, a', b' \in K$ are such that $\min \lbrace v(a - a') + \frac{1}{2}v(b), v(b - b') \rbrace > v(a^2 - 4b)$, then $\Psi_{a, b}$ has a root in the splitting field of $\Psi_{a', b'}$. In particular, if $\Psi_{a, b}$ does not have a root in $K$, then neither does $\Psi_{a', b'}$, and the splitting fields of $\Psi_{a, b}$ and $\Psi_{a', b'}$ over $K$ are isomorphic.
\end{enumerate}
\end{lem}
\begin{proof}
$(a)$\, Let $a,b\in K$ be such  that $v(b) > 2v(a)$. In particular $a\in\mg{K}$, and we obtain that $a^{-2}\Psi_{a,b}(aT)=T^2 - T + a^{-2}b=\Psi_{1,a^{-2}b}$.
As $(K,v)$ is henselian and $v(a^{-2}b)>0$, it follows that $\Psi_{1,a^{-2}b}$ splits over $K$, and hence so does $\Psi_{a,b}$.

$(b)$\, Let $a, b, a', b' \in K$ be such that $\min \lbrace v(a - a') + \frac{1}{2}v(b), v(b - b') \rbrace > v(a^2 - 4b)$. 
In particular $a^2 - 4b \in\mg K$, whereby $\Psi_{a, b}$ is separable.
If $\Psi_{a,b}$ is split over $K$, then the claim is trivially satisfied, so we may assume that $\Psi_{a,b}$ is irreducible in $K[T]$.
Let $M$ be the splitting field of $\Psi_{a, b}\cdot \Psi_{a', b'}$ over $K$.
We extend $v$ to a valuation $w$ on $M$. 
We write $\Psi_{a, b} = (T - x_1)(T - x_2)$ and $\Psi_{a', b'} = (T - y_1)(T - y_2)$ with $x_1, x_2, y_1, y_2 \in M$.
Note that $a^2-4b=(x_1+x_2)^2 - 4x_1x_2=(x_1-x_2)^2$.
Hence the hypothesis yields that 
$\min \lbrace v(a - a') + \frac{1}{2}v(b), v(b - b')\}>2w(x_1-x_2)$.
Since $(K,v)$ is henselian and $\Psi_{a,b}$ is irreducible over $K$, we have that $w(x_1)=w(x_2)$ and thus $w(b) = w(x_1x_2) =  2w(x_1)$.
Using this together with the hypothesis, we obtain that
\begin{eqnarray*}
 w(x_1 - y_1) + w(x_1 - y_2) & = & w((x_1 - y_1)(x_1 - y_2))  = w(\Psi_{a', b'}(x_1)) \\ &= & w(\Psi_{a', b'}(x_1) - \Psi_{a, b}(x_1)) = w((a - a')x_1 + (b'-b)) \\ &\geq &\min\lbrace v(a -a') + \hbox{$\frac{1}{2}$}w(b), v(b-b')\rbrace \\ & >  &    2w(x_1-x_2)\,.
\end{eqnarray*}
Thus we may assume without loss of generality  that $w(x_1 -y_1) > w(x_1 - x_2)$.
By Krasner's Lemma \autocite[Theorem 4.1.7]{Eng05}, since $\Psi_{a, b}$ is separable, this implies that $x_1 \in K(y_1)$. Hence $K(y_1)$ is the splitting field of $\Psi_{a,b}$ and of $\Psi_{a', b'}$ over $K$.
\end{proof}

\begin{prop}\label{Elocallem0}
Let $(K, v)$ be a henselian valued field. Let $q$ be a regular quadratic form defined over $K$ and $c \in \mg{K}$. 
\begin{enumerate}[$(a)$]
\item\label{it:Scqnecessary} If $q$ is anisotropic, then $2v(1-x) \geq v(c)$ for every $x \in \qss{q}{c}$.
\item\label{it:ScqinOvd} If $q$ is anisotropic and $0\in \qss{q}{c}$, then $\qss{q}{c} \subseteq \mc{O}_vc$.
\item\label{it:Scqsufficient} If $0\in \qss{q}{c}$, then $x \in \qss{q}{c}$ for every $x\in K$ with $v(x) > v(1+4c) - \frac{1}{2}v(c)$.
\end{enumerate}
\end{prop}
\begin{proof}
$(\ref{it:Scqnecessary})$
Assume that $q$ is anisotropic and let $x \in \qss{q}{c}$. 
If $x=1$, then trivially $v(c)\leq 2v(1-x)$. Assume now that $x\neq 1$.
Then $T^2 - (1-x)T - c$ is irreducible in $K[T]$, by \Cref{Sbasic-c} $(c)$, and we conclude by \Cref{L:henselianKrasner} that $v(c)\leq 2v(1-x)$.

$(\ref{it:ScqinOvd})$
Assume that $q$ is anisotropic and $0 \in \qss{q}{c}$. 
Then  $v(c) \leq 2v(1) = 0$, by $(\ref{it:Scqnecessary})$. In particular $c \in\mg K$.
For $x\in\qss{q}c$, we have $2v(1-x) \geq v(c)$, by $(\ref{it:Scqnecessary})$, and as $v(c)\leq 0$, we obtain that $v(x) \geq v(c)$.
This shows that $\qss{q}{c} \subseteq \mc{O}_vc$.

$(\ref{it:Scqsufficient})$ 
Consider $x\in K$ with $v(x) > v(1+4c) - \frac{1}{2}v(c)$. 
Let $L$ denote the splitting field of $T^2-(1-x)T-c$ over $K$.
It follows by \Cref{L:henselianKrasner}~$(b)$ that 
 $\sep{K}c\subseteq L$.
Assume now that $0\in\qss{q}c$. Then $q_{\sep{K}c}$ is isotropic, by \Cref{P:QS-characterisation}.
Hence $q_L$ is isotropic,
and if $x \neq 1$ or $\car(K) \neq 2$ we may conclude by \Cref{P:QS-characterisation} that $x \in \qss{q}{c}$.
If $x=1$ and $\car(K)=2$, then the hypothesis forces $v(c) > 0$,
and so $\sep{K}c = K$, $q$ is isotropic, and $\qss{q}{c} = K$ by \Cref{Sbasic-c}.
\end{proof}

\section{Function fields in one variable} 
\label{S:ffiov}

The main results of the present article will be about subsets, in particular subrings, of a function field in one variable.
In the first part of this section we assemble some basic results in this setup.
Recall that by a \emph{function field in one variable} we mean a finitely generated field extension of transcendence degree $1$.

In the sequel of this section, let $K$ be a field and $F/K$ a function field in one variable.
We denote by $\mc{V}(F/K)$ the set of $\zz$-valuations on $F$ which are trivial on $K$.
Every nontrivial valuation on $F$ which is trivial on $K$ is equivalent to a unique valuation in $\mc{V}(F/K)$.
While the set $\mc{V}(F/K)$ is always infinite, for any $a \in\mg{F}$, the subset $\{v \in \mc{V}(F/K)\mid v(a)\neq 0\}$
is finite; see e.g.~\autocite[Corollary 1.3.4]{Sti09}.

We now collect some standard tools for $F/K$ using the valuations in $\Val(F/K)$.
We will often use a well-known consequence of the Riemann-Roch Theorem:

\begin{thm}[Strong Approximation]\label{T:SAT}
Let $n \in \nat$ and $v_1, \ldots, v_n \in \Val(F/K)$ distinct.
Let $S \subsetneq \mc{V}(F/K)\setminus\{v_1,\dots,v_n\}$. 
Then for any $a_1, \ldots, a_n \in F$ and $\gamma_1, \ldots, \gamma_n \in \zz$, there exists an element $x \in F$ such that $v_i(x - a_i) > \gamma_i$ for all $1 \leq i \leq n$ and $v(x) \geq 0$ for all $v \in S$.
\end{thm}
\begin{proof}
See \autocite[Theorem 1.6.5]{Sti09}.
\end{proof}

For a proper subset $S\subsetneq\Val(F/K)$, we define $$\mc{O}(S)=\bigcap\limits_{v\in S} \mc{O}_v.$$

\begin{cor}\label{C:holo-valsets}
Let $S\subsetneq\Val(F/K)$ and $v\in \Val(F/K)\setminus S$.
Then there exists an element $x\in \mc{O}(S)$ with $v(x)<0$.
\end{cor}
\begin{proof}
Fix some $v'\in \Val(F/K)\setminus\{v\}$ and set $S'=\Val(F/K)\setminus \{v,v'\}$.
By \Cref{T:SAT} there exists a nonzero element $x\in\mc{O}(S')$ with $v'(x)>0$.
Then $x$ is not algebraic over $K$ and hence $w(x)<0$ for some $w\in\Val(F/K)$.
Since $w(x)\geq 0$ for all $w\in S'\cup\{v'\}$, we conclude that $w=v$.
Hence $v(x)<0$, and since $S\subseteq S' \cup \lbrace v' \rbrace$, we have $x\in\mc{O}(S' \cup \{ v' \})\subseteq\mc{O}(S)$.
\end{proof}

\Cref{C:holo-valsets} shows that $\mc{O}(S)\neq\mc{O}(S')$ for any $S,S'\subsetneq \Val(F/K)$ with $S\neq S'$.
Note further that $\mc{O}(\emptyset)=F$.

An integrally closed subring of $F$ containing $K$ and with fraction field $F$ is called a \emph{holomorphy ring of $F/K$}.

\begin{prop}\label{P:holo}
Let $R\subseteq F$ and  $S=\{v\in\Val(F/K)\mid R\subseteq \mc{O}_v\}$.
Then $R$ is a holomorphy ring of $F/K$ if and only if $S\neq \Val(F/K)$ and $R=\mc{O}(S)$.
\end{prop}
\begin{proof}
See  
\cite[Theorem 3.1.3]{Eng05}.
\end{proof}

A holomorphy ring of  $F/K$ will be called \emph{finitary} if it is finitely generated as a $K$-algebra.

\begin{prop}\label{P:fingenhol}
For $R\subseteq F$ the following are equivalent:
\begin{enumerate}[$(i)$]
\item $R$ is a finitary holomorphy ring of $F/K$.
\item $R$ is the integral closure of $K[x]$ for some $x\in F$ transcendental over $K$.
\item $R=\mc{O}(S)$ for some  $S\subsetneq \Val(F/K)$ where $\Val(F/K)\setminus S$ is finite.
\end{enumerate}
\end{prop}
\begin{proof}
$(i\Rightarrow iii)$\, 
Let $S=\{v\in\Val(F/K)\mid R\subseteq \mc{O}_v\}$.
Assume that $R$ is finitely generated as a $K$-algebra. Then $\Val(F/K)\setminus S$ is finite.
Assume further that $R$ is a holomorphy ring of $F/K$. 
Then $S\subsetneq\Val(F/K)$ because $F$ is the fraction field of $R$. Furthermore
$R=\mc{O}(S)$, by \Cref{P:holo}.

$(iii\Rightarrow ii)$\,
Let $R=\mc{O}(S)$ for a set $S\subsetneq\Val(F/K)$ such that the complement $\ovl{S}=\Val(F/K)\setminus S$ is finite.
For each $w\in\ovl S$, by \Cref{C:holo-valsets} we choose an element $x_w \in F$ such that 
$\{v\in\Val(F/K)\mid v(x_w)<0\}=\{w\}$.
We set $x=\sum_{v\in\ovl S}x_v$. Then $\{v\in\Val(F/K)\mid v(x)<0\}=\ovl S$.
It follows that $R$ is the intersection of the valuation rings of $F$ containing $K[x]$. 
This is the integral closure of $K[x]$.
Furthermore, $x$ is transcendental over $K$ because $S\neq\Val(F/K)$.

$(ii\Rightarrow i)$\,
For $x\in F$ transcendental over $K$, the integral closure of $K[x]$ in $F$ is finitely generated as a $K[x]$-module \cite[Corollary 13.13]{Eis04}.
Hence this implication is obvious.
\end{proof}

\begin{lem}\label{DeltaqnotF}
Let $b \in F$. There exist infinitely many $v \in \mc{V}(F/K)$ such that $X^2 - X - b$ has a root in $F_v$.
\end{lem}
\begin{proof}
We may replace $K$ by its relative algebraic closure in $F$ and assume that $K$ is relatively algebraically closed in $F$.
It now suffices to consider the case where $F/K$ is a rational function field: in the general case $F$ is a finite extension of a field $E$ with $b\in E$ and such that $E/K$ is a rational function field, and every $\zz$-valuation $w\in \Val(E/K)$ extends to a discrete valuation on $F$, which is then equivalent to a $\zz$-valuation $v\in\Val(F/K)$, and then $E_w\subseteq F_v$.
Assume therefore that $F/K$ is a rational function field.

Consider first the case where $b\in K$ and $F=K(T)$ for some $T\in F\setminus K$.
Any monic irreducible polynomial $p\in K[T]$ induces a $p$-adic valuation $v_p\in\mc{V}(F/K)$, and if there exists further some $K$-embedding $\sep{K}{b}\to K[T]/(p)$, then $X^2-X-b$ has a root in $F_v$.
Since there are infinitely many such polynomials $p\in K[T]$, we are done for this case.

Assume now that $b\in F\setminus K$. Then $b$ is transcendental over $K$ and we restrict to the case where $F=K(b)$. It suffices now to show that there exist infinitely many monic irreducible polynomials $p\in K[b]$ such that $b\equiv c^2-c\bmod p$ for some $c\in K[b]$: for any such polynomial $p$ with the exception of $p = b + \frac{1}{4}$ when $\car(K) \neq 2$, Hensel's lemma yields that $X^2 - X - b$ has a root in $F_{v_p}$ where $v_p$ is the $\zz$-valuation on $F$ corresponding to $p$.
If $K$ is infinite, we simply take the linear polynomials $p=b-c^2+c$ where $c$ varies over $K$.
Assume now that $K$ is finite. Set $q=|K|$. For $r\in\nat^+$, let $f_r=b^{2q^r+1}-b^{q^r}-1\in K[b]$, and observe that $f_r$ is separable and $b\equiv c^2-c\bmod f_r$ for $c=b^{q^r+1}$. 
There exist infinitely many monic irreducible polynomials $p$ dividing a polynomial $f_r$ for some $r\in\nat^+$, and they all have the desired property.
\end{proof}

\begin{prop}\label{P:HolomorphyToValuation}
Let
$\Delta \subseteq \mc{V}(F/K)$ and $R=\mc{O}(\Delta)$.
Let $c \in \mg{R}\cap\qse{R}$ be such that $T^2 - T - c$ is irreducible over $F_w$ for every $w \in \Delta$.
Then, for every $v \in \Delta$, there exists an element $g\in\mg{F}$ such that
$$
\mc{O}_v \,=\, \left\{f\in F\,\,\left|\,\, \mbox{$\frac{f^2}{1-fg-c(fg)^2}$}\right. \in R\right\}\,\,\mbox{ and }\,\,\,
\mfm_v \,=\, \left\{f\in F\,\,\left|\,\, \mbox{$\frac{f^2}{(1-f-cf^2)g^2}$}\right. \in R\right\}.
$$
In particular, if $R$ is $\exists$-$\Lar(F)$-definable, then so are $\mc{O}_v$ and $\mfm_v$ for every $v \in \Delta$.
\end{prop}
\begin{proof}
By \Cref{DeltaqnotF},  $\mc{V}(F/K) \setminus \Delta$ is infinite.
Fix $v \in \Delta$ arbitrary.
By the Strong Approximation \Cref{T:SAT}, we can find $g \in F^\times$ such that $v(g) = 1$ and $w(g) \leq 0$
for all $w \in \Delta \setminus \lbrace v \rbrace$.
For every $w\in\Delta$, the hypothesis implies that $\pfi{c}{F_w}$ is anisotropic, whence $\pfi{\ovl{c}^w}{Fw}$ is anisotropic, by \Cref{P:anisotropicResidue}.
It follows by \Cref{uniformapprox} that
$\mfm_v=\mc{O}_vg=\{f\in F\,\mid\, f^2(1-f-cf^2)^{-1}g^{-2}\in R\}$.
Since $\mc{O}_v=\{f\in F\mid fg\in\mfm_v\}$, the statement follows.
\end{proof}
We will apply \Cref{P:HolomorphyToValuation} to sets $\Delta$ given in a natural way by a quadratic form, as explained below.
We will subsequently explain how to find the corresponding element $c$, and put things together in \Cref{C:HolomorphyToValuation}.

For a quadratic form $q$ over $F$ 
we set $$\Delta_K q=\{v \in \mc{V}(F/K)\mid q_{F_v}\mbox{ anisotropic}\}\,.$$
The following statement considers the corresponding subring $R=\mc{O}(\Delta_K q)$, and it reduces the existential definability of $R$ in $F$ to the existential definability in $F$ of some subset $C\subseteq K$ which satisfies $R=C\cdot \qss{q}{d}$ for some $d\in\Split(q)\cap\mg{R}$.

\begin{prop}\label{P:BigcapDeltaQ}
Let $q$ be a regular quadratic form over $F$. Set $R=\mc{O}(\Delta_K q)$. Let $C \subseteq K$.
Then $$C\cdot\big( \bigcup_{c\in \Split(q)\cap\mg{F}} c^{-1}\qss{q}c\big)\,\,\subseteq\,\, R\,,$$
 and this inclusion is an equality if $R=C\cdot \qss{q}d$ holds for some $d\in\Split(q) \cap R^\times$.
\end{prop}
\begin{proof}
For $c \in \Split(q)\cap\mg{F}$ and $v \in \Delta_K q$, we have $\qss{q}c \subseteq \mc{O}_vc$, by  \Cref{Elocallem0}~$(b)$. 
Hence $c^{-1}\qss{q}c\subseteq R$ for all $c\in\Split(q)\cap\mg{F}$.
Since $C\subseteq K\subseteq R$, we obtain that $C\cdot\big( \bigcup_{c\in \Split(q)\cap\mg{F}} c^{-1}\qss{q}c\big)\subseteq R$.

Assume now that there exists $d\in \Split(q) \cap R^\times$ such that $R=C\cdot \qss{q}d$. Then
$dR=R=C\cdot \qss{q}d$, whence $R=C\cdot d^{-1}\qss{q}d\subseteq C\cdot\big( \bigcup_{c\in \Split(q)\cap\mg{F}} c^{-1}\qss{q}c\big)$.
\end{proof}

\begin{prop}\label{P:AE-criterion}
Let $n\in\nat$ and let $\psi$ and $\rho$ be two $\exists$-$\Lar(F)$-formulas in $n$ and $n+1$ free variables 
such that the following hold:
\begin{enumerate}[$(i)$]
\item\label{it:AE-1} For any $(a_1,\dots,a_n)\in\psi(F^n)$, $\rho(F,a_1,\dots,a_n)$ is a $K$-subalgebra which is integrally closed in $F$.
\item\label{it:AE-2} For any finite subset $S\subseteq\Val(F/K)$ and any $w\in\Val(F/K)\setminus S$, there exists $(a_1,\dots,a_n)\in\psi(F^n)$ such that $\rho(F,a_1,\dots,a_n)\subseteq \mc{O}_w$ and 
such that $\rho(F,a_1,\dots,a_n)\not\subseteq \mc{O}_v$ for any $v\in S$.
\end{enumerate}
Consider the following $\Lar(F)$-formula $\widetilde{\gamma}(x, f)$ in the free variables $x, f$:
$$ \forall a_1, \ldots, a_n (\psi(a_1, \ldots, a_n) \wedge \rho(f, a_1, \ldots, a_n) \rightarrow \rho(x, a_1, \ldots, a_n)) $$
Then $\widetilde{\gamma}(x, f)$ is equivalent to a $\forall\exists$-$\Lar(F)$-formula $\gamma(x, f)$.
Furthermore, for any $f \in F$ transcendental over $K$, $\widetilde{\gamma}(F, f)$ is the integral closure of $K[f]$ in $F$.
In particular, for every finitary holomorphy ring $R$ of $F/K$, there exists some $f \in F$ transcendental over $K$ such that $R = \widetilde{\gamma}(F,f)=\gamma(F, f)$.
\end{prop}

\begin{proof}
Observe first that, for any two $\exists$-formulas $\phi$ and $\theta$, the formula $\phi\rightarrow \theta$ is  equivalent to an $\forall\exists$-formula.
Hence the formula $\widetilde{\gamma}(x,f)$ is equivalent to an $\forall\exists$-$\Lar(F)$-formula $\gamma(x, f)$.

Now fix $f \in F$ transcendental over $K$ and denote by $R$ the integral closure of $K[f]$ in $F$.
Let $R' = \widetilde{\gamma}(F, f)$, which is equal to $\gamma(F,f)$.
We need to show that $R = R'$.

For any $(a_1, \ldots, a_n) \in \psi(F^n)$, using $(i)$ and that $R$ is integral over $K[f]$, we have that $f\in\rho(F,a_1,\dots,a_n)$ if and only if $R\subseteq \rho(F,a_1,\dots,a_n)$.
Since this holds for all $(a_1, \ldots, a_n) \in \psi(F^n)$, we conclude that $R\subseteq R'$.

To prove the converse, we fix the sets $S=\{v\in\Val(F/K)\mid v(f)< 0\}$ and $\ovl{S}=\Val(F/K)\setminus S$.
Then $S$ is finite, and since $f$ is transcendental over $K$, we have $S\neq \emptyset$.
By \Cref{P:holo} we obtain that $R=\mc{O}(\ovl{S})$.

Consider an arbitrary $w\in  \ovl{S}$.
By  $(ii)$, there exists a tuple $(a_1, \ldots, a_n) \in \psi(F^n)$ such that $\rho(F, a_1, \ldots, a_n) \subseteq \mc{O}_w$ and $\rho(F, a_1, \ldots, a_n) \nsubseteq \mc{O}_v$ for any $v \in S$.
Let $S^\ast=\{v\in\Val(F/K)\mid \rho(F,a_1,\dots,a_n)\subseteq\mc{O}_v\}$.
Then $w\in S^\ast$ and $S^\ast\subseteq \ovl{S}$.
Since $\rho(F, a_1, \dotsc, a_n)$ is integrally closed in $F$, but not a subset of $\mc{O}_v$ for any $v \in S \neq \emptyset$,
it is a holomorphy ring,
and so \Cref{P:holo} implies $\rho(F, a_1, \ldots, a_n)=\mc{O}(S^\ast)$.
Hence we have $f\in R=\mc{O}(\ovl{S})\subseteq\mc{O}(S^\ast)=\rho(F,a_1,\dots,a_n)$, whereby $\rho(f,a_1,\dots,a_n)$ holds.
We obtain that $R'=\widetilde{\gamma}(F,f)\subseteq \rho(F,a_1,\dots,a_n)\subseteq \mc{O}_w$.

This argument shows that $R'\subseteq\mc{O}(\ovl{S})=R$. Hence $R=R'=\widetilde{\gamma}(F,f)=\gamma(F,f)$.

Consider now an arbitrary finitary holomorphy ring $R$ of $F/K$.
By \Cref{P:fingenhol}, there exists $f\in F$ transcendental over $K$ such that $R$ is the integral closure of $K[f]$ in $F$. 
Then we obtain that
$R=\widetilde{\gamma}(F,f)=\gamma(F,f)$.
\end{proof}


We now turn back to \Cref{Q:slot-ring} in the situation of a holomorphy ring of $F/K$.
\Cref{Pfisterstronggev} will provide a positive answer to this question for holomorphy rings given by a set of valuations that is characterised by the local anisotropy of a given quadratic form over $F$.

\begin{samepage}
\begin{lem}\label{L:confinitedependence}
Let $n, k \in \nat^+$, let $x_1, \ldots, x_k \in F^n$ be linearly independent.
Let $$S=\{v\in\mc{V}(F/K)\mid x_1, \ldots, x_k \in \mc{O}_v^n \mbox{ and } (\ovl{x}_1^v, \ldots, \ovl{x}_k^v)\mbox{ $Fv$-linearly independent}\}.$$
Then $\mc{V}(F/K)\setminus S$ is finite. 
\end{lem}
\begin{proof}
Let $\alpha$ denote the $(k\times n)$-matrix over $F$ whose rows are given by $x_1,\dots,x_k$.
Let $A$ denote the finite subset of $F$ consisting of all nonzero entries of $\alpha$ and of all nonzero $k$-minors of $\alpha$.
It follows that the set $S'=\bigcup_{a\in A} \{v\in\mc{V}(F/K)\mid v(a)\neq 0\}$ is finite and $\mc{V}(F/K)=S\cup S'$.
\end{proof}
\end{samepage}

\begin{lem}\label{L:SATvectors}
Let $m, n \in \nat$ with $n \geq 2$.
Let  $v_1, \ldots, v_m \in \Val(F/K)$ distinct, $S  \subsetneq \mc{V}(F/K)\setminus \{v_1,\dots,v_m\}$ and $\mc{O}=\bigcap_{v \in S } \mc{O}_v$.
Let $x_1,\dots,x_m,y_1,\dots,y_{n-1}\in F^n$ and $\gamma \in \zz$.
There exists $x \in\mc{O}^n\setminus \sum_{i=1}^{n-1} Fy_i$ such that $v_i(x-x_i) > \gamma$ for $1\leq i\leq m$ and 
$\ovl{x}^v \not\in \ovl{\sum_{i=1}^{n-2} Fy_i}^v$ for any $v \in S $.
\end{lem}
\begin{proof}
We may assume without loss of generality that $S $ is infinite.

By permuting the coordinates, we may assume without loss of generality that $(0, \ldots, 0, 1)\notin \sum_{i=1}^{n-1}Fy_i$. 
It follows by \Cref{L:confinitedependence} and  \Cref{L:LinearIndependence} that the set
$ S_1 = \lbrace v \in S \mid (0, \ldots, 0, 1) \in \ovl{\sum_{i=1}^{n-1}Fy_i}^v \rbrace $ is finite.         

For an arbitrary vector $z=(z^{(1)},\dots,z^{(n)})\in F^n$ we denote by $z^\ast$ its truncation $(z^{(1)},\dots,z^{(n-1)})\in F^{n-1}$.
Consider the $F$-subspace $W=\sum_{i=1}^{n-2} Fy_i^\ast$ of $F^{n-1}$.
Note that $\dim_F W\leq n-2$.
It follows by \Cref{L:LinearIndependence}
that $\ovl{W}^v$ is a proper $Fv$-subspace of $(Fv)^{n-1}$ for every $v \in \Val(F/K)$.
Using the Strong Approximation \Cref{T:SAT}, we find a vector $x'=(x^{(1)},\dots,x^{(n-1)}) \in\mc{O}^{n-1}\setminus W$ such that $v_i(x'-x_i^\ast) > \gamma$ for $1\leq i \leq m$ and $\ovl{x'}^v \in (Fv)^{n-1}\setminus \ovl{W}^v$ for every $v \in S_1$.
Let $S_2=\left\{v\in S \mid \ovl{x'}^v \in \ovl{W}^v \right\}$.
Then $S_2$ is finite, by \Cref{L:confinitedependence}, and $S_1\cap S_2=\emptyset$.

For every $v \in S_2$, as $v \not\in S_1$, there exists at most one element $\xi \in Fv$ such that $(\ovl{x^{(1)}}^v, \ldots, \ovl{x^{(n-1)}}^v, \xi) \in \ovl{\sum_{i=1}^{n-2} F y_i}^v$.
Similarly, since $(0, \ldots, 0, 1)\notin \sum_{i=1}^{n-1}Fy_i$, there exists at most one $\zeta\in F$ such that $(x^{(1)},\dots,x^{(n-1)},\zeta)\in \sum_{i=1}^{n-1}Fy_i$.
Hence, by the Strong Approximation \Cref{T:SAT}, there exists $x^{(n)} \in \mc{O}$ such that ${v_i(x^{(n)} - x_i^{(n)}) > \gamma}$ for $1 \leq i \leq m$, such that $x=(x^{(1)}, \ldots, x^{(n)})\in F^n\setminus \sum_{i=1}^{n-1}Fy_i$ and such that, for all $v \in S_2$, one has $\ovl{x}^v \not\in \ovl{\sum_{i=1}^{n-2}Fy_i}^v$.
Then $x^\ast=x'$ and $x$ is as desired.
\end{proof}

\begin{thm}\label{Pfisterstronggev}
Let $q$ be a non-degenerate quadratic form defined over $F$ with $\dim(q) \geq 3$.
Let $S \subseteq \mc{V}(F/K)$ with $|S|<\infty$ and  $R=\bigcap_{v \in \Delta_K q \cup S} \mathcal{O}_v$.
Assume that $\Split_v(q)\neq \emptyset$ for all $v \in \Delta_K q$.
Then $\Split_R(q)\neq \emptyset$.
\end{thm}
\begin{proof}
Let $n = \dim(q)$ and assume that $q$ is defined on $F^n$, that is, $q\in\Quad_n(F)$.
We consider the $F$-subspace $W=\{x\in F^n\mid \mf{b}_q(x, y)=0 \mbox{ for all }y\in F^n\}$ of $F^n$.
Since $q$ is non-degenerate, by \Cref{P:non-degenerate-characterisations}, we have $\dim_F W= 1$ if $\dim(q)$ is odd and $\car(F) = 2$, and otherwise we have $\dim_F W = 0$.
For all but finitely many valuations $v\in\Val(F/K)$, we have that $q\in \Quad_n(\mc{O}_v)$ and $\ovl{q}^v\in\Quad_n(Fv)$ is non-degenerate.
Hence, by adding finitely many elements to $S$, we may assume that $S\neq\emptyset$ and that, for all $v \in \Delta_K q \setminus S$, we have that $q \in \Quad_n(\mc{O}_v)$ and $\ovl{q}^v\in\Quad_n(Fv)$ is non-degenerate, so in particular regular.
For any $v\in\Delta_K q\setminus S$, we therefore have the following properties:
\begin{itemize}
\item $\ovl{q}^v$ anisotropic (\Cref{P:anisotropicResidue}).
\item $q(x)\in\mg{\mc{O}}_v$ for every $x\in \mc{O}_v^n\setminus \mfm_v^n$ (\Cref{P:binaryFormValuations}).
\item $\Split(q)\cap\mfm_v=\emptyset$ (\Cref{P:SI}), and hence $\Split_v(q)=\Split(q)\cap\qse{\mc{O}}_v$.
\item for $x \in \mc{O}_v^n \setminus (W + \mfm_v^n)$, the map $\mf{b}_{\ovl{q}^v}(\ovl{x}^v, \cdot):(Fv)^n \to Fv$ is nonzero.
\end{itemize}

By \Cref{L:splitoverinertial} and the hypothesis on $\Delta_K q$ we have $\Split_v(q)\neq \emptyset$ for all $v\in\Val(F/K)$.
For each $v \in S$, we fix $x_v, y_v \in F^n$ such that $-\frac{q(x_v)q(y_v)}{\mf{b}_q(x_v, y_v)^2}\in \mg{\mc{O}}_v\cap\qse{\mc{O}}_v$.
Since $S$ is finite, Weak Approximation allows us to choose $y\in F^n$ which, for each $v\in S$, is close enough to $y_v$ to guarantee that $-\frac{q(x_v)q(y)}{\mf{b}_q(x_v, y)^2} \in \mg{\mc{O}}_v\cap\qse{\mc{O}}_v$.

By \Cref{DeltaqnotF}, $\mc{V}(F/K)\setminus \Delta_K q$ is infinite. So in particular $\Delta_K q\cup S\neq \Val(F/K)$.
We apply \Cref{L:SATvectors} with $y_1\in W$ such that $W=Fy_1$, $y_2=y$, and $y_j = 0$ for $3 \leq j \leq n-1$, to obtain an element $x \in F^n$ which is linearly independent of $y$ and such that $-\frac{q(x)q(y)}{\mf{b}_q(x, y)^2} \in \mc{O}_v^\times\cap \qse{\mc{O}}_v$ for all $v\in S$, and $x \in \mc{O}_v^n\setminus (W+\mfm_v^n)$ for all $v \in \Delta_K q \setminus S$.

Rescaling $q$ by $q(x)^{-1}$ does not affect the claim nor the assumptions made on the valuations in $\Delta_K q \setminus S$, so we may assume without loss of generality that $q(x)=1$.
Rescaling $y$ by $\mf{b}_q(x,y)^{-1}$ does not affect the assumptions involving $y$ which were made so far, so we may as well assume  that $\mf{b}_q(x, y)=1$.
Hence $-q(y)=-\frac{q(x)q(y)}{\mf{b}_q(x, y)^2} \in\mg{\mc{O}}_v\cap\qse{\mc{O}}_v$ for all $v\in S$.
In particular, since $S \neq \emptyset$, we observe that $-q(y) \in \qse{F}$, i.e.~ $4q(y) \neq 1$.
From this we also conclude that $x$ and $y$ are $F$-linearly independent.

Consider the finite set $S'=\{v \in \Delta_K q\mid -q(y) \not\in \qse{\mc{O}}_v\}$.
Then  $S \cap S' = \emptyset$.
Hence, for $v\in S'$, we have that $x \in \mc{O}_v^n$ and $\mf{b}_{\ovl{q}^v}(\ovl{x}^v, \cdot):(Fv)^n\to Fv$ is not the zero map.
We apply Weak Approximation to obtain a vector $z_1 \in \bigcap_{v \in S'} \mc{O}_v^n$ 
such that $\mf{b}_q(x, z_1) \in \bigcap_{v \in S'} \mg{\mc{O}}_v$.
After rescaling $z_1$ by $\mf{b}_q(x, z_1)^{-1} \in \bigcap_{v \in S'} \mc{O}_v^\times$, we may assume $\mf{b}_q(x, z_1) = 1$.
Since $S' \subseteq \Delta_K q\setminus S$, we have $v(q(z_1))= 0$ for all $v\in S'$.
Since $n\geq 3$, we may choose $z_2 \in F^n\setminus (Fx\oplus Fy)$ with $\mf{b}_q(x, z_2) = 0$.
By Weak Approximation, there exists an element $\alpha \in F$ such that $\alpha z_2\in\mc{O}_v^n$ and $1-4q(z_1 + \alpha z_2) \not\in \mf{m}_v$ for all $v \in S'$, and $z_1 + \alpha z_2\notin Fx\oplus Fy$.
We set $z = z_1 + \alpha z_2$. 
We obtain that $\mf{b}_q(x, z) = 1$, and further $z \in \mc{O}_v^n$ and $-q(z)\in\qse{\mc{O}}_v$ for any $v\in S'$.

Recall that $x\in\mc{O}_v^n\setminus\mfm_v^n$ for all $v\in\Delta_K q\setminus S$ and set
$$S''=\left\{ v \in \Delta_K q \setminus (S \cup S')\,\left\vert\,  \mbox{ if }y,z \in \mc{O}_v^n \mbox{ then }\ovl{x}^v,\ovl{y}^v,\ovl{z}^v\mbox{ are $Fv$-lin.~dependent}\right.\right\}.$$
By \Cref{L:confinitedependence}, $S''$ is finite.

Since $\Delta_K q\cup S\neq \Val(F/K)$, the
Strong Approximation \Cref{T:SAT} implies that we can find $\varepsilon \in \bigcap_{v \in \Delta_K q \cup S} \mathcal{O}_v$ satisfying the following conditions:
\begin{itemize}
\item $v(\varepsilon^2) \geq -v(q(y))$ and $v((1-\varepsilon)^2) \geq -v(q(z))$ for all $v \in \Delta_K q$;
\item $v((1-\varepsilon)^2) > \max \lbrace 0, -v(q(z)), -2v(\mf{b}_q(y,z)) \rbrace$ for all $v \in S \cup S''$;
\item $v(\varepsilon^2) > \max \lbrace -v(q(y)), -2v(\mf{b}_q(y, z)) \rbrace$ for all $v \in S'$.
\end{itemize}
We set $u = \varepsilon y + (1-\varepsilon) z$ and $c=-q(u)$. 
Since $x, y, z$ are $F$-linearly independent, so are $x$ and $u$.
Since $\mf{b}_q(x,z)=\mf{b}_q(x,y)=1$, we get that $\mf{b}_q(x, u) = 1=q(x)$.
Hence 
 $c=- \frac{q(x)q(u)}{\mf{b}_q(x, u)^2}\in\Split(q)$.
We want to show that $c\in\Split_R(q)$, which would establish the statement.
Since $\Split_R(q)=\bigcap_{v\in \Delta_Kq\cup S} \Split_v(q)$, we thus need to show that $c\in \Split_v(q)$ for every $v\in\Delta_Kq\cup S$.
Recall in this context that, for all $v\in\Delta_K q\setminus S$, we have $\Split(q)\cap\qse{\mc{O}}_v=\Split_v(q)$.

By the choice of $c$ and $u$, we have
$$-c=q(u)=\varepsilon^2 q(y) + (1 - \varepsilon)^2 q(z) + \varepsilon(1 - \varepsilon)\mf{b}_q(y, z). $$

For $v \in S \cup S''$, the choice of $\varepsilon$ implies that $c\equiv - \varepsilon^2q(y)\equiv -q(y) \bmod \mfm_v$.
Hence, for $v\in S$, as $-q(y)\in \qse{\mc{O}}_v\cap\mc{O}_v^\times$, we obtain that $c\in\Split_v(q)$.
For $v \in S''$, since $S''\subseteq\Delta_K q\setminus S'$, we have $-q(y)\in\qse{\mc{O}}_v$ and conclude that 
$c\in\qse{\mc{O}}_v\cap\Split(q)=\Split_v(q)$.
For $v \in S'$, we have $-q(z)\in\qse{\mc{O}}_v$ and $c \equiv -q(z) \bmod \mfm_v$ by the choice of $\varepsilon$, so again we conclude that 
$c\in\qse{\mc{O}}_v\cap\Split(q)=\Split_v(q)$.

Now consider $v \in \Delta_K q \setminus (S \cup S' \cup S'')$.
Then $x,y,z\in\mc{O}_v^n$ and $\varepsilon, q(y), q(z) \in \mc{O}_v$.
By the choice of $S''$, we have that $\ovl{x}^v,\ovl{y}^v,\ovl{z}^v$ are $Fv$-linearly independent.
In particular $2u - x \in\mc{O}_v^n\setminus \mfm_v^n$.
As $v\in\Delta_Kq$, it follows by \Cref{P:binaryFormValuations} that $v(q(2u-x))=0$.
Since $-q(2u - x)= -(4q(u) - 2\mf{b}_q(u, x) + q(x)) = 1 + 4c$, we conclude that $c\in\qse{\mc{O}}_v\cap\Split(q)=\Split_v(q)$.
\end{proof}
\begin{cor}\label{cor:Pfisterstronggev}
Let $F/K$ be a function field in one variable, $q$ a non-degenerate quadratic form defined over $F$ with $\dim(q) \geq 3$.
If $\car(K) = 2$, assume additionally that $\dim(q) > 2[K : K\pow{2}]$.
Let $S \subseteq \Val(F/K)$ with $\lvert S \rvert < \infty$ and $R = \bigcap_{v \in \Delta_K q \cup S} \mc{O}_v $.
Then $\Split_R(q)\neq \emptyset$.
\end{cor}
\begin{proof}
If $\car(K)=2$, then for $v\in\Val(F/K)$, as $Fv/K$ is a finite extension, we have that $[Fv : (Fv)\pow{2}] = [K: K\pow{2}]$.
We obtain by \Cref{L:splitoverinertial} that $\Split_v(q) \neq \emptyset$ for every $v\in\Delta_Kq$.
The claim now follows from \Cref{Pfisterstronggev}.
\end{proof}
As an application, we obtain under very mild hypotheses that existential definability of a holomorphy ring of the form $\mc{O}(\Delta_K q)$ in a function field in one variable $F/K$ implies existential definability of all valuation rings containing it.
Variations of the corollary below will be used in Sections \ref{sec:large} and \ref{sec:ff-global}.
\begin{cor}\label{C:HolomorphyToValuation}
Let $F/K$ be a function field in one variable, $q$ a non-degenerate quadratic form defined over $F$ with $\dim(q) \geq 3$.
If $\car(K) = 2$, then assume further that $\dim(q) > 2[K : K\pow{2}]$.
Set $R=\mc{O}(\Delta_K q)$.
Then for every $v \in \Delta_K q$, there exist $c, g\in F$ such that $T^2 - T - c$ is irreducible over $F_w$ for all $w \in \Delta_K q$,
$$
\mc{O}_v \,=\, \left\{f\in F\,\,\left|\,\, \mbox{$\frac{f^2}{1-fg-c(fg)^2}$}\right. \in R\right\}, \, \mbox{ and }\,\,\,
\mfm_v \,=\, \left\{f\in F\,\,\left|\,\, \mbox{$\frac{f^2}{(1-f-cf^2)g^2}$}\right. \in R\right\}.
$$
In particular, if $R$ is $\exists$-$\Lar(F)$-definable, then so are $\mc{O}_v$ and $\mfm_v$ for every $v \in \Delta_K q$.
\end{cor}
\begin{proof}
By \Cref{cor:Pfisterstronggev} there exists $c \in \Split_R(q)$.
For every $w\in\Delta_K q$, the binary form $\pfi{c}{F_w}$ is similar to a subform of the anisotropic form $q_{F_w}$ by \Cref{P:0-in-Split(q)}, hence in particular itself anisotropic, so that $T^2 - T - c$ is irreducible over $F_w$ for every $w \in \Delta_K q$.
The element $c$ thus satisfies the hypotheses of \Cref{P:HolomorphyToValuation} and the rest of the statement follows.
\end{proof}

We mention the following consequence of \Cref{Pfisterstronggev}.
It gives a presentation of a Pfister form that is suitable for all valuations in $\Val(F/K)$ at once to the computation of its residue forms.
The reader should compare this to \Cref{C:presentation-Pfister-semiglobal}, where only a finite set of valuations is considered.

\begin{cor}\label{cor:pfister-slot-finding}
  Let $F/K$ be a function field in one variable.
  Let $n \geq 2$. If $\car(K) = 2$, then assume that $2^{n-1} > [K : K\pow{2}]$.
  Then every $n$-fold Pfister form over $F$ is isometric to $\pfi{a_1, \dotsc, a_{n-1}, c}F$ for some $a_1, \dotsc, a_{n-1} \in F^\times$
  and $c \in F$ with $v(c) = v(1 + 4c) = 0$ for all $v \in \Delta_K \pfi{a_1, \dotsc, a_{n-1}, c}F$.
\end{cor}
\begin{proof}
Let $q$ be an $n$-fold Pfister form over $F$ and set $R = \mc{O}(\Delta_Kq)$.
By \Cref{cor:Pfisterstronggev}, it follows that $\Split_R(q)\neq \emptyset$.
Now the statement follows by fixing $c \in \Split_R(q)$ and applying \Cref{P:Split(q)=rightslot}.
\end{proof}

\section{Local-global principles for function fields} 
\label{S:LGP}
\label{sec:KTheoryCohomQuadForms}

The Hasse--Minkowski Theorem asserts that an anisotropic quadratic form over a number field remains anisotropic over some completion of the number field. It holds more generally for global fields. This classical local-global principle for isotropy of quadratic forms has been used extensively, also in the context of definability and Hilbert's 10th Problem.
Although meanwhile a variety of local-global principles for isotropy of certain classes of quadratic forms are known over other fields than global fields, their usage for obtaining definability statements has been very modest so far.

In this section, we review a couple of local-global principles for function fields in one variable over certain base fields, such as \Cref{cor:LGP_HigherLocBase} and \Cref{cor:LGP_GlobLocBase} below.
These will serve us in the sequel.
\medskip

As already mentioned, a rank-$1$ valuation is a valuation whose value group can be viewed as an ordered subgroup of $\rr$.

The following result is shown in \autocite[Corollary 3.19]{Mehmeti_PatchingBerkQuad} under the hypothesis that $\car(K) \neq 2$; the arguments in \autocite{Mehmeti_PatchingBerkQuad} extend to prove the more general version formulated here. 
In the case where $w$ is a non-dyadic $\zz$-valuation on $K$, the result is also covered by \autocite[Theorem 3.1]{CTPS12}.

\begin{thm}[V.~Mehmeti]\label{T:Mehmeti}
  Let $(K,w)$ be a complete rank-$1$ valued field and let $F/K$ be a function field in one variable.
Let $q$ be a non-degenerate, anisotropic quadratic form of dimension at least $3$ over $F$.
Then $q$ is anisotropic over $F_v$ for some rank-$1$ valuation $v$ on $F$ such that either $v\vert_K$ is trivial or $v\vert_K = w$.
\end{thm}

\begin{proof}
Let $X$ be the projective variety defined by $q$. Since $q$ is non-degenerate, $X$ is smooth, by \Cref{P:non-degenerate-characterisations}.
We want to show that there exists some rank-$1$ valuation $v$ on $F$ which is either trivial on $K$ or extends $w$, and such that $X$ has no rational points over $\hat{F}_v$. 
We will obtain this by applying the local-global principle established in \autocite[Corollary 3.18~(2)]{Mehmeti_PatchingBerkQuad}; that result is stated for fields of meromorphic functions of projective compact irreducible normal $K$-analytic curves, but according to \autocite[Remark 3.20]{Mehmeti_PatchingBerkQuad}, any function field in one variable over $K$ arises in such a way.
Since the completion $\hat{F}_v$ contains the henselisation $F_v$ as a subfield, it will follow that $X$ has no rational points over $F_v$, and this means that $q_{F_v}$ is anisotropic.

Let $\Or(q)$ and $\SO(q)$ denote the orthogonal group and the special orthogonal group of $q$ as defined in \autocite[Section VI.23]{BOI} (where $\SO$ is denoted $\mathbf{O}^+$).
$\SO(q)$ is a connected rational linear algebraic group: a characteristic-free proof of this can be found in \autocite[Proposition 2.4]{ChernousovMerkurjev}.

The group $\SO(q)$ acts naturally on $X$. 
We claim that the action of $\SO(q)$ on $X$ is strongly transitive, i.e.\ $\SO(q_E)$
acts transitively on $X(E)$ for every field extension $E/F$ with $X(E) \neq \emptyset$.
Consider therefore an arbitrary field extension $E/F$
and two isotropic vectors $x$ and $y$ of $q_E$.
By \autocite[Theorem 8.3]{ElmanKarpenkoMerkurjev}, there exists an isometry $\s\in\Or(q_E)$ such that $\s(x)=y$.
If $\car(K) = 2$ and $\dim(q)$ is odd, then $\s \in \SO(q_E)$:
indeed $\SO$ is defined as the kernel of the determinant in this case (see \autocite[Section~VI.23, p.~348, (b)]{BOI}),
and $\s$ has determinant $\pm 1 = 1$ since $\car(K) = 2$.
Therefore $\s$ witnesses transitivity of the action of $\SO(q_E)$ on $X(E)$ in this case,
so we may now assume that $\car(K) \neq 2$ or $\dim(q)$ is even.
Since $\dim(q)\geq 3$, there exists a hyperplane reflection $\tau$ of $q_E$ such that $\tau(x)=x$.
We have $\tau\notin \SO(q_E)$, cf.\ \autocite[Example III.12.13]{BOI}.
Therefore either $\s$ or $\s \circ \tau$ lies in $\SO(q_E)$, and they both send $x$ to $y$, again proving transitivity.

In other terms, we have in $\SO(q)$ a connected rational linear algebraic group that acts strongly transitively on $X$, hence satisfying the hypotheses of \autocite[Corollary 3.18 (2)]{Mehmeti_PatchingBerkQuad}. Using this result, the statement follows.
\end{proof}

We now give a variation of this local-global principle, where we express the local conditions in terms of valuations which are not necessarily of rank-$1$, but which on the other hand are extensions of the given valuation on the base field $K$.
This will turn out to be more suited to our purposes.

\begin{cor}\label{cor:LGP_HigherLocBase}
  Let $(K,v_K)$ be a complete rank-$1$ valued field and let $F/K$ be a function field in one variable.
  Let $q$ be an anisotropic non-degenerate quadratic form over $F$ of dimension at least $3$.
  Then there exists a valuation $w$ of $F$ extending $v_K$ such that $q_{F_w}$ is anisotropic.
\end{cor}
\begin{proof}
  By \Cref{T:Mehmeti}, there is a valuation $v$ of rank $1$ on $F$ such that $q_{F_v}$ is anisotropic and such that $v|_K$ is either trivial or equal to $v_K$.
If $v|_K=v_K$, then we set $w = v$.
Assume now that the restriction $v|_K$ is trivial. 
Then the residue field $Fv$ is a finite extension of $K$. 
Since $v_K$ is a complete rank-$1$ valuation, it extends uniquely to a valuation $v_{Fv}$ on $Fv$, which is again complete of rank $1$, and in particular henselian. 
We denote by $v'$ the natural extension of $v$ to $F_v$.
Then $(F_v,v')$ is a henselian rank-$1$ valued field and $(F_v)v'=Fv$.
We compose $v'$ with the valuation $v_{Fv}$ defined on the residue field of $v'$; see \autocite[p.~45]{Eng05}.
This yields a rank-$2$ valuation $v''$ on $F_v$ with valuation ring $\mc{O}_{v''}=\{x\in \mc{O}_{v'}\mid \ovl{x}^{v'}\in\mc{O}_{Fv}\}$.
Since $v'$ and $v_{Fv}$ are henselian, we obtain by \autocite[Corollary 4.1.4]{Eng05} that $v''$ is henselian.
We set $w=v''|_F$. Then $w$ is a refinement of $v$ and $w|_K=v''|_K=v_K$.
Since $(F_v,v'')$ is henselian, we have $F_w\subseteq F_{v}$.
Therefore $q_{F_w}$ is anisotropic.
\end{proof}

Some local-global principles for quadratic forms over real fields apply only to totally indefinite quadratic forms.
E.~Witt  showed in \cite[Satz 22]{Witt37} that, over a function field in one variable over $\rr$, every totally indefinite quadratic form of dimension at least $3$ is isotropic.
For a quadratic form, the condition to be totally indefinite is itself given by the local isotropy over all real closures of the field. Hence Witt's result is also a local-global principle, although not in terms of valuations.
A generalisation of Witt's theorem was obtained in \cite[Theorem~I]{ELP73}, formulating sufficient conditions on the constant field of a function field in one variable which is real (i.e.~which admits a field ordering). Here we extend that statement to a full characterisation of the base fields $K$ such that the local-global principle with respect to real closures for quadratic forms of dimension at least $3$ holds for some  function field in one variable $F/K$.

\begin{thm}\label{P:Witt}
Let $F/K$ be a function field in one variable.
Every totally indefinite quadratic form of dimension at least $3$ over $F$ is isotropic if and only if $K(\sqrt{-1})$ has no finite field extension of even degree. 
\end{thm}
\begin{proof}
Assume first that $K(\sqrt{-1})$ has no finite field extension of even degree.
Let $K'$ be an algebraic closure of $K$ and let $F'/F$ be a compositum of $F$ and $K'$.
Then $F'/K'$ has transcendence degree $1$.
It follows by Tsen's Theorem \cite[Corollary 5.1.5]{PfisterBook} that every $3$-dimensional quadratic form over $F'$ is isotropic. 
As $K'/K(\sqrt{-1})$ is a direct limit of finite extensions of odd degree, so is $F'/F(\sqrt{-1})$.
It follows by Springer's Theorem  \cite[Theorem 6.1.12]{PfisterBook} that every anisotropic quadratic form over $F(\sqrt{-1})$ stays anisotropic over $F'$. 
Hence every $3$-dimensional quadratic form over $F(\sqrt{-1})$ is isotropic.
We conclude by \cite[Theorem F]{ELP73} that every totally indefinite quadratic form over $F$ of dimension at least $3$ is isotropic.

Assume now that $K(\sqrt{-1})$ has a finite field extension of even degree.
It follows that there exists a finite separable field extension $K'/K(\sqrt{-1})$ such that $K'$ has a quadratic field extension.
If $K'$ has an inseparable quadratic extension, then $K'$ and every finite extension are imperfect of characteristic $2$.
If $K'$ has a separable quadratic extension,
then by \autocite[Theorem 2]{WhaplesAlgebraic} $K'$ has a cyclic field extension of degree $2^r$ for every $r\in\nat$.
In either case, every finite field extension of $K'$ has a quadratic field extension.
We can now find a $\zz$-valuation $v$ on $F$ which is trivial on $K$ and whose residue field $Fv$ contains $K'$. 
Then $Fv$ has a quadratic field extension, and we conclude by 
  \Cref{P:aniso-2-fold-via-residue} that there exists an anisotropic totally indefinite $2$-fold Pfister form over $F$.
    \end{proof}

We now turn our attention to global fields.
By a \emph{global field} we mean a field which is either a number field or a \emph{global function field}, i.e.~a function field in one variable over a finite field.
We write $\mbb{P}(K)$ for the set of \emph{places} of a global field $K$, i.e.\ the set of equivalence classes of nontrivial absolute values; see \autocite[Chapter III, §1]{Neu99}.
Consider $\mf{p} \in \mbb{P}(K)$. We denote by $\hat{K}_{\mf{p}}$ the corresponding completion.
The place $\mf{p}$ is called \emph{real} (resp.~\emph{complex}), if the completion is isomorphic to $\rr$ (resp.~to $\cc$).
If $\mf{p}$ is neither real nor complex, then $\mf{p}$ is called \emph{finite}, and $\mf{p}$ is given by an ultra-metric induced by a unique $\zz$-valuation $v$ on $K$, and in this case $\hat{K}_\mf{p}=\hat{K}_v$.
On a global function field all places are finite. In any case, on a global field, there are only finitely many places which are not finite.
We denote by $K_{\mathfrak{p}}$ the henselisation $K_v$ if $\mathfrak{p}$ is a finite place induced by a valuation $v$,
the real closure for a real place and the algebraic closure for a complex place;
equivalently, $K_{\mathfrak{p}}$ is the relative algebraic closure of $K$ in $\hat{K}_{\mathfrak{p}}$ in all cases.

When $K_1/K$ and $K_2/K$ are field extensions such that $K_1 \otimes_K K_2$ is a domain (e.g.~it suffices that one of $K_1/K$ or $K_2/K$ is regular) 
we denote by $K_1K_2$ the fraction field of $K_1 \otimes_K K_2$, and we view $K_1$ and $K_2$ as subfields of $K_1K_2$ via the natural embeddings into $K_1 \otimes_K K_2$.

\begin{thm}[K.~Kato]\label{T:Kato}
  Let $K$ be a global field and $F/K$ a regular function field in one variable.
  Let $q$ be an anisotropic $3$-fold Pfister form over $F$.
  There are at least one and at most finitely many $\mf p \in \mbb P(K)$ such that $q$ is anisotropic over $F K_{\mathfrak{p}}$.
  For such $\mf p \in \mbb P(K)$, $q$ is additionally anisotropic over $F\hat{K}_{\mf p}$.
\end{thm}
\begin{proof}
We translate the statement into one about cohomology classes.
For an arbitrary field $K$, one  defines the cohomology group $H^3(K, \mathbb{Z}/2(2))$: if ${\car(K) \neq 2}$ this notation refers to the Galois cohomology group $H^3(\Gal(K^{\operatorname{sep}}/K), \mu_2^{\otimes 2})$, where $\mu_2$ denotes the multiplicative group $\mg{\zz}=\lbrace 1, -1 \rbrace$, and if $\car(K) = 2$, we refer to \autocite[143]{Kato} for the definition of this group, cf.~also \autocite[Section 101]{ElmanKarpenkoMerkurjev}.

As explained in detail in \autocite[146]{Kato}, for an arbitrary field $K$, there is a natural way to identify the set of isometry classes of $3$-fold Pfister forms over $K$ with a subset of $H^3(K, \mathbb{Z}/2(2))$.
Under this identification, the zero element of $H^3(K, \mathbb{Z}/2(2))$ corresponds to the (up to isometry) unique  isotropic $3$-fold Pfister form over $K$.
This identification commutes with field extensions.

Returning to the situation where $K$ is a global field and $F/K$ a regular function field in one variable, we first remark that for any $\mf p \in \mbb P(K)$, the natural map $H^3(FK_{\mathfrak{p}}, \mathbb{Z}/2(2)) \to H^3(F\hat{K}_{\mathfrak{p}}, \mathbb{Z}/2(2))$ is injective by
  \autocite[Theorem~2.12]{Jannsen},
  and so a $3$-fold Pfister form $q$ is anisotropic over $F\hat{K}_{\mathfrak{p}}$ if and only if it is anisotropic
  over $F K_{\mathfrak{p}}$.

It now remains to show that the natural map
  \[ H^3(F, \mathbb{Z}/2(2)) \to \prod_{\mf{p} \in \mbb{P}(K)} H^3(F\hat{K}_{\mf{p}}, \mathbb{Z}/2(2)) \]
  is injective and that its image is contained in $\bigoplus_{\mf{p}\in \mbb{P}(K)} H^3(F\hat{K}_{\mf{p}}, \mathbb{Z}/2(2))$. The injectivity is given by \autocite[Theorem 0.8(2)]{Kato}, and in the proof of that statement, on \autocite[p.~167]{Kato}, it is observed that the image of this map actually lies in $\bigoplus_{\mf{p}\in \mbb{P}(K)} H^3(F\hat{K}_{\mf{p}}, \mathbb{Z}/2(2))$; alternatively this latter fact is contained in \autocite[Proposition 1.2]{Jannsen} for $\car(K) \neq 2$.
\end{proof}

The following is an analogue to \Cref{cor:LGP_HigherLocBase} over a global base field.
\begin{cor}\label{cor:LGP_GlobLocBase}
  Let $F/K$ be a regular function field in one variable over a global field.
  Let $q$ be an anisotropic totally indefinite $3$-fold Pfister form over $F$.
  Then there exists a valuation $v$ on $F$ such that $v|_K$ is nontrivial and $q_{F_v}$ is anisotropic.
\end{cor}
\begin{proof}
Since $q$ is anisotropic, it follows by \Cref{T:Kato} that $q_{F\hat{K}_{\mf p}}$ is anisotropic for some $\mf p\in\mathbb{P}(K)$. 
Since $F\hat{K_{\mf p}}/\hat{K}_{\mf p}$ is a function field in one variable, we obtain by \Cref{P:Witt} that $\hat{K}_{\mf p}(\sqrt{-1})$ has a field extension of even degree. 
Hence $\hat{K}_{\mf p}$ is not isomorphic to $\rr$ or $\cc$, whereby $\mf p$ is a finite place.
Then $\hat{K}_{\mf p}=\hat{K}_w$ for some $\zz$-valuation $w$ on $K$.
Let $w'$ denote the natural extension of $w$ to $\hat{K}_w$.
 By \Cref{cor:LGP_HigherLocBase} there exists a valuation ${v}'$ on $F\hat{K}_w$ extending $w'$  such that $q$ is anisotropic over the henselisation $(F\hat{K}_w)_{{v}'}$. Let $v$ be the restriction of ${v}'$ to $F$. Then $v$ is a valuation on $F$ with $v|_K=w$, and since the henselisation $F_v$ is contained in $(F\hat{K}_w)_{{v}'}$, $q$ is anisotropic over $F_v$.
\end{proof}

\begin{rem}
  \Cref{cor:LGP_GlobLocBase} can also be established by translating the cohomological result \autocite[Proposition 5.2]{Kato}, without recourse to the patching techniques of \autocite{Mehmeti_PatchingBerkQuad}.
See for instance the presentation in \autocite[Section 2]{DittmannPop}.
\end{rem}

Local--global principles for Pfister forms that are, such as the one in \Cref{cor:LGP_GlobLocBase}, obtained by a translation to cohomology have served previously to establish definability results in the literature: see for instance \autocite{Pop_ElemEquivVsIsomII} and \autocite{DittmannPop}.

\section{Definability and decidability in function fields} 

Let $F$ be a field.
We will now look at the set of all valuations $v$ on $F$ such that a given quadratic form $q$ over $F$ remains anisotropic over $F_v$. If $q$ is anisotropic and belongs to a class of quadratic forms subject to a local-global principle in terms of valuations, then this set contains some nontrivial valuation. 

We denote by $\Val(F)$ 
the set of equivalence classes of valuations on $F$.
Since equivalent valuations have the same valuation ring, valuation ideal, residue field and henselisation, respectively, we may unambiguously associate to an equivalence class $v \in \Val(F)$ the valuation ring $\mc{O}_v$, valuation ideal $\mfm_v$, residue field $Fv$ and henselisation $F_v$, respectively, of any of its representatives.
Similarly, given $v \in \Val$ and $x, y \in F$, we can state conditions like $v(x) \geq v(y)$ or $v(x) > v(y)$ without ambiguity, as they do not depend on the representative of~$v$.

In the sequel, we abbreviate $\Val=\Val(F)$.
For any $a\in F$, we define $$\Val(a) = \{ v \in \Val \mid a \in \mfm_v \}\,.$$
The \emph{constructible topology} (or \emph{patch topology}) \emph{on $\Val$} is defined
as the coarsest topology such that sets $\Val(a)$ are open and closed for all $a \in F$.
Equivalently, this is the coarsest topology such that the set $\{ v \in \Val \mid b \in \mc{O}_v \}$ is open and closed for every $b \in F$, where the equivalence is due to the fact that $a \in \mfm_v$ if and only if $a^{-1} \not\in \mc{O}_v$ for $a \in F^\times$ and $v \in \Val$.
The topological space $\Val$ thus obtained is a compact Hausdorff space; see e.g.~\autocite[Section 1]{HuberKnebusch}.
The same topology is used in \autocite[Subsection~3.3]{DittmannPop}.

\begin{lem}\label{L:Delta-q-compact}
Let $q$ be a quadratic form over $F$ and $$\Delta = \lbrace v \in \Val \mid q_{F_v} \mbox{ is anisotropic} \rbrace\,.$$
Then $\Delta$ is a compact subspace of $\Val$.
\end{lem}
\begin{proof}
Since $\Val$ is a compact Hausdorff space, it suffices to show that $\Delta$ is closed in $\Val$.
  Consider $v \in \Val \setminus \Delta$. Then $q_{F_v}$ is isotropic.
  Since $F_v/F$ is an algebraic extension, $q_E$ is isotropic for some finite field extension $E/F$ contained in $F_v$.
  Now, by \autocite[Lemma 3.9]{DittmannPop}, the set of $v\in \Val$ for which $E$ embeds into $F_{v}$ over $F$ is open. This shows that $\Val \setminus \Delta$ is open.
\end{proof}
\begin{lem}\label{L:Delta-lemma}
Let $d \in F^\times$ and let $\Delta \subseteq \Val(d)$ be compact.
We have $$\bigcap_{v\in \Delta}\mc{O}_v[d^{-1}]=\bigcup_{n\in\nat} d^{-n}\left(\bigcap_{v\in \Delta} \mfm_v\right).$$
\end{lem}
\begin{proof}
The right-to-left inclusion is clear. To show the left-to-right inclusion,
consider $x\in\bigcap_{v\in \Delta}\mc{O}_v[d^{-1}]$.
For any $v \in \Delta$, since $x \in \mc{O}_v[d^{-1}]$, we have $d^{k} x \in \mc{O}_v$ for some $k \in \nat$, and then $d^{k+1}x \in \mfm_v$.
This shows that $\Delta \subseteq \bigcup_{n \in \nat} \Val(d^nx)$.
As $\Delta$ is compact and $\Val(d^nx) \cap \Val(d)\subseteq \Val(d^{n+1}x)$ for all $n \in \nat$, we obtain that $\Delta \subseteq \Val(d^mx)$ for some $m \in \nat$, whereby $d^mx \in \bigcap_{v \in \Delta} \mfm_v$.
\end{proof}

In the sequel, we return to the setup of a function field in one variable $F/K$.
Identifying any $\zz$-valuation on $F$ with its class in $\Val(F)$, we view $\Val(F/K)$ as a subset of $\Val(F)$.
Recall from \Cref{S:ffiov} that, for a quadratic form $q$ over $F$, we set $$\Delta_K q = \lbrace v \in \mc{V}(F/K) \mid q \text{ is anisotropic over } F_v \rbrace\,.$$

\begin{lem}\label{L:Delta-corsening}
Let $v\in\Val(F)$ be such that $v|_K$ is a rank-$1$ valuation and $d\in\mg{K}$ such that $v(d)>0$. 
Let $q$ be a quadratic form over $F$ such that $q_{F_v}$ is anisotropic. Then either $\mc{O}_v[d^{-1}]=F$ or $\mc{O}_v[d^{-1}]=\mc{O}_w$ for some $w\in\Delta_Kq$.
\end{lem}
\begin{proof}
Note that $\mc{O}_v[d^{-1}]$ is a valuation ring of $F$ containing $\mc{O}_v$. 
Since $v$ restricts to a rank-$1$ valuation on $K$ and $v(d) > 0$, we have $K \subseteq \mc{O}_v[d^{-1}]$.
Assume now that $\mc{O}_v[d^{-1}] \subsetneq F$.
Then $\mc{O}_v[d^{-1}] = \mc{O}_w$ for some $w \in \mc{V}(F/K)$.
Since $\mc{O}_v \subseteq \mc{O}_w$, we have
$F_w\subseteq F_v$. 
Since $q_{F_v}$ is anisotropic, so is $q_{F_w}$, whereby $w \in \Delta_K q$.
\end{proof}

The following technical statement will allow us to carry out the proof of our main results in the two main scenarios, which will be treated separately in the next two sections.

\begin{prop}\label{LocalglobalQT}
Let $F/K$ be a function field in one variable and $q$ a non-degenerate quadratic form over $F$. 
Assume one of the following two cases:
\begin{enumerate}[$(i)$]
\item $K$ is a global field, $F/K$ is regular, $q$ is a totally indefinite $3$-fold Pfister form over $F$, and $W=\{w\in\Val(K)\mid w \text{ nontrivial and }q_{K_wF}\text{ anisotropic}\}$. 
\item $(K,v_K)$ is a complete rank-$1$ valued field, $W=\{v_K\}$ and $\dim q\geq 3$.
\end{enumerate}
Let $\Delta=\{v\in\Val(F)\mid q_{F_v}\mbox{ anisotropic and } \mc{O}_v\cap K= \mc{O}_w \mbox{ for some }w\in W\}$ and $J=\{x\in K\mid w(x)>0\mbox{ for all }w\in W\}$. 
Then the following hold:
\begin{enumerate}[$(a)$]
\item $W$ is finite, $J\setminus\{0\}\neq \emptyset$, and $\Delta$ is a compact subspace of $\Val(F)$.
\item $\qss{q}{c} = F\cap \bigcap_{v\in\Delta} {\qss{q_{F_v}}{c}}$ for every $c\in \qse{F}$.
\item For any $d\in J\setminus\{0\}$, we have $\mc{O}(\Delta_Kq)\subseteq \bigcup_{n \in \nat} d^{-n} \left( \bigcap_{v \in \Delta} \mfm_v \right)$.
\item For any subset $C\subseteq K$ with $K= C\cdot\left(\bigcap_{w\in W}\mc{O}_w\right)$ and any $c \in\Split_{\mc{O}(\Delta_Kq)}(q)$, we have that $\mc{O}(\Delta_Kq)=C\cdot\qss{q}c$.
\end{enumerate}
\end{prop}
\begin{proof}
If $q$ is isotropic, then $\Delta = \emptyset$, $\qss{q}{c} = F$, and all parts of the statement are trivial.
We may therefore assume that $q$ is anisotropic throughout.

$(a)$\, 
If $K$ is a global field, then \Cref{T:Kato} says that $W$ is finite. Hence, in either of the two cases, $W$ is finite and consists of equivalence classes of \hbox{rank-$1$} valuations.
It follows by Weak Approximation that there exists an element $d\in\mg{K}$ such that $w(d)>0$ for all $w\in W$.
 In particular $J\setminus\{0\}\neq \emptyset$.

We set $\Delta'=\{v\in\Val(F)\mid v(d)>0\mbox{ and }\mc{O}_w\subseteq\mc{O}_v\mbox{ for some }w\in W\}$ and observe that $\Delta'$ is a closed subspace of $\Val(F)$.
Since $\Delta=\{v\in\Delta'\mid q_{F_v}\mbox{ anisotropic}\}$, we conclude by \Cref{L:Delta-q-compact} that $\Delta$ is compact.
 
 $(b)$\,
Consider $c\in\qse{F}$.
Let $\mathbb{E}$ be the set of all fields $L_v$ where $L/F$ is a separable field extension with $[L : F] \leq 2$ and 
$v\in\Val(L)$ is such that $v|_F\in\Delta$.
Clearly $$\qss{q}{c}\subseteq F\cap \bigcap_{v\in \Delta}\qss{q_{F_v}}c\subseteq F\cap \bigcap_{E\in\mathbb{E}} \qss{q_E}c\,.$$
Hence, to prove the statement, it suffices to show that $\qss{q}{c}= F\cap \bigcap_{E\in\mathbb{E}} \qss{q_E}c$.
This will follow 
by \Cref{localglobalaxiom-c} if we can show that $\mathbb{E}$ satisfies the hypothesis of that statement.

Consider an arbitrary separable field extension $L/F$ of degree at most $2$ such that $q_{L}$ is anisotropic.

Assume first that $(K,v_K)$ is complete.
Since $q_{L}$ is anisotropic, it follows by \Cref{cor:LGP_HigherLocBase} that there exists $v\in\Val(L)$ with $\mc{O}_{v}\cap K=\mc{O}_{v_K}$ such that $q_{L_{v}}$ is anisotropic, and then, since $F_{v|_F}\subseteq L_v$, we obtain that $v|_F\in \Delta$ and $L_v\in\mathbb{E}$.

Assume now that $K$ is a global field.
By \Cref{cor:LGP_GlobLocBase}, $q_{L_v}$ is anisotropic for some $v\in\Val(L)$ such that $v|_K$ is nontrivial. Since $K$ is a global field, we obtain that $\mc{O}_{v}\cap K=\mc{O}_w$ for a $\zz$-valuation $w$ on $K$. Then $K_wF\subseteq F_{v|_F} \subseteq L_v$, and since $q_{L_{v}}$ is anisotropic, so are $q_{K_wF}$ and $q_{F_{v|_F}}$, whereby  $w\in W$, $v|_F\in\Delta$ and $L_v\in\mathbb{E}$.

Hence the hypothesis on $\mathbb{E}$ in \Cref{localglobalaxiom-c} holds in either of the two cases.

$(c)$\, Let $d\in J\setminus\{0\}$. 
As $v|_K$ is a rank-$1$ valuation for every $v\in\Delta$, we obtain by \Cref{L:Delta-corsening} and \Cref{L:Delta-lemma} that 
$\mc{O}(\Delta_K q) \subseteq \bigcap_{v \in \Delta} \mc{O}_v[d^{-1}] = \bigcup_{n \in \nat} d^{-n} \left( \bigcap_{v \in \Delta} \mfm_v \right)$.

$(d)$\,
Set $R=\mc{O}(\Delta_Kq)$ and consider an arbitrary $c\in\Split_R(q)$. Then $0\in\qss{q}c$.
For every $v\in\Delta_Kq$, we have $c\in \mg{\mc{O}}_v\cap\qse{\mc{O}}_v$ and obtain by \Cref{Elocallem0}~$(b)$ that $\qss{q}c\subseteq c\mc{O}_v=\mc{O}_v$.
Since $K\subseteq R$, we conclude that $K\cdot\qss{q}{c}\subseteq R$.

Using $(a)$, we fix $d\in J\setminus\{0\}$. Then $R\subseteq \bigcup_{n \in \nat} d^{-n} \left( \bigcap_{v \in \Delta} \mfm_v \right)$, by $(c)$.

Let now $\mc{O}=\bigcap_{w\in W}\mc{O}_w$ and $C\subseteq K$ such that $C\cdot \mc{O}=K$.
For every $n\in\nat$ we fix $z_n\in C$ such that $d^{-n}\in z_n\mc{O}$.
Since for every $v\in\Delta$ there exists $w\in W$ with $\mc{O}\subseteq\mc{O}_w\subseteq\mc{O}_v$, we obtain that 
$v(z_n)\leq -nv(d)$ for all $n\in\nat$ and all $v\in\Delta$.

Consider $x \in R$ arbitrary.  Since $c\in\qse{R}$, we have $(1+4c)^{-2}x^2c\in R$.
By the choice of $d$, there exists $n\in\nat$ such that $d^{2n}(1+4c)^{-2}x^2c\in\bigcap_{v \in \Delta} \mfm_v$.
For any $v\in\Delta$ we have  $-v(z_n)\geq v(d^n)$ and thus $v(z_n^{-1}x)\geq v(d^nx) > v(1+4c) - \frac{1}{2}v(c)$, which by
\Cref{Elocallem0}~\eqref{it:Scqsufficient} implies that $z_n^{-1} x \in\qss{q_{F_v}}c$.
Hence, we obtain that $z_n^{-1}x\in\bigcap_{v\in\Delta} \qss{q_{F_v}}c=\qss{q}{c}$, in view of $(b)$.
Having this for all $x\in R$, we have shown that $R\subseteq C\cdot\qss{q}c$.
Since $C\cdot\qss{q}c\subseteq K\qss{q}c\subseteq R$, we conclude that $R=C\cdot\qss{q}c$.
\end{proof}


One motivation for studying existential definability of valuation rings in function fields relies on a well-known link to Hilbert's 10th Problem.
This was first studied for function fields of characteristic zero in \autocite{DenefDiophantine} and for function fields of positive characteristic in \autocite{PheidasHilbert10}.
For the applications which we will consider in the following sections, it is desirable to have precise control over the parameters involved in negative results on Hilbert's 10th Problem; see also \Cref{E:H10-why-control-parameters} on why this is important.
For the reader's convenience, we thus briefly discuss the general technique here, with an extra focus on the control of the parameters.
\Cref{P:Hilbert10CriterionPrecise} gives the relation between definability and Hilbert's 10th Problem which we are ultimately interested in.

Let $p$ in the sequel be a prime number.
For a field $K$ with $\car(K) = p$, we consider the subset
$P(K) = \{ (x^{p^s}, x) \colon x \in K, s \in \nat \}$ of $K\times K$.

\begin{lem}\label{L:FromDefiningP(F)toH10}
Let $F/K$ be a function field in one variable where $\car(K)=p$ and let $F_0 \subseteq F$ be a finitely generated subfield.
Assume that, for some $v \in \Val(F/K)$ with $v|_{F_0}$ nontrivial, the valuation ring $\mc{O}_v$ is $\exists$-$\Lar(F_0)$-definable in $F$, and that $P(F)$ is $\exists$-$\Lar(F_0)$-definable in $F$.
Then Hilbert's 10th Problem for $F$ with coefficients in $F_0$ has a negative answer.
\end{lem}
\begin{proof}
This is explained in \autocite[Section 2]{EisShlap17}, following a strategy developed in \autocite{Pheidas87} and \cite{PheidasHilbert10} (and \cite{VidelaHilbert10} for $p=2$), namely showing that the existential $\Lar(F_0)$-theory of $F$ encodes the existential theory of the structure $(\nat, +, \mid_p, 0, 1)$, where $\mid_p$ is a certain binary relation.
  The latter theory was shown in \autocite{Pheidas87} to be undecidable.
\end{proof}
\begin{thm}\label{T:DefiningP(F)}
Let $F/K$ be a function field in one variable where $\car(K)=p$.
Let $t \in F$ be transcendental over $K$.
Then $P(F)$ is $\exists$-$\Lar(\ff_p(t))$-definable.
\end{thm}

\begin{proof}
  This is \autocite[Corollary 4.5]{Daans-charp}.
\end{proof}

\begin{rem}\label{rem:H10criteron-precise-char2}
Before the recent \autocite{Daans-charp}, the following weaker variants of \Cref{T:DefiningP(F)} were available in the literature:
By \autocite[Section 5]{EisShlap17}, $P(F)$ is $\exists$-$\Lar(F)$-definable.
Furthermore, if $\car(F) \neq 2$, then by \autocite[Theorem~1.5]{Pasten_FrobeniusOrbits}
(on which \autocite{Daans-charp} is based),
$P(F)$ is $\exists$-$\Lar(\mathbb{F}_p(t))$-definable for any $t \in F$
such that $v(t) = 1$ for some $v \in \Val(F/K)$.
\end{rem}

\begin{prop}\label{P:Hilbert10CriterionPrecise}
Let $F/K$ be a function field in one variable. Let $F_0 \subseteq F$ and $K_0 \subseteq F_0 \cap K$ be subfields such that $K/K_0$ is regular, $F_0/K_0$ is a function field in one variable and $F$ is generated by $F_0 \cup K$.
Assume that there is a valuation $v \in \Val(F/K)$ whose valuation ring $\mc{O}_v$ is $\exists$-$\Lar(F_0)$-definable in $F$ and such that $v|_{F_0}$ is nontrivial.
Then there exists a finitely generated subfield $F_1 \subseteq F_0$
such that Hilbert's 10th Problem over $F$ with coefficients in $F_1$ has a negative answer.
\end{prop}
\begin{proof}
  There exists a finitely generated subfield $F_0' \subseteq F_0$ such that $\mathcal{O}_v$ is $\exists$-$\Lar(F_0')$-definable in $F$
  and $v|_{F_0'}$ is nontrivial.
  Suppose first that $\car(K) = 0$.
  Then \autocite[Theorem 10.3]{Mor05} asserts that,
  for a well-chosen element $f \in F$ as in \autocite[Paragraph 10.1.1]{Mor05},
  the ring $\mc{O}_v$ has undecidable positive existential $\Lar(f)$-theory,
  i.e.\ Hilbert's 10th Problem for $\mc{O}_v$ with the single parameter $f$ has a negative answer.
  Since $\mc{O}_v$ is existentially $\Lar(F_0')$-definable in $F$,
  it follows by a standard interpretation argument (see e.g.~\autocite[Proposition 2.7.8~$(ii)$]{Mor05}) that Hilbert's 10th Problem for $F$ with coefficients in $F_0'(f)$ has a negative answer.
  Actually, the element $f$ can be chosen to lie in $F_0$:
  Indeed, the requirements of \autocite[Paragraph 10.1.1]{Mor05} are that
  $f$ has simple ramification, no multiple zeroes and no multiple poles, and lies in the maximal ideal of $\mc{O}_v$.
  If $f \in F_0$ satisfies these conditions for the function field $F_0/K_0$
  (this makes sense since $v$ is not trivial on $F_0$,
  and such $f$ always exists, see \autocite[Remark 2.3.3]{Mor05}),
  then $f$ will automatically also satisfy these conditions for the function field $F/K$.
  Thus we can take $F_1 = F_0'(f) \subseteq F_0$.

  Suppose now that $\car(K) =p>0$.
  Let $t \in F_0$ be any element which is transcendental over $K$, and set $F_1 = F_0'(t)$.
  By \Cref{T:DefiningP(F)} the set $P(F)$
  is $\exists$-$\Lar(F_1)$-definable in $F$.
  Together with the existential $\Lar(F_0')$-definability of $\mc{O}_v$ in $F$,
  we conclude with \Cref{L:FromDefiningP(F)toH10}.
\end{proof}

We conclude with an example illustrating that stating \Cref{P:Hilbert10CriterionPrecise} only with $F_0 = F$ and $K_0 = K$ would give a less strong statement.
  
\begin{ex}\label{E:H10-why-control-parameters}
  Consider the function field $F = \rr(T)$ over $\rr$ and let $c \in \rr$ be a real number
  such that the Dedekind cut $\{ q \in \mathbb{Q} \mid q \leq c \}$ is an undecidable subset of $\qq$;
  in other words, $c$ is an uncomputable real number \autocite[§1]{Rice_recursive-reals}.
  For $q \in \mathbb{Q}$, the equation $X^2 - c + q = 0$ has a solution in $\rr(T)$
  if and only if it has a solution in $\rr$, and this is the case if and only if
  $c - q$ is non-negative, i.e. if and only if $c \geq q$.
  However, by the choice of $c$,
  there is no algorithm which decides whether this is the case
  (with $q$ as input).
  This shows that Hilbert's 10th Problem for $\rr(T)$
  with coefficients in $\qq(c)$ has a negative answer for trivial reasons.

  However, once one shows that the valuation ring $\mc{O}$ of the degree valuation of $F$
  is existentially $\Lar(T)$-definable,
  one can apply \Cref{P:Hilbert10CriterionPrecise} with $K=\rr$ and $K_0 = \rr \cap \overline{\qq}$,
  the field of algebraic real numbers, and $F_0 = K_0(T)$ to obtain that
  Hilbert's 10th Problem for $F$ with coefficients in $\qq(c_1, \dotsc, c_n, T)$ has a
  negative answer for some $c_1, \dotsc, c_n \in K_0$.
  This is now a nontrivial statement.
  In fact, one can eliminate the required parameters $c_1, \dotsc, c_n \in K_0$ by
  quantifying over roots of their minimal polynomials over $\qq$.
  This recovers the special case of \autocite[Theorem B]{DenefDiophantine} asserting
  that Hilbert's 10th Problem for $\rr(T)$ with parameters in $\qq(T)$ has a negative answer.
\end{ex}

\section{Function fields over large fields} \label{sec:large}

In this section and the next, we will prove our main theorems regarding existential definability of valuation rings in function fields.
We note first that, when $F$ is a rational function field in one variable over $\qq_p$ for some prime number $p$, then existential definability of valuation rings of $F$ was proven by Degroote and Demeyer in \autocite[Corollary 5.5]{DD12}.
Existential definability of valuation rings in arbitrary function fields in one variable over algebraic extensions of $\qq$ contained in $\qq_p$ for some odd prime $p$ can be found in recent work by Miller and Shlapentokh \autocite[Theorem 6.1]{MillerShlapentokh_v2}.
Both papers use a variation of the techniques developed by Kim and Roush in \autocite{KimRoush_DiophantineUnsolvabilityPAdic}, and as such do not rely on the more involved local-global principles from \Cref{sec:KTheoryCohomQuadForms}.

However, using the local-global principles from \Cref{sec:KTheoryCohomQuadForms}, we will refine the aforementioned results in two directions.
On the one hand, for function fields in one variable over global fields, we will obtain in \Cref{sec:ff-global} the existential definability of valuation rings in a uniform way.
On the other hand, in the current section, we discuss the analogous result for function fields over a class of base fields covering fields like $\qq_p$, and more generally all complete discretely valued fields.
To state it in full generality, we recall the notion of a \textit{large} field.

When $\La$ is a first-order language, $B$ an $\La$-structure and $A$ an $\La$-substructure of $B$, we say that $A$ is \emph{existentially closed in $B$} and write $A \prec_\exists B$ if every $\exists$-$\La(A)$-formula which holds in $B$ holds in $A$ as well.
In the class of fields (considered as $\Lar$-structures), an inequation $\neg (x = 0)$ can always be replaced by the equivalent formula $\exists y ( x \cdot y = 1 )$, so every existential formula is equivalent to a positive-existential formula, i.e.~an existential formula without negation symbols. It follows that in fact every existential $\Lar$-formula is equivalent in the class of fields to a system of polynomial equations, see for example \cite[Remark 3.4]{DDF}.
In particular, if $L/K$ is an extension of fields, then $K$ is existentially closed in $L$ (as an $\Lar$-substructure) if and only if every system of polynomial equations defined over $K$ has a root over $K$ provided that it has a root over $L$.

A field $K$ is called {\em large}\index{field!large} (or \emph{ample}) if every smooth curve over $K$ with a $K$-rational point has infinitely many $K$-rational points.
Equivalently, a field $K$ is large if $K\prec_\exists K(\!(t)\!)$; see \cite[Proposition 1.1]{Pop}.
For more on large fields see e.g.~\cite{BarysorokerFehm, Pop_LittleSurvey}.
Obviously, algebraically closed fields are large.
More generally, pseudo-algebraically closed fields are large, and so are real closed fields and fields which carry a nontrivial henselian valuation; see \autocite[Examples 1.1.A]{Pop_LittleSurvey}.
Finally, every algebraic field extension of a large field is again large; see \autocite[Proposition 2.7]{Pop_LittleSurvey}.

\begin{prop}\label{L:ffvextension}
Let $F/K$ be a function field in one variable, $K'/K$ a regular field extension and $v\in \mc{V}(F/K)$. 
Then there exists a unique $K'$-trivial valuation $w$ on $FK'$ such that $w|_F=v$.
Furthermore, $w\in\mc{V}(FK'/K')$, and
the natural homomorphism $Fv\otimes_K K'\to (FK')w$ is an isomorphism of $Fv$-algebras.
\end{prop}
\begin{proof}
Let $K^\ast$ denote the relative algebraic closure of $K$ in $F$. 
Then $K^\ast/K$ is a finite field extension and, since $K'/K$ is regular,  $K^\ast\otimes_KK'$ is a field, and by \cite[Corollary 2.6.8]{Fri08}, the extension $(K^\ast\otimes_K K')/K^\ast$ is regular.
Note that $FK'$ can be viewed as the fraction field of $F\otimes_{K^\ast}(K^\ast\otimes_K K')$.
Furthermore, the $Fv$-algebra $Fv \otimes_{K^\ast} (K^\ast\otimes_K K' )$ is a domain, and we have a natural isomorphism  $Fv \otimes_K K' \to Fv \otimes_{K^\ast} (K^\ast\otimes_K K' )$.
Hence, for the remainder of the proof, we replace  $K$ by $K^\ast$ and  $K'$ by $K^\ast\otimes_{K}K'$ to assume without loss of generality that $K$ is relatively algebraically closed in $F$.

By \autocite[Chapter~V, §5, Theorem~3]{Chevalley_FunctionFields}, the $K'$-trivial valuations on $FK'$ extending $v$ are in bijection with the minimal nonzero ideals of $\Sigma_{K'} = Fv \otimes_K K'$.
Since $K'/K$ is regular and $Fv/K$ is algebraic, $\Sigma_{K'}$ is a field and hence has only one minimal nonzero ideal.
This shows the existence and uniqueness of the extension of $v$ to a $K'$-trivial valuation $w$ on $FK'$.
Moreover, \autocite[Chapter~V, §5, Theorem~3]{Chevalley_FunctionFields} yields that $FK'w=\Sigma_{K'}$ and that $wFK'=vF=\zz$, whereby~$w \in \mc{V}(FK'/K')$.
\end{proof}

\begin{cor}\label{C:ffvextension}
Let $F/K$ be a function field in one variable, $K' = K(\!(t)\!)$ and $q$ a non-degenerate quadratic form over $F$.
Then $$\Delta_K q = \{w|_F\mid w\in \Delta_{K'} q_{FK'}\mbox{ with }F\not\subseteq \mc{O}_w\}\,.$$
\end{cor}
\begin{proof}
The extension $K'/K$ is regular, by \autocite[Lemma 2.6.9~$(b)$]{Fri08}.
Write $F' = FK'$.
Set $\Delta=\{w|_F\mid w\in \Delta_{K'} q_{F'}\mbox{ with }F\not\subseteq \mc{O}_w\}$.
Clearly, we have $\Delta\subseteq \Delta_K q$.
To show that $\Delta_K q\subseteq\Delta$, consider an arbitrary valuation $v \in\Delta_K q$.
By \Cref{L:ffvextension}, there exists $w\in\mc{V}(F'/K')$ such that $v=w|_F$. We claim that $w\in \Delta_{K'} q_{F'}$.
We have  $F'w\simeq Fv\otimes_KK'\simeq Fv(\!(t)\!) $ as $Fv$-algebras.
Since $Fv(\!(t)\!)$ (and hence also $F'w$) carries a valuation with residue field $Fv$, it follows from \Cref{P:binaryFormValuations} that every anisotropic form over $Fv$ remains anisotropic over $F'w$.
By \Cref{C:QF-over-unramified}, this implies that any non-degenerate anisotropic form over $F_v$ remains anisotropic over $F'_w$.
Applying this to $q_{F_v}$, we conclude that $w\in\Delta_{K'}q_{K'F}$.
\end{proof}

\begin{prop}\label{P:large-exist-closed-Delta-ring}
Let $K$ be a large field and $F/K$ a function field in one variable.
Let $K'=K(\!(t)\!)$, $F'=FK'$ and $q$ a non-degenerate quadratic form over $F$.
Set $R=\bigcap_{v\in\Delta_K q} \mc{O}_v$ and $R'=\bigcap_{w\in\Delta_{K'} q_{F'}} \mc{O}_w$.
Let $\varphi$ be an $\exists$-$\Lar(F)$-formula with one free variable such that $\varphi(F')=R'$.
Then $\varphi(F)=R$.
\end{prop}
\begin{proof}
It follows by \Cref{C:ffvextension} that $R'\cap F=R$.
Since $K$ is large, by \cite[Proposition 1.1]{Pop} we have that $K \prec_\exists K'$.
By \autocite[Lemma 7.2]{DDF} this implies that $F \prec_\exists F'$.
Hence we conclude that
 $\varphi(F)= \varphi(F')\cap F=R'\cap F=R$.
\end{proof}

\begin{prop}\label{P:largefields-delta-O-ring-intermsof-S-sets}
Let $K$ be a large field and $F/K$ a function field in one variable.
Let $q$ be a non-degenerate quadratic form over $F$ with $\dim q\geq 3$ and $R=\mc{O}(\Delta_Kq)$.
Let $c\in\Split_R(q)$ and let $\psi$ be an $\exists$-$\Lar(F)$-formula with one free variable such that
$K(\!(t)\!)=\psi(K(\!(t)\!)F)\cdot K[\![t]\!]$.
Then $R=\psi(F)\cdot \qss{q}{c}$.
\end{prop}
\begin{proof}
Let $K'=K(\!(t)\!)$, $F'=K'F$ and $\smash{R'=\bigcap_{w\in\Delta_{K'} q_{F'}} \mc{O}_w}$.
By \Cref{LocalglobalQT} $(d)$ (case $(ii)$), we have $R'=\psi(F')\qss{q_{F'}}c$.
From \Cref{P:S-existential} and the definition of $\qss{q}c$ we obtain an $\exists$-$\Lar(F)$-formula $\gamma$ in one free variable such that $\gamma(G)=\qss{q_{G}}c$ for every field extension $G/F$.
Combining $\psi$ with $\gamma$ we obtain an $\exists$-$\Lar(F)$-formula $\varphi$ in one free variable such that $\varphi(F)=  \psi(F)\qss{q}c$ and $\varphi(F')=\psi(F')\qss{q_{F'}}c$. Hence $\varphi(F')=R'$. 
We conclude by \Cref{P:large-exist-closed-Delta-ring} that $\psi(F)\qss{q}c = \varphi(F)=R'\cap F=R$.
\end{proof}

For a function field in one variable $F/K$, we denote by $\gen(F/K)$ the genus of $F/K$ in the sense of \autocite[Chapter II, §1]{Chevalley_FunctionFields}.
Note that $\gen(F/K) = \gen(F/K')$, where $K'$ is the relative algebraic closure of $K$ in $F$.

\begin{ex}\label{E:hypergenus}
Let $g\in\nat$.
We give some examples of curves of genus $g+1$ over prime fields.
Over $\mathbb{F}_2$, the affine plane curve described by the equation $Y^2 - Y = X^{2g+3}$ is geometrically irreducible, and its function field has genus $g+1$, by \autocite[p.~226]{Sti09}.
Over $\mathbb{Q}$ and over $\mathbb{F}_p$ for $p$ not dividing $2(2g+3)$, the affine plane curve described by $Y^2 = -X^{2g+3} + 1$ has the same properties, by \autocite[Proposition 6.2.3~$(b)$]{Sti09}, because $-X^{2g+3}+1$ is separable.
Lastly, over $\mathbb{F}_p$ for $p$ dividing $2g+3$, the curve described by $Y^2 = X^{2g+3} - X + 1$ has again the same properties, because $X^{2g+3}- X + 1$ is separable.
\end{ex}

\begin{prop}\label{P:hyperellipticformula}
  Let $g\in\nat$ and let $\Psi_g$ be the following $\exists$-$\Lar$-formula in one free variable $x$:
  \begin{align*}
    (2 = 0 &\wedge \exists y: x^{2g+3}(y^2-y) = 1) \vee \\
    (2(2g+3) \neq 0 &\wedge \exists y: x^{2g+3}y^2 = -1 + x^{2g+3}) \vee \\
    (2g + 3 = 0 &\wedge \exists y: x^{2g+3}y^2 = 1 - x^{2g+2} + x^{2g+3})
  \end{align*}
Let $F/K$ be a function field in one variable with $\gen(F/K)\leq g$ and such that $K$ is relatively algebraically closed in $F$.
Then $$K(\!(t)\!)=\Psi_g(K(\!(t)\!) F)\cdot K[\![t]\!]\,.$$
\end{prop}
\begin{proof}
Let $f$ be the polynomial $Y^2 - Y - X^{2g+3}$ if $\car(K) = 2$, the polynomial $Y^2 + X^{2g+3} - 1$ if $\car(K)$ does not divide $2(2g+3)$, and the polynomial $Y^2 - X^{2g+3} + X - 1$ if $\car(K)$ divides $2g+3$.
For any field extension $L/K$ and $x\in L$, the sentence $\Psi_g(x)$ asserts that $x \neq 0$ and there exists $y\in L$ with $f(x^{-1},y) = 0$.

Hensel's Lemma applied for the complete discrete valuation ring $K[\![t]\!]$ yields that all elements of the maximal ideal $tK[\![t]\!]$ are represented by the polynomial $Y^2-Y$, and if $\car(K) \neq 2$ that further all elements of $1 + tK[\![t]\!]$ are represented by the polynomial $Y^2$.
This implies that, for any $x \in K(\!(t)\!) \setminus K[\![t]\!]$, there exists $y \in K(\!(t)\!)$ with $f(x^{-1},y) = 0$, and so 
$K(\!(t)\!)\setminus K[\![t]\!]\subseteq \Psi_g(K(\!(t)\!))\subseteq \Psi_g(K(\!(t)\!) F)$.
Since every element of $K(\!(t)\!)$ can be written as $t^{-k}f$ with $k\in\nat^+$ and  $f\in K[\![t]\!]$, we obtain that $K(\!(t)\!)\subseteq \Psi_g(K(\!(t)\!) F)\cdot K[\![t]\!]$.

By \autocite[Chapter~V, §6, Theorem~5]{Chevalley_FunctionFields} we have $\gen(FK(\!(t)\!)/K(\!(t)\!))\leq \gen(F/K)\leq g$.
By \Cref{E:hypergenus} the function field of the curve described by $f(X,Y) = 0$ over the prime subfield of $K$
has genus $g+1$.
It follows that all $FK(\!(t)\!)$-rational points of this curve lie in $K(\!(t)\!)$; see e.g.~\autocite[Lemma 3.2]{Koe02}.
Therefore $\Psi_g(K(\!(t)\!) F)\subseteq K(\!(t)\!)$.
We conclude that $K(\!(t)\!)= \Psi_g(K(\!(t)\!) F)\cdot K[\![t]\!]$.
\end{proof}

\begin{cor}\label{P:2Pfister-large-def}
Let $g\in\nat$. 
There exists an $\exists$-$\Lar$-formula $\phi$ in $4$ free variables such that, for every function field in one variable $F/K$ with $\gen(F/K)\leq g$ where $K$ is large and for every $a\in \mg F, b\in\qse F$ and $c \in \Split_{\mc{O}(\Delta_K \pfi{a,b}F)}(\pfi{a,b}F)$, we have $\phi(F,a,b,c)= \mc{O}(\Delta_K \pfi{a,b}F)$.
\end{cor}
\begin{proof}
By \Cref{P:Pfister-qss-uniformdef}, there exists an $\exists$-$\Lar$-formula $\varphi$ in $4$ free variables such that $\varphi(F,a,b,c)=\qss{\pfi{a,b}F}c$ for every field $F$ and every $a,b,c\in F$.
Let $\Psi_g$ be the $\exists$-$\Lar$-formula from \Cref{P:hyperellipticformula}.
We combine $\varphi$ and $\Psi_g$ to obtain an $\exists$-$\Lar$-formula $\phi$ in $4$ free variables such that, for every field $F$ and every $a,b,c \in F$ we have $\phi(F, a, b, c) = \Psi_g(F) \cdot \qss{\pfi{a, b}F}c$.

Consider a function field in one variable $F/K$ with $\gen(F/K) \leq g$ and where $K$ is large.
Let $K'$ denote the relative algebraic closure of $K$ in $F$. Then $K'$ is large, and by \Cref{P:hyperellipticformula}, we have $K'(\!(t)\!)=\Psi_g(K'(\!(t)\!) F)\cdot K'[\![t]\!]$.
Since $\Val(F/K')=\Val(F/K)$, we obtain by \Cref{P:largefields-delta-O-ring-intermsof-S-sets} for any $a\in F^\times$, $b \in \qse F$ and $c \in \Split_{\mc{O}(\Delta_K \pfi{a,b}F)}(\pfi{a,b}F)$   that 
$\mc{O}(\Delta_K \pfi{a,b}F)=\Psi_g(F) \cdot \qss{\pfi{a, b}F}c=\phi(F,a,b,c)$.
\end{proof}

\begin{prop}\label{P:DefiningValuationsFunctionLargeUniform}
Let $F/K$ be a  function field in one variable.
Let $a,b\in F$ be such that  $a(1+4b)\neq 0$.
Set $\Delta=\Delta_K\pfi{a,b}F$ and $R=\mc{O}(\Delta)$.
Let  $v\in \Delta$ be such that $b\in\mg{R}\cap\qse{R}$,  $v(a)=1$ and $w(a)\leq 0$ for all $w\in\Delta\setminus\{v\}$.
Then
$$\mc{O}_v = \{x\in \mg F\mid (x^{-2}-ax^{-1} -a^2b)^{-1}\in R\}\cup \{0\}\,.$$
\end{prop}
\begin{proof}
Consider $x\in \mg F$. 
Applying  \Cref{uniformapprox} with $c=b$, $f=ax$ and $g=a$ yields that $x\in\mc{O}_v$ if and only if 
$(x^{-2}-ax^{-1} -a^2b)^{-1}=f^2g^{-2}(1-f-cf^2)^{-1}\in R$. This yields the claimed equality.
\end{proof}

Our next aim is to show that mild hypotheses on the base field $K$ are sufficient to have that, for any function field in one variable $F/K$ and any $v\in\mc{V}(F/K)$, one can choose
$a,b\in F$ such that the hypotheses in \Cref{P:DefiningValuationsFunctionLargeUniform} are satisfied.

\begin{samepage}
\begin{lem}\label{L:from-v-choose-2fold-delta-set}
Assume that $K(\sqrt{-1})$ has a separable quadratic field extension.
Let $F/K$ be a function field in one variable, let $S \subseteq \mc{V}(F/K)$ be finite and let $v\in\mc{V}(F/K) \setminus S$.
Then there exist elements $a\in \mg{F}$ and $b\in\qse{F}$ such that $v \in \Delta_K\pfi{a,b}F \subseteq \mc{V}(F/K) \setminus S$, and for  $R=\mc{O}(\Delta_K\pfi{a,b}F)$, we have $b\in\mg{R}\cap\qse{R}$, $v(a)=1$ and $w(a)\leq 0$ for all $w\in\Delta_K\pfi{a,b}F\setminus\{v\}$.
\end{lem}
\begin{proof}
By \autocite[Theorem 2]{WhaplesAlgebraic}, the hypothesis implies that $K(\sqrt{-1})$ has a cyclic field extension of degree $2^r$ for any $r\in\nat$.
Since $Fv/K$ is a finite field extension, it follows that $Fv$ has a separable quadratic field extension.
Hence, we can choose an element $b\in\mg{\mc{O}}_v\cap\qse{\mc{O}}_v$  such that the residue polynomial $T^2-T-\ovl{b}$ over $Fv$ is separable and irreducible.
Then the $1$-fold Pfister form $\pfi{\ovl b}{Fv}$ is anisotropic.

Set $W=\{w \in \mc{V}(F/K)\mid w(b) \neq 0\mbox{ or }w(1+4b) \neq 0\} \cup S$ and $\Delta=\Delta_K\pfi{b}F$. Note that the set $W$ is finite and $v\notin W$.
For $w\in W$, choosing $y\in F$ close enough to $1$ with respect to the natural topology
 on $F_w$, we have by \Cref{P:IsotropyQuadraticFormOpen} that $\pfi{y, b}{F_w}$ is isotropic.
 We fix $m\in\nat^+$ large enough such that this holds for any $y\in F$ and any $w\in W$ with $w(y-1)\geq m$.
 By \Cref{DeltaqnotF}, the set $\mc{V}(F/K) \setminus \Delta$ is infinite and in particular non-empty.
 Hence, by the Strong Approximation \Cref{T:SAT}, there exists $z \in \mg{F}$ with $v(z)=-1$, $w(1-z) = m$ for all $w\in W$, and $w(z)\geq 0$ for all $w \in \Delta \setminus \lbrace v \rbrace$.
We set $a=z^{-1}$.
It follows that $a\in\mg{F}$, $v(a)=1$ and $w(1-a)=m$ for all $w\in W$, and $w(a) \leq 0$ for all $w \in \Delta \setminus \lbrace v \rbrace$.

It follows by \Cref{L:existsPfister} that $v \in \Delta_K\pfi{a, b}F$, and our choices of $a,b$ and $m$ imply that $\Delta_K \pfi{a, b}F \subseteq \Delta \setminus W \subseteq \mc{V}(F/K) \setminus S$ and $b \in \mg{R}\cap \qse{R}$ for $R = \mc{O}(\Delta_K \pfi{a, b}F)$.
\end{proof}
\end{samepage}

\begin{thm}\label{T:DefiningValuationsFunctionLargeUniform}
Let $g\in\nat$. 
There exists an $\exists$-$\Lar$-formula $\rho$ in $3$ free variables such that, for every field $K_0$ such that $K_0(\sqrt{-1})$ has a separable quadratic field extension, every function field in one variable $F_0/K_0$ with $\gen(F_0/K_0)\leq g$ and every $v_0\in\Val(F_0/K_0)$, there exist $a\in \mg F_0$ and $b\in\mg F_0\cap\qse F_0$ such that, for every regular extension $K/K_0$  where $K$ is large, one has
$$\rho(F,a,b)=\mc{O}_v$$
where $F = F_0K$ and $v\in\Val(F/K)$ is such that $v|_{F_0}=v_0$.
\end{thm}
\begin{proof}
Let $\phi$ be the $\Lar$-formula from \Cref{P:2Pfister-large-def}. Let $\rho$ denote the $\exists$-$\Lar$-formula in the free variables $x, a, b$ given by
$$x = 0 \vee (\exists y,z : zx=1\wedge (z^2-az-a^2b)y = 1\wedge \phi(y,a,b,b))\,.$$
We claim that $\rho$ has the desired property.

To verify this, let $K_0$ be a field such that $K_0(\sqrt{-1})$ has a separable quadratic extension, let $F_0/K_0$ be a function field in one variable with $\gen(F_0/K_0)\leq g$ and $v_0 \in \mc{V}(F_0/K_0)$.
By \Cref{L:from-v-choose-2fold-delta-set} there exist $a\in\mg{F}_0$ with $v_0(a)=1$ and $b\in\qse{F}_0$ such that, for $R_0 = \bigcap_{w \in \Delta_{K_0}\pfi{a,b}{F_0}} \mc{O}_w$, we have $b\in \Split_{R_0}(\pfi{a,b}{F_0})$, $v_0 \in \Delta_{K_0}\pfi{a,b}{F_0}$ and $w(a) \leq 0$ for all $w \in \Delta_{K_0} \pfi{a,b}{F_0} \setminus \lbrace v_0 \rbrace$.
Note that in particular $b \in \mg{\mc O}_{v_0} \cap \qse{\mc O}_{v_0}$, whence $\pfi{\ovl{b}^{v_0}}{F_0v_0}$ is a Pfister form, hence regular, and $\pfi{b}{(F_0)_{v_0}}$ is anisotropic.
It follows from \Cref{P:anisotropicResidue} that $\pfi{\ovl{b}^{v_0}}{F_0v_0}$ is anisotropic.

Now consider a regular extension $K/K_0$ such that $K$ is large.
Set $F = F_0K$ and $R = \mc{O}(\Delta_K \pfi{a, b}F)$.
We have that $\gen(F/K) \leq \gen(F_0/K_0) \leq g$, and by \Cref{L:ffvextension}, $v_0$ has a unique extension to a $K$-trivial $\zz$-valuation $v$ on $F$, for which furthermore $Fv/F_0v_0$ is a regular extension.
In particular, since $\pfi{\ovl{b}^{v_0}}{F_0v_0}$ is anisotropic, it follows that $\pfi{\ovl{b}^v}{Fv}$ is anisotropic. 
By \Cref{L:existsPfister}, we conclude that $\pfi{a, b}{F_v}$ is anisotropic, whereby $v \in \Delta_K \pfi{a, b}F$.
Furthermore, for any $w \in \Delta_K \pfi{a, b}F \setminus \lbrace v \rbrace$, we have $w(a) \leq 0$ and $b \in \mg{R} \cap \qse{R}$.
By \Cref{P:2Pfister-large-def}, we obtain that $R=\phi(F,a,b,b)$.
Combining this with \Cref{P:DefiningValuationsFunctionLargeUniform}, we obtain that
\begin{align*}
\mc{O}_v & = \{0\}\cup \{x\in \mg F\mid (x^{-2}-ax^{-1} -a^2b)^{-1}\in R\} \\
& = \{0\}\cup \{x\in \mg F\mid (x^{-2}-ax^{-1} -a^2b)^{-1}\in \phi(F, a, b, b)\}\\
                 & =  \rho(F,a,b)\,. \qedhere
\end{align*}
\end{proof}

The following standard reduction argument  allows us to pass to a given finite field extension when proving existential definability of certain sets.
\begin{lem}\label{E:InterpretationFieldExtension}
Let $L_0/K_0$ be a finite field extension and $K/K_0$ a field extension such that $L_0 \otimes_{K_0} K$ is a domain.
Let $L = L_0K$.
Assume that $D \subseteq L$ is $\exists$-$\Lar(L_0)$-definable in $L$.
Then $D \cap K$ is $\exists$-$\Lar(K_0)$-definable over $K$.
\end{lem}
\begin{proof}
See e.g.~\autocite[Lemma 6.3]{MillerShlapentokh_v2}; there the statement is made for the case $K_0 = K$, but inspection of the main ingredient of the proof, \autocite[Lemma 2.1.17]{Shlapentokh}, reveals that one actually obtains the present refined statement.
\end{proof}

\begin{thm}\label{T:DefiningValuationsFunctionLarge}
Let $K_0$ be a field such that $K_0(\sqrt{-1})$ has a finite separable field extension of even degree.
Let $F_0/K_0$ be a function field in one variable and let $v_0 \in \Val(F_0/K_0)$.
Let $K/K_0$ be a regular extension such that $K$ is large and set $F = F_0K$.
Denoting by $v$ the extension of $v_0$ to $F$, $\mc{O}_v$ has an $\exists$-$\Lar(F_0)$-definition in $F$.
\end{thm}
\begin{proof}
We can find a finite field extension $F'_0/F_0$ such that, denoting by $K'_0$ the relative algebraic closure of $K_0$ in $F'_0$, $K'_0(\sqrt{-1})$ has a separable quadratic field extension and $F'_0/K'_0$ is regular.
Setting $K' = K_0'K$ we have that, as an algebraic extension of a large field, $K'$ is large.

Let $F' = F_0'F$.
There exists $v'\in\Val(F'/K')$ such that $\mc{O}_{v'}\cap F=\mc{O}_{v}$.
By \Cref{T:DefiningValuationsFunctionLargeUniform}, $\mc{O}_{v'}$ has an $\exists$-$\Lar(F'_0)$-definition in $F'$.
It follows by \Cref{E:InterpretationFieldExtension} that $\mc{O}_{v}$ has an $\exists$-$\Lar(F_0)$-definition in $F$.
\end{proof}
\begin{cor}\label{C:Hilbert10-large}
Let $K_0$ be a field such that $K_0(\sqrt{-1})$ has a finite separable extension of even degree.
Let $F_0/K_0$ be a function field in one variable.
Let $K/K_0$ be a regular extension such that $K$ is large and set $F=F_0K$.
Then there exists a finitely generated subfield $F_1 \subseteq F_0$ such that
Hilbert's 10th Problem over $F$ with coefficients in $F_1$ has a negative answer.
\end{cor}
\begin{proof}
By \Cref{T:DefiningValuationsFunctionLarge}, there exists a valuation $v \in \Val(F/K)$ such that $\mc{O}_v$ is $\exists$-$\Lar(F_0)$-definable in $F$.
The  statement now follows by \Cref{P:Hilbert10CriterionPrecise}.
\end{proof}

\begin{ex}\label{EX:[p-adic-H10}
\Cref{T:DefiningValuationsFunctionLarge}, and hence also \Cref{C:Hilbert10-large}, applies to the case where $K$ is a $p$-adic local field (finite extension of $\qq_p$) for some prime number $p$ and 
 $K_0$ is the relative algebraic closure of $\qq$ in $K$.
In that case, the result is covered by \autocite[Corollary 5.5]{DD12} for rational function fields, and closely related to \autocite[Theorem 6.1]{MillerShlapentokh_v2} for general function fields in one variable.
\end{ex}

\begin{rem}\label{rem:DefiningValuationsGlobal}
The uniformity statement contained in \Cref{T:DefiningValuationsFunctionLargeUniform} appears to be entirely new, also in the case discussed in \Cref{EX:[p-adic-H10}, and it can also not possibly be recovered from the methods used in \autocite[Corollary 5.5]{DD12} and \autocite[Theorem 6.1]{MillerShlapentokh_v2}.
In particular, the proof given in~\autocite{MillerShlapentokh_v2} involves a reduction to the case of $Fv = K$ by a passage to an extension of the constant field in combination with the interpretation argument from \Cref{E:InterpretationFieldExtension}, and this step cannot be made uniform as it depends on the degree of $Fv/K$.
As a consequence, we further obtain that, given a function field $F$ over a large field $K$ for which $K(\sqrt{-1})$ has a separable extension of even degree, all valuation rings $\mc{O}_v$ with $v \in \mc{V}(F/K)$ can be defined existentially with the same number of quantifiers.
\end{rem}

\begin{rem}
In the proof of \Cref{T:DefiningValuationsFunctionLarge}, the existence of an anisotropic $3$-dimensional quadratic form over the function field $F$ is crucial. Hence, in view of  \Cref{P:Witt}, the assumption that $K(\sqrt{-1})$ has a finite  extension of even degree is essential to our technique. We leave the question open whether this even degree extension of $K(\sqrt{-1})$ needs to be separable. In other words, could the presence of an inseparable quadratic extension serve the same purpose?
\end{rem}

\begin{thm}\label{E-large}
Let $n,g \in \nat$ be such that $n \geq 2$. 
There exists an $\exists$-$\Lar$-formula $\rho$ in $n+1$ free variables such that,
for any function field in one variable $F/K$ with $\gen(F/K)\leq g$ and for any
$a_1,\dots,a_{n-1}\in\mg{F}$ and $a_n\in\qse{F}$, the following hold:
\begin{enumerate}[$(i)$]
\item $\rho(F,a_1,\dots,a_n) \subseteq \mc{O}(\Delta_K\pfi{a_1,\dots,a_n}F)$.
\item 
If $K$ is large and $\Split_{\mc{O}(\Delta_K\pfi{a_1,\dotsc,a_n}F)}(\pfi{a_1, \dots, a_n}F)\neq \emptyset$, then 
$$\rho(F,a_1,\dots,a_n) = \mc{O}(\Delta_K\pfi{a_1,\dots,a_n}F)\,.$$
This equality holds in particular if $K$ is large and either $\car(K) \neq 2$, or $\car(K) = 2$ and $2^{n-1} > [K : K\pow{2}]$.
\end{enumerate}
\end{thm}
\begin{proof}
Let $\Psi_g$ be the $\exists$-$\Lar$-formula from \Cref{P:hyperellipticformula}, which has the property that $K(\!(t)\!) = \Psi_g(K(\!(t)\!) F) \cdot K[\![t]\!]$ for every field $K$ and every function field in one variable $F/K$ with $\gen(F/K)\leq g$ and $K$ relatively algebraically closed in $F$.
By combining $\Psi_g$ with the $\exists$-$\Lar$-formulas from \Cref{EX:Pfi-Split} and \Cref{P:Pfister-qss-uniformdef} defining 
$\Split(\pfi{a_1, \dots, a_n}F)$ and $\qss{\pfi{a_1, \dots, a_n}F}e$, respectively, in terms of the parameters $a_1,\dots,a_n$ and $e$, we obtain
an $\exists$-$\Lar$-formula $\rho$ in $n+1$ free variables such that, for any field $F$, $a_1, \ldots, a_{n-1} \in F^\times$ and $a_n \in \qse{F}$, 
one has
\begin{align*}
&\rho(F, a_1, \ldots, a_n) \\
=\enspace &\{x \in F \mid \exists c\in \Psi_g(F), e\in\Split(\pfi{a_1, \dots, a_n}F) \cap\mg{F} : ex\in c\qss{\pfi{a_1, \dots, a_n}F}e\} \\
=\enspace &\Psi_g(F) \left(\bigcup_{e \in \Split(\pfi{a_1, \ldots, a_n}F)\cap\mg{F}}e^{-1}\qss{\pfi{a_1, \ldots, a_n}F}e\right).
\end{align*} 

To show that the formula $\rho$ is as desired, consider a function field in one variable $F/K$ with $\gen(F/K)\leq g$.
To confirm $(i)$ and $(ii)$, we may replace $K$ by its relative algebraic closure in $F$ and thus assume without loss of generality that $K$ is relatively algebraically closed in $F$. 
Let $a_1, \ldots, a_{n-1} \in F^\times$ and $a_n \in \qse{F}$.
Set $q = \pfi{a_1, \ldots, a_n}F$ and $R = \mc{O}(\Delta_K q)$.
We need to show that $\rho(F, a_1, \ldots, a_n) \subseteq R$, and that equality holds when $K$ is large and $\Split_R(q) \neq \emptyset$.

Since $\Psi_g(K(\!(t)\!)F) \subseteq K(\!(t)\!)$, we have that $\Psi_g(F) \subseteq K(\!(t)\!) \cap F = K$.
Hence, the inclusion $\rho(F, a_1, \ldots, a_n) \subseteq R$ follows from \Cref{P:BigcapDeltaQ}.
Assume now that $K$ is large and $\Split_R(q) \neq \emptyset$. We fix an element $c \in \Split_R(q)$.
By \Cref{P:largefields-delta-O-ring-intermsof-S-sets} we have $R = \Psi_g(F) \qss{q}c$. 
By \Cref{P:BigcapDeltaQ} we conclude that $\rho(F, a_1, \ldots, a_n) = R$, as desired.

For the final part of the statement, it suffices to recall that, if we either have ${\car(K) \neq 2}$, or $\car(K) = 2$ and $2^{n-1} > [K : K\pow{2}]$, then by \Cref{cor:Pfisterstronggev} we automatically have that $\Split_{\mc{O}(\Delta_K\pfi{a_1,\dotsc,a_n}F)}(\pfi{a_1, \dots, a_n}F)\neq \emptyset$.
\end{proof}

\begin{thm}\label{T:AE-large}
Let $K$ be a large field such that $K(\sqrt{-1})$ has a separable extension of even degree.
Let $F/K$ be a function field in one variable.
There exists an $\forall\exists$-$\Lar(F)$-formula $\gamma(x, f)$ such that every finitary holomorphy ring of $F/K$ is equal to $\gamma(F, f)$ for some $f \in F$.
\end{thm}
\begin{proof}
We want to invoke \Cref{P:AE-criterion} and need to verify its hypotheses.
That is, we must construct formulas $\rho$ and $\psi$ satisfying conditions $(i)$ and $(ii)$.

As in the proof of \Cref{T:DefiningValuationsFunctionLarge}, we can find a finite field extension $F'/F$ such that, denoting by $K'$ the relative algebraic closure of $K$ in $F'$, $K'(\sqrt{-1})$ has a separable quadratic field extension.
Suppose that $\rho'(x, a_1, a_2)$ and $\psi'(a_1, a_2)$ are $\exists$-$\Lar(F')$-formulas.
Let $m = [F' : F]$ and let $\lbrace \beta_1, \ldots, \beta_m \rbrace$ be an $F$-basis of $F'$.
By a standard interpretation argument (generalising \Cref{E:InterpretationFieldExtension}) one obtains $\exists$-$\Lar(F)$-formulas $\rho$ and $\psi$ in $2m+1$ and $2m$ variables, respectively, such that, for $f, a_1^{(1)}, \ldots, a_1^{(m)}, a_2^{(1)}, \ldots, a_2^{(m)} \in F$, one has
\begin{align*}
F \models \psi (a_1^{(1)}, \ldots, a_1^{(m)}, a_2^{(1)}, \ldots, a_2^{(m)})
&\enspace\Leftrightarrow\enspace F' \models \psi'\left(\sum_{i=1}^m a_1^{(i)} \beta_i, \sum_{i=1}^m a_2^{(i)} \beta_i\right), \\
F \models \rho (f, a_1^{(1)}, \ldots, a_1^{(m)}, a_2^{(1)}, \ldots, a_2^{(m)})
&\enspace\Leftrightarrow\enspace F' \models \rho'\left(f, \sum_{i=1}^m a_1^{(i)} \beta_i, \sum_{i=1}^m a_2^{(i)} \beta_i\right).
\end{align*}
It readily follows that if $\rho'$ and $\psi'$ satisfy conditions $(i)$ and $(ii)$ of \Cref{P:AE-criterion} for the function field $F'/K'$ with $n = 2$, then $\rho$ and $\psi$ satisfy the corresponding conditions for the function field $F/K$ with $n = 2m$ (since every $v \in \Val(F/K)$ extends to finitely many elements of $\Val(F'/K')$, and if $\mc O'$ is an integrally closed $K'$-subalgebra of $F'$, then $\mc O' \cap F$ is an integrally closed $K$-subalgebra of $F$).
It thus suffices to find such formulas $\rho'$ and $\psi'$ for $F'/K'$.
For the remainder of the proof, we may hence assume without loss of generality that $F = F'$ and $K = K'$, i.e.~that $K(\sqrt{-1})$ has a separable quadratic extension.

We let $\rho$ be the formula from \Cref{E-large} for $n = 2$.
We write $ \psi(a_1, a_2) $ for the formula
$ a_1 \neq 0 \wedge 1+4a_2 \neq 0 \wedge \rho(a_2, a_1, a_2) $.
We show that the conditions $(i)$ and $(ii)$ from \Cref{P:AE-criterion} are satisfied with these $\psi$ and $\rho$.

Consider $(a_1, a_2) \in \psi(F^2)$. Then $a_2 \in \rho(F, a_1, a_2) \subseteq \mc{O}(\Delta_K \pfi{a_1, a_2}F)$, in view of \Cref{E-large}.
Hence $a_2 \in \Split(\pfi{a_1, a_2}F) \cap \mc{O}(\Delta_K \pfi{a_1, a_2}F)$.
By \Cref{C:splitoverinertial} combined with \Cref{Pfisterstronggev} we infer that $\Split_{\mc{O}(\Delta_K \pfi{a_1, a_2}F)}(\pfi{a_1,a_2}F) \neq \emptyset$.
\Cref{E-large} therefore yields that $\rho(F, a_1, a_2) = \mc{O}(\Delta_K \pfi{a_1, a_2}F)$, which is a holomorphy ring of $F/K$, as desired. This verifies $(i)$.

To confirm $(ii)$, consider a finite subset $S \subseteq \mc{V}(F/K)$ and a valuation $w \in \mc{V}(F/K) \setminus S$.
By \Cref{L:from-v-choose-2fold-delta-set}, there exist $a_1 \in \mg F$ and $a_2 \in \qse{F}$ such that $w \in \Delta_K \pfi{a_1, a_2}F \subseteq \mc{V}(F/K) \setminus S$ and further $a_2 \in \Split_{\mc{O}(\Delta_K\pfi{a_1,a_2}F)}(\pfi{a_1,a_2}F)$. 
In particular $a_2\in \mc{O}(\Delta_K\pfi{a_1,a_2}F)=\rho(F,a_1,a_2)$. Hence $(a_1, a_2) \in \psi(F^2)$, as desired.
\end{proof}

\section{Function fields over global fields}\label{sec:ff-global} 

Consider a function field in one variable $F/K$.
To make use of \Cref{LocalglobalQT} in order to existentially define valuation rings containing $K$ in $F$, we need to find an existentially definable subset $C\subseteq K$ as in \Cref{LocalglobalQT}~$(d)$.
In the case where $K$ is large, this was done in \Cref{P:hyperellipticformula}, and this led us in \Cref{T:DefiningValuationsFunctionLarge} to the desired conclusion.
In fact, when $K$ is large and $F/K$ is regular, then $K$ is existentially definable in $F$, by \cite[Theorem 2]{Koe02}, so we could have alternatively taken $C=K$ in that case.

The same approach by taking a hyperelliptic curve of sufficiently high genus would not work in the case where $K$ is a global field. 
In this case, it is more difficult to find sufficiently large subsets $C\subseteq K$ which are existentially definable.
Note that global fields are never large.
In fact, it is not known whether a global field $K$ is existentially definable in the rational function field $K(T)$. 
Recently, Garc\'ia-Fritz and Past\'en \autocite{GarciaFritzPasten-diophantineconstants} obtained a positive result on this in the case $\qq(T)/\qq$ which is conditional on certain conjectures on elliptic surfaces.

The goal of this section is to first show how to find, for an arbitrary regular function field $F/K$ where $K$ is global, a set $C\subseteq K$ which is existentially definable in $F$ such that \Cref{LocalglobalQT} $(d)$ (case $(i)$) can be applied.
This will be achieved with \Cref{cor:fatSetFromCurve}.
From that we will conclude in \Cref{T:DefiningValuationsFunctionGlobal} the existential definability of discrete valuation rings of $F$ containing $K$.

To find our set ${C}$, we study the sets of $F$-points of elliptic curves defined over $K$.
For a global field $K$ and a given function field in one variable $F/K$, we show that one can find an elliptic curve $E$ such that 
$E(K)$ is infinite and equal to $E(F)$. Putting the curve $E$ in Weierstraß form, this will allow us to give an $\exists$-$\Lar(K)$-definable subset of $F$ which is contained in $\mg{K}$ and contains elements of arbitrarily high value with respect to a given finite set of valuations on $K$.

In the case where $K$ is a number field, essentially this method was used in \autocite{Mor05} and \autocite{EisHil10padic}. 
Our crucial statement on existential definability of a big subset of $K$ in $F$, \Cref{cor:fatSetFromCurve}, is partially a generalisation of \autocite[Lemma 11.1]{Mor05} and \autocite[Theorem 5.4]{EisHil10padic} to function fields over arbitrary global fields.

By definition, an elliptic curve is a smooth projective curve of genus $1$ with a distinguished rational point.
We will make use of basic terminology and facts regarding elliptic curves, such as concerning their group law, Weierstraß equations, and isogenies, which can be found in \cite[Chapter III]{Silverman_EllCurves}.

Let $K$ be a field. We write $\overline K$ for an algebraic closure of $K$.
For a variety $X$ over $K$ and a field extension $L/K$, we denote by $X_{L}$ the base change of $X$ to $L$.
Given two elliptic curves $E$ and $E'$ over $K$, if there exists a field extension $L/K$ such that $E_L$ and $E'_L$ are isogenous, then this happens in particular for $L=\ovl K$.
To an elliptic curve $E$ over $K$ is associated an element $j(E)\in K$, its so-called \emph{$j$-invariant};  see \autocite[Section 3.3]{HusemoellerElliptic} for the definition.

\begin{prop}\label{L:elliptic-j-reduction}
Let $v$ be a $\zz$-valuation on $K$.
Let $E$ and $E'$ be two elliptic curves over $K$ such that $v(j(E))<0\leq v(j(E'))$. 
Then $E_{\ovl K}$ and $E'_{\ovl K}$ are non-isogenous.
\end{prop}
\begin{proof}
After replacing $(K, v)$ by a larger $\zz$-valued field, we may assume without loss of generality that $Kv$ is perfect.
Suppose that $E_{\ovl K}$ and $E'_{\ovl K}$ are isogenous.
Then there exists a finite field extension $K'/K$ such that $E_{K'}$ and $E'_{K'}$ are isogenous.
Since $v(j(E))<0\leq v(j(E'))$, we obtain by \autocite[Theorem 5.7.6]{HusemoellerElliptic} that there exists a further finite field extension $K''/K'$ such that $E_{K''}$ is of bad reduction and $E'_{K''}$ is of good reduction with respect to a discrete valuation $v''$ on $K''$ extending $v$. Hence they cannot be isogenous, by \autocite[Corollary 2]{GoodReductionAbelianVarieties}, so we have a contradiction.
\end{proof}

\begin{cor}\label{ellipticfinite}
Let $K$ be a field which is not an algebraic extension of a finite field. 
Let $n\in\nat$ and let $E_1, \dotsc, E_n$ be elliptic curves over $K$. 
Then there exists an elliptic curve $E$ over $K$ such that $E_{\overline K}$ is not isogenous to $(E_i)_{\overline K}$ for $1\leq i\leq n$.
\end{cor}
\begin{proof}
We first consider the case where $K=\qq$ or $K=\mathbb{F}_p(X)$ for some prime number $p$.

Since for any element $a \in K^\times$ one has $a \in \mc{O}_v$ for all but finitely many $\zz$-valuations $v$ on $K$, we can find a $\zz$-valuation $v$ on $K$ such that $j(E_i) \in \mc{O}_v$ for all $1 \leq i \leq n$.
Since any element of $K$ occurs as the $j$-invariant of some elliptic curve over $K$ \autocite[Proposition III.1.4(c)]{Silverman_EllCurves}, we can find an elliptic curve $E$ over $K$ such that $j(E) \not\in \mc{O}_v$.
Now we conclude by \Cref{L:elliptic-j-reduction} that $E_{\ovl K}$ is not isogenous to $(E_i)_{\ovl K}$ for $1\leq i\leq n$.
This concludes the proof in the cases where $K = \qq$ or $K = \ff_p(X)$ for some prime $p$.

In the general case, the field $K$ contains a subfield $K_0$ such that either $K_0 = \qq$ or $K_0 \simeq \ff_p(X)$ for some prime number $p$. Let $I_0$ denote the set of $i \in \{ 1, \ldots, n \}$ for which $(E_i)_{\ovl{K}}$ is isogenous to $(E_i')_{\ovl{K}}$ for some elliptic curve $E_i'$ defined over $K_0$, and fix such an elliptic curve $E_i'$ over $K_0$ for each $i \in I_0$.
By the special case established in the previous paragraph, we can find an elliptic curve $E$ over $K_0$ such that $E_{\ovl {K_0}}$ is not isogenous to $(E_i')_{\ovl {K_0}}$ for any $i \in I_0$.
By \cite[Corollary~20.4~(b)]{Milne_abelian-varieties}, $E_{\ovl {K}}$ is then also not isogenous to $(E_i')_{\ovl {K}}$ for any $i \in I_0$.
Since isogeny (over $\ovl K$) is an equivalence relation, we infer that $E_{\ovl {K}}$ is not isogenous to $(E_i)_{\ovl K}$ for $i \in I_0$.
On the other hand, by the choice of the set $I_0$, we also have that $E_{\ovl K}$ is not isogenous to $(E_i)_{\ovl K}$ for $i \in \{ 1, \ldots, n \} \setminus I_0$, since $E$ is defined over $K_0$.
Thus, $E$ is as desired.
 \end{proof}

\begin{prop}
\label{ellipticnonimbeddingcurve}
Let $K$ be a field which is not an algebraic extension of a finite field. 
Let $F/K$ be a function field in one variable.
Then there exists an elliptic curve $E$ over $K$ such that, for the compositum $F' = F \ovl{K}$, we have $E(F')=E(\ovl{K})$.
\end{prop}

\begin{proof}
  The field $F'$ is a function field over $\ovl{K}$.
  By \autocite[Proposition 7.3.13]{Liu} there exists a regular projective curve $C$ over $\ovl{K}$ with $F' = \ovl{K}(C)$, which we fix for the remainder of the proof.
  The Jacobian variety $J$ of $C$ is isogenous to a finite product of simple abelian varieties; see for instance \autocite[Corollary 19.1]{Mumford_AbelianVarieties}.
  It follows from  \Cref{ellipticfinite} that there exist infinitely many elliptic curves over $K$
  which remain pairwise non-isogenous over $\ovl{K}$.
  Therefore we may choose an elliptic curve $E$ over $K$ such that
  the base change $E_{\overline K}$ is not isogenous to any of the simple factors of $J$.

  Now suppose that $E(F') \setminus E(\ovl{K}) \neq \emptyset$. We fix a point $P\in E(F')\setminus E(\ovl{K})$. The image of the corresponding morphism $\Spec(F') \to E_{\ovl{K}}$ cannot be a closed point and hence must be the generic point, whereby we obtain an embedding $\ovl{K}(E_{\ovl{K}}) \to F'$.
  This embedding induces a non-constant morphism $C \to E_{\ovl{K}}$ \autocite[Theorem II.2.4]{Silverman_EllCurves}.
  But any morphism $C \to E_{\ovl{K}}$ factors through a morphism $J \to E_{\ovl{K}}$, and since $E_{\ovl{K}}$ is not isogenous to any of the factors of $J$, such a morphism has to be constant, which yields a contradiction.
   This proves that $E(F') = E(\ovl{K})$.
\end{proof}

The Mordell-Weil Theorem (see \autocite[Theorem VIII.6.7]{Silverman_EllCurves}) implies that, for an elliptic curve $E$ over a number field $K$, $E(K)$ is a torsion group if and only if it is finite. 
By  \autocite[Lecture 1, Section 5]{Ulmer_EllCurvesFunctionFields}, this  also holds for global function fields.
The following can be seen as a strengthening of this statement.

\begin{thm}[Levin, Merel]\label{T:LevinMerel}
Let $K$ be a global field.
There exists $n \in \nat^+$ such that, for every elliptic curve $E$ over $K$, $E(K)$ contains at most $n$ torsion points, and in particular, either $|E(K)|\leq n$, or $E$ has a $K$-rational point of infinite order.
\end{thm}
\begin{proof}
See \autocite{Merel} for the case where $K$ is a number field. For the case of a global function field, the statement essentially goes back to \cite{Levin68}; see also  \autocite[Lecture 1, Section 7]{Ulmer_EllCurvesFunctionFields} for a proof.
\end{proof}

A \emph{quadratic twist} of an elliptic curve $E$ over $K$ is another elliptic curve $E'$ over $K$ such that $E_{L}$ and $E'_{L}$ are isomorphic for some quadratic field extension $L/K$.
See \autocite[Section X.5]{Silverman_EllCurves} for more on twists.

\begin{lem}\label{elliptictwist}
  Let $E$ be an elliptic curve over a global field $K$. There exists a quadratic twist $E'$ of $E$ such that $E'(K)$ is infinite.
\end{lem}
\begin{proof}
We choose $n \geq 2$ for $K$ according to \Cref{T:LevinMerel}.
Hence for any elliptic curve $E'$ over $K$, the order of any element of $E'(K)$ either divides $n!$ or is infinite.
Let $H$ denote the $n!$-torsion subgroup of $E(\overline K)$.
Multiplication by $n!$ defines a morphism $E \to E$, and its kernel is $H$.
It follows by \autocite[Proposition II.2.6]{Silverman_EllCurves} that $H$ is finite.

We may assume that $E$ is presented by an affine Weierstraß equation which we will specify below, distinguishing three cases.
For each case, we then consider the set $S=E(\overline K)\cap (K\times \overline K)$.
Since $K$ is infinite, so is $S$.
We will show that for each $(x_0, y_0) \in S \setminus H$, there exists an elliptic curve $E'$ over $K$ and an isomorphism $E_{K(y_0)} \to E'_{K(y_0)}$ which maps $(x_0, y_0)$ to a $K$-rational point $P$ of $E'$ of the same order as $(x_0,y_0)$.
In particular, the order of this $K$-rational $P$ is larger than $n$ and hence  infinite, in view of the choice of $n$. Since $[K(y_0) : K] \leq 2$, we further have that $E'$ is a quadratic twist of $E$.
Since $H$ is finite and $S$ is infinite, it follows that the points of $S\setminus H$ give rise to quadratic twists of $E$ having a $K$-rational point of infinite order.
  
To define $E'$, let us first assume that $\car(K) \neq 2$. Then $E$ without the identity point may be presented by an affine equation $Y^2 = f(X)$ for a separable monic cubic polynomial $f \in K[X]$.
To  $(x_0,y_0)\in S \setminus H$, we associate the elliptic curve $E'$ given by $f(x_0)Y^2 = f(X)$
(noting that $f(x_0) = y_0^2 \neq 0$ since $(x_0,y_0)$ is not $2$-torsion),
and we observe that
$$ E_{K(y_0)} \to E'_{K(y_0)} : (x, y) \mapsto (x, yy_0^{-1}) $$
is an isomorphism that maps $(x_0, y_0)$ to the $K$-rational point $(x_0, 1)$ of $E'$.

Assume now that $\car(K) = 2$.
By \autocite[Proposition A.1.1]{Silverman_EllCurves},  $E$ can be presented by an affine equation of one of the following two types:
\begin{enumerate}[$\rm(a)$]
\item\label{it:nonsupersingular} $Y^2 + XY = X^3 + bX^2 + c$ for $b, c \in K$ with $c \neq 0$;
\item\label{it:supersingular} $Y^2 + aY = X^3 + bX + c$ for $a, b, c \in K$ with $a \neq 0$.
\end{enumerate}
Assume first that $E$ is given by \eqref{it:nonsupersingular}.
For a point $(x_0,y_0)\in S \setminus H$, we set ${d = x_0^{-2}(x_0^3 + bx_0^2 + c)}$ and consider the elliptic curve $E'$ given by $$Y^2 + XY = X^3 + (b+d)X^2 + c.$$
(Note that $x_0 \neq 0$ since otherwise $(x_0,y_0)$ would be $2$-torsion.)
For $t = y_0x_0^{-1}$,
one observes that $t^2 + t = d$ and obtains that
$$ E_{K(y_0)} \to E'_{K(y_0)} : (x, y) \mapsto (x, y+tx) $$
is an isomorphism which maps $(x_0, y_0)$ to the $K$-rational point $(x_0, 0)$ of $E'$.

Assume now that $E$ is given by \eqref{it:supersingular}. For $(x_0,y_0)\in S \setminus H$, we set $d = x_0^3 + bx_0 + c$ and consider the curve $E'$ given by $$Y^2 + aY = X^3 + bX + c + d.$$
One observes that
$$ E_{K(y_0)} \to E'_{K(y_0)} : (x, y) \mapsto (x, y + y_0) $$
is an isomorphism which maps
$(x_0, y_0)$ to the $K$-rational point $(x_0, 0)$ of $E'$.
\end{proof}

\begin{cor}\label{ellipticexist}
  Let $K$ be a global field and $F/K$ a function field in one variable.
  Denote by $K'$ the relative algebraic closure of $K$ in $F$.
  Then there exists an elliptic curve $E/K$ such that $E(K)$ is infinite and $E(F) = E(K')$.
\end{cor}
\begin{proof}
Let $F'$ denote a compositum of $F$ and $\ovl K$ over $K'$. By \Cref{ellipticnonimbeddingcurve}, we can find an elliptic curve $E$ over $K$ with $E(F')=E(\ovl K)$.
This property does not change if we replace $E$ by a twist.
Hence, by \Cref{elliptictwist}, we may further choose $E$ in such way that $E(K)$ is infinite.
\end{proof}

When $K$ is a topological field, then the topology on $K$ naturally induces a topology on the projective plane $\mbb{P}^2_K(K)$.
For any curve $C$ over $K$ embedded into $\mbb{P}^2_K$, the set of $K$-rational points $C(K)$ is a subset of $\mbb{P}^2_K(K)$ and can thus naturally be endowed with the subspace topology.

A \emph{non-archimedean local field} is a complete discretely valued field with a finite residue field.
\begin{prop}\label{prop:ellipticlocal}
  Let $E$ be an elliptic curve over a non-archimedean local field $K$ with its associated topology.
  For any point $P \in E(K)$, we have that $(n! \cdot P)_{n \in \nat}$ converges to the point at infinity $O$.
\end{prop}
\begin{proof}
When $K$ is a non-archimedean local field, its topology is Hausdorff, locally compact, and totally disconnected \autocite[Section 7]{CasselsGlobalFields}.
As such, $\mbb{P}^2_K(K)$ is Hausdorff, compact, and totally disconnected \autocite[Exercise VII.7.6(a)]{Silverman_EllCurves}.
After embedding the elliptic curve $E$ into the projective plane $\mbb{P}^2_K$, $E(K)$ is a subset of $\mbb{P}^2_K(K)$, and in fact a closed subspace of $\mbb{P}^2_K(K)$ \autocite[Exercise VII.7.6(c)]{Silverman_EllCurves}, whereby it is a compact, totally disconnected, Hausdorff topological group.
That is, it is a profinite group.

By \autocite[Proposition IV.2.8]{Neu99} a profinite group is isomorphic as a topological group to a projective limit of finite groups.
Since for a (additively written) finite group $G$ and $x \in G$ one has $n!\cdot x = 0$ for every natural number $n \geq \lvert G \rvert$, one obtains that  $(n!\cdot P)_{n \in \nat}$ converges indeed to $0$ for every point $P \in E(K)$.
\end{proof}

\begin{cor}\label{cor:ellipticlocal-coordinates}
Let $E$ be an elliptic curve over a global field $K$ with $E(K)$ infinite, embedded into the projective plane in Weierstraß form.
Let $V$ be a finite set of $\zz$-valuations on $K$.
For any $m \in \nat$, there exist $x, z \in K^\times$ such that $(x : 1 : z) \in E(K)$ and $v(x), v(z) \geq m$ for all $v\in V$.
\end{cor}
\begin{proof}
  By \Cref{T:LevinMerel} (or the Mordell-Weil Theorem) and the assumption that $E(K)$ is infinite, there exists $P \in E(K)$ of infinite order.
  For all sufficiently large $n\in\nat$, we may write $n!\cdot P=(x_n:1:z_n)$ with $x_n,z_n\in K^\times$,
  since only finitely many points of $E$ have a vanishing homogeneous coordinate.
  By \Cref{prop:ellipticlocal}, for each $v\in V$, the sequence $(n! \cdot P)_{n \in \nat}$ in $E(K)$ converges with respect to the $v$-adic topology to the point at infinity $O = (0 : 1 : 0)$.
  It follows that, for $m\in\nat$, there exists a natural number $n$ such that $x_nz_n\neq 0$ and $v(x_n),v(z_n)\geq m$. 
\end{proof}

\begin{cor}\label{cor:fatSetFromCurve}
  Let $K$ be a global field and $F/K$ a function field in one variable.
  Let $K'$ be the relative algebraic closure of $K$ in $F$.
  Then there exists an $\exists$-$\Lar(K)$-formula $\psi$ in one free variable such that, for any $V \subseteq \Val(K')$ finite, we have 
  $K' = \psi(F) \cdot\left(\bigcap_{v\in V}\mc{O}_v\right)$.
\end{cor}
\begin{proof}
By \Cref{ellipticexist} there is an elliptic curve $E/K$ such that $E(K)$ is infinite and $E(K') = E(F)$.
We may view $E$ embedded into the projective plane in Weierstraß form and obtain in this way an $\exists$-$\Lar(K)$-formula $\psi$ such that 
  \[ K' \supseteq \psi(K')=\psi(F) = \{ x \in F^\times \mid \exists z \in F : (x^{-1} : 1 : z) \in E(F) \} .\]
  It follows that $K' \supseteq \psi(F) \cdot\left(\bigcap_{v\in V}\mc{O}_v\right)$, and it remains to prove the other inclusion.
  Consider $a \in K'$ arbitrary.
  Since $K'$ is a global field, any nontrivial element of $\Val(K')$ is the class of a unique $\zz$-valuation.
The trivial valuation may be excluded from $V$.
It follows by \Cref{cor:ellipticlocal-coordinates} that we may find $x \in \psi(K') = \psi(F)$ such that $v(x) \leq v(a)$ for all $v \in V$.
We conclude that $a \in x\left(\bigcap_{v \in V} \mc{O}_v \right) \subseteq \psi(F) \left(\bigcap_{v \in V} \mc{O}_v \right)$ as desired.
\end{proof}

\begin{lem}\label{L:from-v-choose-3fold-delta-set}
Let $K$ be a global field.
Let $F/K$ be a function field in one variable, let $S \subseteq \mc{V}(F/K)$ be finite and let $v\in\mc{V}(F/K) \setminus S$.
There exist $a_1, a_2\in \mg{F}$ and $a_3\in\qse{F} \cap F\pow{2}$ such that $q = \pfi{a_1, a_2, a_3}F$ is totally indefinite, $v \in \Delta_K q \subseteq \Val(F/K) \setminus S$, $v(a_1)=1$, $w(a_1)\leq 0$ for all $w\in\Delta_K q\setminus\{v\}$ and $a_3\in\mg{\mc{O}(\Delta_K q)}\cap\qse{\mc{O}(\Delta_K q)}$.
\end{lem}
\begin{proof}
Since $Fv$ is a finite extension of $K$, it is a global field.
We fix a $\zz$-valuation $u$ on $Fv$. Since $(Fv)u$ is a finite field, by  \autocite[Lemma 6.5]{DaansGlobal}, we can find $\eta\in \mg{\mc{O}}_u$ with $1+4\eta^2\in\mg{\mc{O}}_u$ such that $\pfi{\ovl{\eta}^2}{(Fv)u}$ is anisotropic.
We then choose $\alpha\in \mg{Fv}$ with $u(\alpha)=1$ and obtain by \Cref{L:existsPfister} that $\pfi{\alpha,\eta^2}{Fv}$ is anisotropic.
Now we take $a_2, e \in \mc{O}_v^\times$ with $a_2+\mfm_v=\alpha$ and $e+\mfm_v=\eta$.
Set $a_3 = e^2$.
There are only finitely many valuations $w \in \mc{V}(F/K)$ for which $w(a_3) \neq 0$ or $w(1+4a_3) \neq 0$.
Therefore, and in view of \Cref{P:IsotropyQuadraticFormOpen} and \Cref{DeltaqnotF}, we may use Strong Approximation to find $a_1 \in F$ such that $v(a_1) = 1$, $w(a_1) \leq 0$ for all $w \in \Delta_K \pfi{a_3}F \setminus \lbrace v \rbrace$ and,
for each of the finitely many $w \in \mc{V}(F/K)$ with $w(a_3) \neq 0$, $w(1+4a_3) \neq 0$ or $w \in S$, we have that $\pfi{a_1, a_3 }{F_w}$ is isotropic.
We set $q = \pfi{a_1,a_2, a_3}F$.
Then $v \in \Delta_K q$ by \Cref{L:existsPfister}.
Since $a_3$ is a square in $F$, one easily sees that $\pfi{a_3}F$ and hence also $q$ is totally indefinite.
The rest of the desired properties follows directly from the choice of $a_1$ and $a_3$.
\end{proof}

After these preparations, we are now ready to obtain the existential predicate defining a given valuation ring in a function field over a global field.
We will use in the proof that, if $F$ is a function field in one variable over a global field $K$ such that $K$ is relatively algebraically closed in $F$, then $F/K$ is regular.
This follows because the exponent of imperfection of $K$ is at most $1$, see e.g. \cite[Lemma 2.7.5]{Fri08}.
This will allow us to reduce statements about general function fields in one variable over global fields to the regular case.

\begin{thm}\label{T:DefiningValuationsFunctionGlobal}
Let $K$ be a global field and $F/K$ a function field in one variable.
There exists an $\exists$-$\Lar(F)$-formula $\rho$ in $4$ free variables such that, given any valuation $v\in\Val(F/K)$, there exist $a_1, a_2 \in \mg F$ and $a_3\in\mg F\cap\qse F$ such that 
$$\rho(F,a_1,a_2,a_3)=\mc{O}_v\,.$$
\end{thm}
\begin{proof}
After replacing $K$ by its relative algebraic closure in $F$, we may assume without loss of generality that $F/K$ is regular.
The proof proceeds along similar lines as that of \Cref{T:DefiningValuationsFunctionLargeUniform}.

Let $\psi$ be the $\exists$-$\Lar(K)$-formula in one free variable from \Cref{cor:fatSetFromCurve} (chosen depending on the function field $F/K$), and let $\varphi$ be the $\exists$-$\Lar$-formula from \Cref{P:Pfister-qss-uniformdef} for $n = 3$.
Let $\rho$ denote the $\exists$-$\Lar(K)$-formula in the free variables $x$, $a_1$, $a_2$, $a_3$ given by
$$
\exists y, z (\psi(y) \wedge yz(1-xa_1 - a_3(xa_1)^2) = x^2 \wedge \varphi(z, a_1, a_2, a_3, a_3)).
$$
We claim that $\rho$ has the desired property.

By \Cref{cor:fatSetFromCurve},
for any $V \subseteq \Val(K)$ finite, we have 
  $K = \psi(F) \cdot\left(\bigcap_{v\in V}\mc{O}_v\right)$.
  For a totally indefinite $3$-fold Pfister form $q$ over $F$, the set \[ W=\{w\in\Val(K)\mid w \text{ nontrivial and } q_{K_wF}\mbox{ anisotropic}\}\] is finite, by \Cref{LocalglobalQT} $(a)$ (case $(i)$), and hence it follows that $K = \psi(F) \cdot\left(\bigcap_{w\in W}\mc{O}_w\right)$.
By \Cref{LocalglobalQT} $(d)$ (case $(i)$), we obtain, for any $c \in \Split_{\mc{O}(\Delta_Kq)}(q)$, that ${\mc{O}(\Delta_K q) = \psi(F) \qss{q}{c}}$
.

Consider $v \in \mc{V}(F/K)$.
By \Cref{L:from-v-choose-3fold-delta-set}, there exist some $a_1, a_2 \in F^\times$ and $a_3 \in \qse{F}$
such that $q = \pfi{a_1, a_2, a_3}F$ is totally indefinite, $v \in \Delta_K q$, $v(a_1)=1$, $w(a_1)\leq 0$ for all $w\in\Delta_K q\setminus\{v\}$ and $a_3\in\mg{\mc{O}(\Delta_K q)}\cap\qse{\mc{O}(\Delta_K q)}$.
In particular, $a_3 \in \Split_{\mc{O}(\Delta_Kq)}(q)$.
For every $w \in \Delta_K q$, the form $\pfi{a_3}{F_w}$ is anisotropic,
and so by \Cref{P:anisotropicResidue} also $\pfi{\ovl{a_3}^w}{Fw}$ is anisotropic.
Using \Cref{uniformapprox} (with $f = xa_1$, $g=a_1$, and $c=a_3$) we compute that,
for all $w \in \Delta_K q$, we have that $\frac{x^2}{1 - xa_1 - a_3(xa_1)^2} \in \mc{O}_w$ if and only if either $w(a_1) \leq 0$ or $w(x) \geq 0$.
By the choice of $a_1$ we thus conclude that 
\begin{align*}
\mc{O}_v &= \lbrace x \in F \mid x^2(1 - xa_1 - a_3(xa_1)^2)^{-1} \in \mc{O}(\Delta_K q) \rbrace \\
&= \lbrace x \in F \mid \exists y \in \psi(F) : x^2(1 - xa_1 - a_3(xa_1)^2)^{-1} \in y\qss{\llangle a_1, a_2, a_3 ]]_F}{a_3} \rbrace \\
&= \rho(F, a_1, a_2, a_3). \qedhere
\end{align*}
\end{proof}
We thus reobtain in the case where $K$ is a number field a part of the result of \cite[Theorem 6.1]{MillerShlapentokh_v2}.
The advantage of our technique is that we obtain a single formula $\rho$, depending only on the function field $F/K$, which allows us to define any valuation ring of $F/K$ through a good choice of parameters.
We discuss the dependence on the specific function field in more detail in \Cref{rem:E-global-uniform}.

We also reobtain the following undecidability result.
Note that this is already well-established in all cases; see \cite{Mor05, EisHil10padic} in the number field case, or \cite{Shlapentokh_Hilbert10_charp, Eisentraeger_Hibert10_char2} in the case of characteristic greater than or equal to two, respectively.

\begin{cor}\label{C:Hilbert10-global}
Let $K$ be a global field.
Let $F/K$ be a function field in one variable. Then Hilbert's 10th Problem for $F$ with coefficients in $F$ has a negative answer.
\end{cor}
\begin{proof}
This follows by combining \Cref{T:DefiningValuationsFunctionGlobal} and \Cref{P:Hilbert10CriterionPrecise}.
\end{proof}

\begin{thm}\label{E-global}
Let $K$ be a global field and $F/K$ a function field in one variable.
There exists an $\exists$-$\Lar(K)$-formula $\rho$ in $4$ free variables such that, for any $a_1,a_2\in\mg{F}$ and $a_3\in\qse{F}$ for which $\pfi{a_1,a_2,a_3}F$ is totally indefinite, we have 
$$\rho(F,a_1,a_2,a_3)=\mc{O}(\Delta_K\pfi{a_1,a_2,a_3}F)\,.$$
\end{thm}
\begin{proof}
Let $K'$ denote the relative algebraic closure of $K$ in $F$ and recall that $F/K'$ is a regular function field.
By \Cref{cor:fatSetFromCurve}, there exists an $\exists$-$\Lar(K)$-formula $\psi$ in one free variable such that,
for any $V \subseteq \Val(K')$ finite, we have 
  $K' = \psi(F) \cdot\left(\bigcap_{v\in V}\mc{O}_v\right)$.

Let $a_1, a_2 \in F^\times$, $a_3 \in \qse{F}$ be such that $q = \llangle a_1, a_2, a_3]]_F$ is totally indefinite.
Set $R = \mc{O}(\Delta_K\pfi{a_1,a_2,a_3}F)$.
Note that $q$ is non-degenerate and $\dim q=8$.
We conclude by \Cref{cor:Pfisterstronggev} that $\Split_R(q)\neq \emptyset$.

Fix $d \in \Split_{R}(q)$.
By \Cref{LocalglobalQT} $(a)$ (case $(i)$), the set \[W=\{w\in\Val(K')\mid w \text{ nontrivial and } q_{K'_wF}\mbox{ anisotropic}\}\] is finite, whence $K' = \psi(F) \cdot\left(\bigcap_{w\in W}\mc{O}_w\right)$.
We conclude by \Cref{LocalglobalQT} (d) (case (i)) that $R = \psi(F)\cdot \qss{q}d$.
By \Cref{P:BigcapDeltaQ} we obtain that
\begin{eqnarray*}
R & = & \psi(F) \left(\bigcup_{e \in \Split(\pfi{a_1, a_2, a_3}F)\cap F^\times}e^{-1}\qss{\pfi{a_1, a_2, a_3}F}e\right)\\
    & = & \left\{ x \in F \enspace\middle|\enspace \begin{array}{l}
    \exists c\in \psi(F), e\in\Split(\pfi{a_1, a_2, a_3}F) \cap F^\times : \\ ex\in c\qss{\pfi{a_1, a_2, a_3}F}e
\end{array}\right\}.
\end{eqnarray*}
Hence, by combining $\psi$ accordingly with the $\exists$-$\Lar$-formulas from \Cref{EX:Pfi-Split} and \Cref{P:Pfister-qss-uniformdef} defining 
$\Split(\pfi{a_1, a_2, a_3}F)$ and $\qss{\pfi{a_1, a_2, a_3}F}e$, respectively, in terms of the parameters $a_1,a_2,a_3$ and $e$, we obtain
an $\exists$-$\Lar(K)$-formula $\rho$ with the desired property.
\end{proof}
\begin{thm}\label{T:AE-global}
Let $K$ be a global field.
Let $F/K$ be a function field in one variable.
There exists an $\forall\exists$-$\Lar(F)$-formula $\gamma(x, f)$ such that every finitary holomorphy ring of $F/K$ is equal to $\gamma(F, f)$ for some $f \in F$.
\end{thm}
\begin{proof}
We want to invoke \Cref{P:AE-criterion} and need to verify the conditions.
We let $n = 3$ and let $\rho$ be the formula from \Cref{E-global}.
We set
$ \psi(a_1, a_2, a_3) $ to be the formula
$ a_1 \neq 0 \wedge a_2 \neq 0 \wedge 1+4a_3 \neq 0 \wedge \exists d (a_3 = d^2) $.
We show that the conditions $(i)$ and $(ii)$ in \Cref{P:AE-criterion} are satisfied.

For $(i)$, consider $(a_1, a_2, a_3) \in \psi(F^3)$.
The form $q = \pfi{a_1,a_2,a_3}F$ is totally indefinite since $a_3$ is a square, and hence by \Cref{E-global} $\rho(F, a_1, a_2, a_3) = \mc{O}(\Delta_K \pfi{a_1, a_2, a_3}F)$, which is a holomorphy ring of $F/K$, as desired.

For $(ii)$, consider $S \subseteq \mc{V}(F/K)$ finite and $w \in \mc{V}(F/K) \setminus S$.
By \Cref{L:from-v-choose-3fold-delta-set} we find $a_1, a_2 \in \mg F$, $a_3 \in \qse{F} \cap F\pow{2}$ such that $w \in \Delta_K \pfi{a_1, a_2,a_3}F \subseteq \mc{V}(F/K) \setminus S$.
We conclude that $(a_1, a_2, a_3) \in \psi(F^3)$ is as desired.
\end{proof}

\begin{rem}\label{rem:E-global-uniform}
In \Cref{E-global}, the constructed formula $\rho$ depends on the specific global base field $K$ and the function field $F/K$.
Inspection reveals that this dependence is only through the reliance on \Cref{cor:fatSetFromCurve}, which in turn relies on a choice of an elliptic curve $E$ defined over $K$ such that $E(K)$ is infinite and satisfies $E(K') = E(F)$.
As an elliptic curve over $K$ can be presented in Weierstraß form depending on five parameters in $K$, a refined version of \Cref{E-global} could thus be stated as follows: 
\begin{quotation}
\it There exists an $\exists$-$\Lar$-formula $\rho'$ in $9$ free variables 
such that, for any global field $K$ and any function field in one variable $F/K$, there exist $b_1, \ldots, b_5 \in K$ such that, for any $a_1, a_2 \in F^\times$ and $a_3 \in \qse{F}$ for which $\pfi{a_1, a_2, a_3}F$ is totally indefinite, we have
$$\rho'(F,a_1,a_2,a_3, b_1, \ldots, b_5)=\mc{O}(\Delta_K\pfi{a_1,a_2,a_3}F)\,.$$
\end{quotation}
\end{rem}

\Cref{E-large} and \Cref{E-global} provide an existential predicate to define uniformly certain holomorphy rings in function fields in one variable over large fields and over global fields, parametrised by the coefficients of a quadratic form.
An existential predicate of this sort was first obtained in \cite{Poo09}, uniformly defining certain holomorphy rings in global fields.
This had opened the way to first-order definitions of other natural subsets of global fields, in particular a universal-existential definition of $\zz$ in $\qq$ \cite{Poo09}, then a purely universal definition of $\zz$ in $\qq$ \cite{Koe16}; see also \cite{Par13, Eis18, DaansGlobal} for generalisations of Koenigsmann's result to other global fields, or \cite{Dit17} for a discussion of universally definable subsets of global fields which are also existentially definable.

It is possible to use \Cref{E-large} and \Cref{E-global} along the same lines to obtain universal definitions of holomorphy rings, so for example~a universal definition of $\qq[T]$ in $\qq(T)$.
This method has been developed in \autocite[Chapter 7]{DaansThesis}, and is presented in \cite{DD}.

\printbibliography
\end{document}